\documentclass[a4paper,12pt]{article}
\usepackage[utf8]{inputenc}
\usepackage[centertags]{ amsmath}
\usepackage{amssymb,amsthm}
\usepackage{amsfonts}
\usepackage[colorlinks=true,linkcolor=blue,citecolor=blue]{hyperref}
\usepackage{url}
\usepackage{mdframed}
\usepackage{multirow}
\usepackage{mathtools}
\usepackage[margin=2.5cm]{geometry}
\usepackage{verbatim}
\usepackage{xcolor,graphicx}

\usepackage{environ}

\usepackage{listings}

\title{Generalized Continued Logarithms and Related Continued Fractions}
\author{Jonathan M. Borwein\footnote{\url{jonathan.borwein@newcastle.edu.au}. CARMA, University of Newcastle, Callaghan NSW 2308, Australia. Research of J. M. Borwein was supported by CARMA, University of Newcastle.} 
\and Kevin G. Hare\footnote{\url{kghare@uwaterloo.ca}. Department of Pure Mathematics, Univsersity of Waterloo, Waterloo, Ontario, N2L 3G1, Canada. Research of K. G. Hare was supported by NSERC Grant RGPIN-2014-03154} 
\and Jason G. Lynch\footnote{\url{j4lynch@uwaterloo.ca}. Research of J. G. Lynch was supported by CARMA, University of Newcastle}}

\def\Z{\hbox{$\mathbb Z$}}
\def\Q{\hbox{$\mathbb Q$}}
\def\R{\hbox{$\mathbb R$}}
\def\N{\hbox{$\mathbb N$}}

\def\M{\hbox{$\mathcal M$}}

\def\a{\hbox{$\mathbf a$}}
\def\c{\hbox{$\mathbf c$}}
\def\d{\hbox{$\,\text{d}$}}
\def\cl{\text{cl}}

\def\KLi{\hbox{$\mathcal{KL}^{\text{I}}$}}
\def\KLii{\hbox{$\mathcal{KL}^{\text{II}}$}}
\def\KLiii{\hbox{$\mathcal{KL}^{\text{III}}$}}

\newcommand{\e}{\varepsilon}

\newcommand{\mat}[1]{\left( \begin{matrix} #1 \end{matrix} \right)}

\newenvironment{longonly}{}{}
\newenvironment{shortonly}{}{}

\newtheorem{theorem}{Theorem}
\newtheorem{lemma}[theorem]{Lemma}

\newtheorem{corollary}[theorem]{Corollary}
\theoremstyle{definition}
\newtheorem{definition}{Definition}

\newtheorem{remark}{Remark}
\newtheorem{fact}{Fact}
\newtheorem{definitionx}{Definition}[definition]
\theoremstyle{plain}
\newenvironment{proofx}{\begin{proof}}{\end{proof}}
\newtheorem{theoremx}{Theorem}[theorem]
\newtheorem{lemmax}{Lemma}[theorem]
\newtheorem{propositionx}{Proposition}[theorem]
\newtheorem{corollaryx}{Corollary}[theorem]
\newtheorem{remarkx}{Remark}[remark]

\NewEnviron{killcontents}{}

\def\cFrac#1#2{%
\begin{array}{@{}c@{}}\multicolumn{1}{c|}{#1}\\%
\hline\multicolumn{1}{|c}{#2}\end{array}}

\begin{document}

\maketitle

\begin{abstract}
We study continued logarithms as introduced by Bill Gosper and studied by J. Borwein et. al.. After providing an overview of the type I and type II generalizations of binary continued logarithms introduced by Borwein et. al., we focus on a new generalization to an arbitrary integer base $b$. We show that all of our so-called type III continued logarithms converge and all rational numbers have finite type III continued logarithms. As with simple continued fractions, we show that the continued logarithm terms, for almost every real number, follow a specific distribution. We also generalize Khinchine's constant from simple continued fractions to continued logarithms, and show that these logarithmic Khinchine constants have an elementary closed form. Finally, we show that simple continued fractions are the limiting case of our  continued logarithms, and briefly consider how we could generalize past continued logarithms.
\end{abstract}


\section{Introduction}\label{sec:intro}

Continued fractions, especially simple continued fractions, have been well studied throughout history. Continued binary logarithms, however, appear to have first been introduced by Bill Gosper in his appendix on Continued Fraction Arithmetic \cite{gosper}. More recently in \cite{clogs}, J. Borwein et. al. proved some basic results about binary continued logarithms and applied experimental methods to determine the term distribution of binary continued logarithms. They conjectured and indicated a proof that, like in the case of continued fractions, almost every real number has continued logarithm terms that follow a specific distribution. They then introduced two different generalizations of binary continued logarithms to arbitrary bases.

\subsection{The Structure of This Paper}

Section~\ref{sec:intro} introduces some basic definitions and results for continued fractions, briefly describes binary continued logarithms as introduced by Gosper, and provides an overview of results relating to the Khinchine constant for continued fractions. Sections~\ref{sec:type1} and~\ref{sec:type2} then provide an overview of the type I and type II continued logarithms introduced by Borwein et. al.. Further details on these can be found in \cite{clogs}. 

Section~\ref{sec:type3} comprises the main body of the paper. In Section~\ref{subsec:type3defrec} we define type III continued logarithms and extend to them the standard continued fraction recurrences. Section~\ref{subsec:type3convratfin} then proves that type III continued logarithms are guaranteed to converge to the correct value, and that every rational number has a finite type III continued logarithm. These are two desirable properties of continue fractions and binary continued logarithms that a complete generalization should have. In Section~\ref{subsec:type3measuretheory} we describe how measure theory can be used to investigate the distribution of continued logarithm terms. This is then applied in Section~\ref{subsec:type3dist} to determine the distribution, and Section~\ref{subsec:type3khinchine} to determine the logarithmic Khinchine constant. The main proofs of these sections are quite technical, and are separated out into Appendices A and B, respectively. Finally, Section~\ref{subsec:type3clogsandcfracs} derives some relationships between simple continued fractions and the limiting case of type III continued logarithms.

Finally, we close the paper in Section~\ref{sec:generalizing} by briefly introducing one way to generalize past continued logarithms.

\subsection{Continued Fractions}

 The material in this section can be found in many places including \cite{nef}.
 
\begin{definition}\label{def:cfrac}
A continued fraction is an expression of the form
\[
  y_1 = \alpha_0 + \cfrac{\beta_1}
                   {\alpha_1 + \cfrac{\beta_2}
 											  {\alpha_2 + \cfrac{\beta_3}{\ddots}}}
\]
or
\[
 y_2 =  \alpha_0 + \cfrac{\beta_1}
                   {\alpha_1 + \cfrac{\beta_2}
 											  {\alpha_2 + \cfrac{\beta_3}{\cdots + \cfrac{\beta_n}{\alpha_n}}}}.
\]
For the sake of simplicity, we will sometimes denote the above as
\[
y_1 = \alpha_0 + \cFrac{\beta_1}{\alpha_1} + \cFrac{\beta_2}{\alpha_2} + \cdots
\]
or
\[
y_2 = \alpha_0 + \cFrac{\beta_1}{\alpha_1} + \cFrac{\beta_2}{\alpha_2} + \cdots + \cFrac{\beta_n}{\alpha_n},
\]
respectively. The terms $\alpha_0, \alpha_1, \dots$ are called denominator terms and the terms $\beta_1, \beta_2, \dots$ are called numerator terms.
\end{definition}

\begin{definition}\label{def:equiv}
Two continued fractions
\[
y = \alpha_0 + \cFrac{\beta_1}{\alpha_1} + \cFrac{\beta_2}{\alpha_2} + \cdots
\hspace{1cm} \text{and} \hspace{1cm}
y' = \alpha'_0 + \cFrac{\beta'_1}{\alpha'_1} + \cFrac{\beta'_2}{\alpha'_2} + \cdots
\]
are called equivalent if there is a sequence $(d_n)_{n=0}^\infty$ with $d_0 = 1$ such that $\alpha_n' = d_n\alpha_n$ for all $n \ge 0$ and $\beta_n' = d_nd_{n-1}\beta_n$ for all $n \ge 1$. 
\end{definition}

The $c_n$ terms can be thought of as constants that are multiplied by both numerators and denominators of successive terms.

\begin{definition}\label{def:cfracconv}
The $n$th convergent of the continued fraction
\[
y = \alpha_0 + \cFrac{\beta_1}{\alpha_1} + \cFrac{\beta_2}{\alpha_2} + \cdots
\]
is given by
\[
x_n = \alpha_0 + \cFrac{\beta_1}{\alpha_1} + \cFrac{\beta_2}{\alpha_2} + \cdots + \cFrac{\beta_n}{\alpha_n}.
\]
\end{definition}

\begin{definition}\label{def:cfracremainderterm}
The $n$th remainder term of the continued fraction
\[
y = \alpha_0 + \cFrac{\beta_1}{\alpha_1} + \cFrac{\beta_2}{\alpha_2} + \cdots
\]
is given by
\[
r_n = \alpha_n + \cFrac{\beta_{n+1}}{\alpha_{n+1}} + \cFrac{\beta_{n+2}}{\alpha_{n+2}} + \cdots.
\]
\end{definition}

The following results will be useful for generalizing to continued logarithms.

\begin{fact}\label{fact:cfracrecurrence}
Suppose $x = \alpha_0 + \cFrac{\beta_1}{\alpha_1} + \cFrac{\beta_2}{\alpha_2} + \cdots$, where $\alpha_n, \beta_n > 0$ for all $n$. Then the convergents are given by
\[
x_n = \frac{p_n}{q_n}
\]
where
\[
p_{-1} = 1, \hspace{5mm} q_{-1} = 0, \hspace{5mm} p_0 = \alpha_0, \hspace{5mm} q_0 = 1,
\]
\begin{align*}
p_n &= \alpha_n p_{n-1} + \beta_n p_{n-2} && n \ge 1, \\
q_n &= \alpha_n q_{n-1} + \beta_n q_{n-2} && n \ge 1.
\end{align*}
\end{fact}

\begin{fact}\label{fact:convcrit}
Suppose $x = \alpha_0 + \cFrac{\beta_1}{\alpha_1} + \cFrac{\beta_2}{\alpha_2} + \cdots$ where $\alpha_n, \beta_n > 0$ for all $n$. Then the continued fraction for $x$ converges to $x$ if $\sum_{n=1}^\infty \frac{\alpha_n \alpha_{n+1}}{\beta_{n+1}} = \infty$.
\end{fact}

\begin{remark}
Throughout this paper, we will use $\M(A)$ or just $\M A$ to denote the Lebesgue measure of a set $A \subseteq \R$.
\end{remark}

\subsection{Binary Continued Logarithms}

Let $1 \le \alpha \in \R$. Let $y_0 = \alpha$ and recursively define $a_n = \lfloor \log_2 y_n \rfloor$. If $y_n - 2^{a_n} = 0$, then terminate. Otherwise, set 
\[
y_{n+1} = \frac{2^{a_n}}{y_n - 2^{a_n}}
\]
and recurse. This produces the binary (base 2) continued logarithm for $y_0$:
\[
y_0 = 2^{a_0} + \cFrac{2^{a_0}}{2^{a_1}} + \cFrac{2^{a_1}}{2^{a_2}} + \cFrac{2^{a_2}}{2^{a_3}} + \cdots
\]
These binary continued logarithms were introduced explicitly by Gosper in his appendix on Continued Fraction Arithmetic \cite{gosper}. Borwein et. al. studied binary continued logarithms further in \cite{clogs}, extending classical continued fraction recurrences for binary continued logs and investigating the distribution of aperiodic binary continued logarithm terms for quadratic irrationalities -- such as can not occur for simple continued fractions. 

\begin{remark}
Jeffrey Shallit \cite{shallit} proved some limits on the length of a finite binary continued logarithm. Specifically, the binary continued logarithm for a rational number $p/q \ge 1$ has at most $2 \log_2 p + O(1)$ terms. Furthermore, this bound is tight, as can be seen by considering the continued fraction for $2^n-1$. Moreover, the sum of the terms of the continued logarithm of $p/q \ge 1$ is bounded by $(\log_2 p) (2\log_2 p + 2)$.
\end{remark}

\subsection{Khinchine's Constant}\label{sec:khinchine}

In \cite{khinchine}, Khinchine proved that for almost every $\alpha \in (0,1)$, where
\[
\alpha = \cFrac{1}{a_1} + \cFrac{1}{a_2} + \cFrac{1}{a_3} + \cdots,
\]
the denominator terms $a_1, a_2, a_3, \dots$ follow a specific limiting distribution. That is,
let $P_\alpha(k) = \lim_{N \to \infty} \frac{1}{N}|\{n \le N : a_n = k\}|$. This is the limiting ratio of the denominator terms that equal $k$, if this limit exists. Then for almost every $\alpha \in (0,1)$,
\[
P_\alpha(k) = \frac{\log\left(1 + \frac{1}{k(k+2)}\right)}{\log 2}
\]
for every $k \in \N$. It then follows for almost every $\alpha \in (0,1)$ that the limiting geometric average of the denominator terms is given by
\[
\lim_{n \to \infty} \sqrt[n]{a_1a_2\cdots a_n} = \prod_{k=1}^\infty \left(1+\frac{1}{r(r+2)}\right)^{\log_2 r} \approx 2.685452.
\]
This constant is now known as Khinchine's constant, $\mathcal{K}$.

\section{Type I Continued Logarithms}\label{sec:type1}

\subsection{Type I Definition and Preliminaries}

Fix an integer base $b \ge 2$. We define type I continued logarithms as follows.
\begin{definition}\label{def:type1clog}
Let $\alpha \in (1,\infty)$. The base $b$ continued logarithm of type I for $\alpha$ is
\[
b^{a_0} + \cFrac{(b-1)b^{a_0}}{b^{a_1}} + \cFrac{(b-1)b^{a_1}}{b^{a_2}} + \cFrac{(b-1)b^{a_2}}{b^{a_3}} + \cdots = [b^{a_0}, b^{a_1}, b^{a_2}, \dots]_{\text{cl}_1(b)},
\]
where the terms $a_0, a_1, a_2 \dots$ are determined by the recursive process below, terminating at the term $b^{a_n}$ if at any point $y_n = b^{a_n}$.
\begin{align*}
y_0 &= \alpha \\
a_n &= \lfloor \log_b y_n \rfloor && n \ge 0 \\
y_{n+1} &= \frac{(b-1)b^{a_n}}{y_n - b^{a_n}} && n \ge 0.
\end{align*}
\end{definition}

The numerator terms $(b-1)b^{a_n}$ are defined as such to ensure that $y_n \in (1,\infty)$ for all n. Indeed, notice that for each $n$, we must have $b^{a_n} \le y_n < b^{a_n+1}$. Thus $0 \le y_n - b^{a_n} < (b-1)b^{a_n}$. If $y_n-b^{a_n} = 0$, then we terminate, otherwise we get $0 < y_n - b^{a_n} < (b-1)b^{a_n}$, so $y_{n+1} = \frac{(b-1) b^{a_n}}{y_n - b^{a_n}} \in (1,\infty)$.

Borwein et al. proved that the type I continued fraction of $\alpha \in (1,\infty)$ will converge to $\alpha$ \cite[Theorem 15]{clogs}. Additionally, numbers with finite type I continued logarithms must be rational. However for $b \ge 3$, rationals need not have finite continued logarithms. For example, the type I ternary continued logarithm for 2 is $[3^0,3^0,3^0,\dots]_{\text{cl}_1(3)}$.

\begin{longonly}
The following lemmas will be useful for studying the limiting distribution of the continued logarithm terms $a_n$.
\end{longonly}

\begin{lemmax}\label{lem:type1equiv}
The continued logarithms
\[
y = b^{a_0} + \cFrac{(b-1)b^{a_0}}{b^{a_1}} + \cFrac{(b-1)b^{a_1}}{b^{a_2}} + \cFrac{(b-1)b^{a_2}}{b^{a_3}} + \cdots
\]
and
\[
y_1 = b^{a_0} + \cFrac{(b-1)b^{a_0-a_1}}{1} + \cFrac{(b-1)b^{-a_2}}{1} + \cFrac{(b-a)b^{-a_3}}{1} + \cdots
\]
are equivalent. ($y_1$ is called the denominator-reduced continued logarithm for $y$.)
\end{lemmax}
\begin{proofx}
Take $d_0 = 1$ and $d_n = b^{-a_n}$ for $n \ge 1$ to satisfy the conditions of Definition~\ref{def:equiv}.
\end{proofx}

\begin{lemmax}\label{lem:type1pqlemma}
The $n$th convergent of $\alpha = [b^{a_0}, b^{a_1}, b^{a_2}, \dots]_{\text{cl}_1(b)}$ is given by
\[
x_n = \frac{p_n}{q_n}
\]
where 
\[
p_{-1} = 1, \hspace{1cm} q_{-1} = 0, \hspace{1cm} p_0 = b^{a_0}, \hspace{1cm} q_0 = 1
\]
and for $n \ge 1$,
\begin{align*}
p_n &= b^{a_n} p_{n-1} + (b-1)b^{a_{n-1}} p_{n-2} \\
q_n &= b^{a_n} q_{n-1} + (b-1)b^{a_{n-1}} q_{n-2}
\end{align*}
\end{lemmax}
\begin{proofx}
This follows from Fact~\ref{fact:cfracrecurrence}, where we have $\alpha_n =  b^{a_n}$ and $\beta_n = (b-1) b^{a_{n-1}}$.
\end{proofx}

\subsection{Distribution of Type I Continued Logarithm Terms and Type I Logarithmic Khinchine Constant}

We now look at the limiting distribution of the type I continued logarithm terms. Consider $\alpha = [b^{a_0}, b^{a_1}, b^{a_2}, \dots]_{\text{cl}_1(b)}$. Assume that the continued logarithm for $\alpha$ is infinite. Furthermore, assume (without loss of generality) that $a_0 = 0$, so that $\alpha \in (1,b)$.

\begin{definition}\label{def:type1Dnk}
For $n \in \N$, let
\[
D_n(k) = \{\alpha \in (1,b) : a_n = k\}
\]
denote the set of $\alpha \in (1,b)$ for which the $n$th continued logarithm term is $b^k$.
\end{definition}

\begin{definition}\label{def:type1zn,An,mn}
Let $x = [1, b^{a_1}, b^{a_2}, \dots]_{\text{cl}_1(b)} \in (1,b)$. The $n$th remainder term of $x$ is $r_n = r_n(x) = [b^{a_n}, b^{a_{n+1}}, \dots]_{\cl_1(b)}$, as in Definition~\ref{def:cfracremainderterm}. Define 
\[
z_n = z_n(x) = \frac{r_n}{b^{a_n}} = [1, b^{a_{n+1}}, b^{a_{n+2}}, \dots]_{\cl_1(b)} \in (1,b),
\]
\[
M_n(x) = \{\alpha \in (1,b) : z_n(\alpha) < x\} \subseteq (1,b),
\]
\[
m_n(x) = \frac{1}{b-1} \M
(M_n(x)) \in (0,1),
\]
and
\[
m(x) = \lim_{n\to\infty} m_n(x),
\]
wherever this limit exists.
\end{definition}

Notice that since $1 < z_n(\alpha) < b$ for all $n \in \N$ and $\alpha \in (1,b)$, we must have $m_n(1) = 0$ and $m_n(b) = 1$ for all $n \in \N$. We can now derive a recursion  for the functions $m_n$. 

\begin{theorem}\label{thm:type1mnrecurrence}
The sequence of functions $m_n$ is given by the recursive relationship
\begin{align}
m_0(x) &= \frac{x-1}{b-1} \label{eq:type1m0} \\
m_n(x) &= \sum_{k=0}^\infty m_{n-1}(1+(b-1)b^{-k}) - m_{n-1}(1+x^{-1}(b-1)b^{-k}) && n \ge 1 \label{eq:type1mn}
\end{align}
for $1 \le x \le b$.
\end{theorem}

\begin{shortonly}
The proof of this is similar to that of Theorem~\ref{thm:mnrecurrence}.
\end{shortonly}

\begin{proofx}
Notice that $r_0(\alpha) = \alpha$ and $a_0 = 0$, so $z_0(\alpha) = \frac{r_0}{b^{a_0}} = \alpha$ and thus 
\[
M_0(x) = \{\alpha \in (1,b) : z_0(\alpha) < x\} = \{\alpha \in (1,b) : \alpha < x\} = (1,x),
\] 
so $m_0(x) = \frac{x-1}{b-1}$. Now fix $n \ge 1$. Since $a_n \in \Z_{\ge 0}$, we have
\[
m_n(x) = \M\{\alpha \in (1,b) : z_n < x\} = \M \bigcup_{k=0}^\infty \{\alpha \in (1,b) : z_n < x, a_n = k\}.
\]
Fix $x \in (1,b)$ and let
\[
A_{k} = \{\alpha \in (1,b) : z_n < x, a_n = k\}
\]
for $k \in \Z_{\ge 0}$. By Definition~\ref{def:type1zn,An,mn}, $z_n < x$ if and only if 
\begin{equation*}
\frac{r_n}{b^{a_n}} < x
\end{equation*}
Notice that
\[
z_{n-1} = [1,c_nb^{a_n}, c_{n+1}b^{a_{n+1}}, \dots] = [1,r_n] = 1+\frac{b-1}{r_n}
\]
so $z_n < x$ if and only if
\[
\frac{b-1}{b^{a_n}(z_{n-1}-1)} < x,
\]
or equivalently
\begin{equation}\label{eq:type1mnrecineq1}
z_{n-1} > 1 + (b-1) b^{-a_n} x^{-1} = 1+(b-1)b^{-k}x^{-1}.
\end{equation}
Additionally, in order to have $a_n = k$, we must have $b^k \le r_n < b^{k+1}$, or equivalently,
\begin{equation}\label{eq:type1mnrecineq2}
1 + (b-1)b^{-(k+1)} < z_{n-1} \le 1+(b+1)b^{-k}.
\end{equation}
Now notice that since $x < b$,
\[
1+(b-1)b^{-(k+1)} < 1+(b-1)b^{-k}x^{-1},
\]
and thus the left hand inequality in \eqref{eq:type1mnrecineq2} is implied by \eqref{eq:type1mnrecineq1}. Therefore $z_n < x$ with $a_n = k$ if and only if
\begin{equation}\label{eq:type1mnrecineq3}
1+(b-1)b^{-k}x^{-1} < z_{n-1} \le 1+(b-1)b^{-k}.
\end{equation}
Thus
\begin{equation}\label{eq:type1mnrecAk}
A_{k} = \{\alpha \in (1,b) : 1+(b-1)b^{-k}x^{-1} < z_{n-1} \le 1+(b-1)b^{-k}\}.
\end{equation}
Now suppose $k_1,k_2 \in \Z$ with $k_1 \ne k_2$. We claim that $A_{k_1}$ and $A_{k_2}$ are disjoint. Suppose without loss of generality that $k_1 > k_2$, so $k_2 - k_1 \le -1$. Then since $x < b$,
\begin{equation}\label{eq:type1mnreccase1}
1+(b-1)b^{-k_1} = 1+(b-1)b^{-k_2}b^{k_2-k_1} \le 1+(b-1)b^{-k_2}b^{-1} < 1+(b-1)b^{-k_2}x^{-1}.
\end{equation}

Now suppose $a_1 \in A_{k_1}$ and $a_2 \in A_{k_2}$. By \eqref{eq:type1mnrecAk} and \eqref{eq:type1mnreccase1}, 
\[
a_1 \le 1+(b-1)b^{-k_1} < 1 + (b-1)b^{-k_2}x^{-1} \le a_2
\]
so $a_1 \ne a_2$ and thus $A_{k_1}$ and $A_{k_2}$ must be disjoint. Therefore
\begin{equation} \label{eq:type1mnrecmun}
m_n(x) = \M \bigcup_{k=0}^\infty A_{k} = \sum_{k=0}^\infty \mathcal{M}(A_{k}).
\end{equation}
Finally, since $m_{n-1}(x) = \mathcal{M}\{\alpha \in (1,b) : z_{n-1} < x\}$, by \eqref{eq:type1mnrecAk} and \eqref{eq:type1mnrecmun} we can conclude
\[
m_n(x) = \sum_{k=0}^\infty m_{n-1}(1+(b-1)b^{-k}) - m_{n-1}(1+x^{-1}(b-1)b^{-k})
\]
which proves the recursion \eqref{eq:type1mn}, and completes the proof of the theorem.
\end{proofx}

We  next derive a formula for $D_n(k)$ in terms of the function $m_n$.
\begin{theorem}\label{thm:type1genprobdist}
\[
\frac{1}{b-1} \M (D_{n+1}(k) )= m_n(1+(b-1)b^{-k}) - m_n(1+(b-1)b^{-(k+1)}).
\]
\end{theorem}

\begin{shortonly}
The proof of this theorem is similar to that of Theorem~\ref{thm:genprobdist}.
\end{shortonly}

\begin{proofx}
Suppose that $\alpha \in D_{n+1}(k)$. Then $a_{n+1} = k$, so
\[
z_n = [1, b^k, r_{n+2}]_{\cl_1(b)} = 1 + \cfrac{b-1}{b^k + \cfrac{(b-1)b^k}{r_{n+2}}}
\]
where $r_{n+2}$ can take any value in $(1,\infty)$. Clearly $z_n$ is a monotonic function of $r_{n+2}$ for fixed $k$, so the extreme values of $z_n$ on $D_{n+1}(k)$ will occur at the extreme values of $r_{n+2}$. Letting $r_n \to 1$ gives $z_n = 1 + \frac{b-1}{b^k + (b-1)b^k} = 1+(b-1)b^{-(k+1)}$, and letting $r_n \to \infty$ gives $z_n = 1 + \frac{b-1}{b^k + 0} = 1+(b-1)b^{-k}$.
Thus,
\begin{align*}
D_{n+1}(k) &= \{\alpha \in (1,b) : 1+ (b-1)b^{-(k+1)} < z_n(\alpha) \le 1+(b-1)b^{-k} \} \\
&= M_n(1+(b-1)b^{-k}) \setminus M_n(1+(b-1)b^{-(k+1)}),
\end{align*}
so
\[
\frac{1}{b-1} \M D_{n+1}(k) = m_n(1+(b-1)b^{-k}) - m_n(1+(b-1)b^{-(k+1)}).
\]
\end{proofx}

Thus, if the limiting distribution $m(x)$ exists, it immediately follows that
\begin{equation}\label{eq:type1limitingdist}
\lim_{n\to\infty} \frac{1}{b-1} \M (D_n(k) ) = m(1+(b-1)b^{-k}) - m(1+(b-1)b^{-(k+1)}).
\end{equation}

\subsection{Experimentally Determining the Type I Distribution}\label{sec:type1experimental}

Now suppose $b>1$ is an arbitrary integer. Let $\mu_b$ denote the limiting distribution function $m$ for the base $b$, assuming it exists.

We may investigate the form of $\mu_b(x)$ by iterating the recurrence relation of Theorem \ref{thm:type1mnrecurrence} at points evenly spaced over the interval $[1,b]$, starting with $m_0(x) = \frac{x-1}{b-1}$. At each iteration, we fit a spline to these points, evaluating each ``infinite'' sum to 100 terms, and breaking the interval $[1,b]$ into 100 pieces. This is practicable since the continued logarithm converges much more rapidly than the simple continued fraction.

\begin{longonly}
This process is given by the Maple code in figure \ref{fig:type1distcode}, and the resulting functions for $b=2,3,4,5$ are shown in figure \ref{fig:type1distplot}.
\end{longonly}

\begin{longonly}
\begin{figure}[ht]
\begin{mdframed}[backgroundcolor=blue!20]
\footnotesize
\begin{lstlisting}
pd_type1 := proc (b, n) 
  local u, i, j, v, res, y; 
  u := z -> (z-1)/(b-1); 
  for i from 1 to n do 
    v := seq([x, evalf(add(u(1+(b-1)*b^(-k))-u(1+(b-1)*b^(-k)/x), 
         k = 0 .. 100))], x = seq(1+i*(b-1)/100, i = 0 .. 100)); 
    u := z -> Spline([v], z);
  end do;
  return [seq([x, u(x)], x = seq(1+i*(b-1)/100, i = 0 .. 100))]; 
end proc;
\end{lstlisting}
\end{mdframed}
\caption{Maple code for generating the type I $\mu_b$ function}
\label{fig:type1distcode}
\end{figure}

\begin{figure}[ht]
	\centering
		\includegraphics[width=0.75\textwidth]{type1dist.jpg}
	\caption{Type I $\mu_b$ for $2 \le b \le 5$ (10 iterations)}
	\label{fig:type1distplot}
\end{figure}
\end{longonly}

We find good convergence of $\mu_b(x)$ after around 10 iterations. We use the 101 data points from this process to seek the best fit to a function of the form
\[
\mu_b(x) = C \log_b \frac{\alpha x + \beta}{\gamma x + \delta}.
\]
We set $\gamma = 1$ to eliminate any common factor between the numerator and denominator. To meet the boundary condition $\mu_b(1) = 0$, we must have $\delta = \alpha + \beta - 1$, and to meet the boundary condition $\mu_b(b) = 1$, we must have $C = \frac{1}{\log_b\frac{b\alpha + \beta}{\alpha + \beta + b-1}}$, leaving the functional form to be fit as
\begin{equation} \label{eq:type1model}
\mu_b(x) = \frac{\log_b\frac{\alpha x + \beta}{x + \alpha + \beta - 1}}{\log_b\frac{b\alpha + \beta}{\alpha + \beta + b - 1}}.
\end{equation}

We sought this superposition form when the simpler structure for simple continued fractions failed.
\begin{longonly}
The code to perform this fit is shown in figure \ref{fig:type1fitcode} and it's output for $b = 2$ to $b = 10$ is shown in in figure \ref{fig:type1fitoutput}.
\end{longonly}

\begin{longonly}
\begin{figure}[ht]
\begin{mdframed}[backgroundcolor=blue!20]
\footnotesize
\begin{lstlisting}
ff_type1 := proc(b)
  local data;
  data := pd_type1(b, n);
  return NonlinearFit(Re(log((alpha*x+beta)/(x+alpha+beta-1))
         /log((b*alpha+beta)/(alpha+beta+b-1))), data, x, 
         initialvalues = [alpha = 0.9, beta = 1], 
         output = parametervalues);
end proc;
for b from 2 to 10 do 
  print(b, ff(b)); 
end do;
\end{lstlisting}
\end{mdframed}
\caption{Maple code for fitting the type I $\mu_b$ function to \eqref{eq:type1model}}
\label{fig:type1fitcode}
\end{figure}

\begin{figure}[ht]
\begin{mdframed}[backgroundcolor=blue!20]
\footnotesize
\begin{lstlisting}
 2, [alpha = 0.50006779707172260, beta = 0.4999760185017539]
 3, [alpha = 0.33334918347940207, beta = 0.6666057296207324]
 4, [alpha = 0.25115557175955190, beta = 0.7506021021686686]
 5, [alpha = 0.19870458936050558, beta = 0.7987970827140846]
 6, [alpha = 0.16713077201225227, beta = 0.8343332515643020]
 7, [alpha = 0.14404250067601154, beta = 0.8593816473715626]
 8, [alpha = 0.12566550377196060, beta = 0.8760629681404826]
 9, [alpha = 0.11126055120886144, beta = 0.8888148824032103]
10, [alpha = 0.09965479172707140, beta = 0.8987267892000272]
\end{lstlisting}
\end{mdframed}
\caption{Maple output from the code in figure \ref{fig:type1fitcode} for $2 \le b \le 10$}
\label{fig:type1fitoutput}
\end{figure}
\end{longonly}

Fitting our data to the model suggests candidate values of $\alpha = \frac{1}{b}$ and $\beta = \frac{b-1}{b}$, from which we get
\begin{equation}\label{eq:type1m}
\mu_b(x) = \frac{\log \frac{bx}{x+b-1}}{\log \frac{b^2}{2b-1}}.
\end{equation}

When we then apply \eqref{eq:type1limitingdist}, we get
\[
\lim_{n \to \infty} \frac{1}{b-1} \M D_n(k) = \frac{\log\left(1+ \frac{b^k(b-1)^3}{(b^{k+1}+b-1)^2}\right)}{\log\frac{b^2}{2b-1}}.
\]

\begin{longonly}
This limiting distribution is shown in figure \ref{fig:type1probplot} for $b = 2, 3, 4, 5$.

\begin{figure}[ht]
	\centering
		\includegraphics[width=0.75\textwidth]{type1prob.jpg}
	\caption{Type I limiting distribution for $2 \le b \le 5$}
	\label{fig:type1probplot}
\end{figure}
\end{longonly}

A proof of this distribution and of the type I Khinchine constant for each integer base $b$, using ergodic theory, can be found in \cite{lascu}. Additionally, it is likely that the proofs in Appendices A and B for the type III continued logarithm distribution and logarithmic Khinchine constant could be appropriately adjusted to prove these results.

If a type I base $b$ Khinchine constant $\KLi_b$ exists (i.e., almost every $\alpha \in (1,\infty)$ has the same limiting geometric mean of denominator terms), and if a limiting distribution $D(k) = \lim_{n \to \infty} D_n(k)$ of denominator terms exists, then
\[
\KLi_b = \prod_{k=0}^\infty b^{k} \frac{\M D(k)}{b-1} = b^{\sum_{k=0}^\infty k \frac{\mathcal{M} D(k)}{b-1}}.
\]
This is because the limiting distribution of denominator terms (if it exists) is essentially the ``average'' distribution over all numbers $\alpha \in (1,b)$. If we then assume that almost every $\alpha \in (1,b)$ has the same limiting geometric mean of denominator terms, then this limiting geometric mean (the logarithmic type I Khinchine constant) must equal the limiting geometric mean of the ``average'' distribution.

Thus, if we assume $\KLi_b$ exists and that the distribution in \eqref{eq:type1limitingdist} is correct, then we must have $\KLi_b = b^{\mathcal{A}}$, where
\[
\mathcal{A} = \sum_{k=0}^\infty k \frac{\M D(k)}{b-1} = \sum_{k=0}^\infty k [\mu_b(1+(b-1)b^{-k}) - \mu_b(1+(b-1)b^{-(k+1)})] = \frac{\log b}{\log \frac{b^2}{2b-1}} - 1,
\]
by Theorem~\ref{thm:type1genprobdist} and a lengthy but straightforward algebraic manipulation. These conjectured type I logarithmic Khinchine constants for $2 \le b \le 10$ are given in Figure~\ref{fig:type1klresults}.

\begin{figure}[ht]
\begin{center}
\begin{tabular}{c|c} 
 $b$ & $\KLi_b$ \\ \hline
2 & 2.656305058 \\
3 & 2.598065150 \\
4 & 2.556003239 \\
5 & 2.524285360 \\
6 & 2.499311827 \\
7 & 2.478977440 \\
8 & 2.461986788 \\
9 & 2.447498976 \\
10 & 2.434942582 \\
\end{tabular}
\end{center}
\caption{Type I logarithmic Khinchine constants for $2 \le b \le 10$}
\label{fig:type1klresults}
\end{figure}

These conjectured values of the type I logarithmic Khinchine constants were supported by empirical evidence, as the numerically computed limiting geometric means of denominator terms for various irrational constants give the expected values.

Notice that the type I logarithmic Khinchine constants have a simple closed form, which is noteworthy as no simple closed form has been found for the Khinchine constant for simple continued fractions.

\section{Type II Continued Logarithms}\label{sec:type2}

\subsection{Type II Definition and Preliminaries}

Fix an integer base $b \ge 2$. We define type II continued logarithms as follows.
\begin{definition}\label{def:type2clog}
Let $\alpha \in \R_{\ge 1}$. The base $b$ continued logarithm for $\alpha$ is
\[
c_0 b^{a_0} + \cFrac{c_0 b^{a_0}}{c_1 b^{a_1}} + \cFrac{c_1 b^{a_1}}{c_2 b^{a_2}} + \cFrac{c_2 b^{a_2}}{c_3b^{a_3}} + \cdots = [c_0 b^{a_0}, c_1 b^{a_1}, c_2 b^{a_2}, \dots]_{\cl_2(b)},
\]
where the terms $a_0, a_1, a_2 \dots$ and $c_0, c_1, c_2, \dots$ are determined by the recursive process below, terminating at the term $c_n b^{a_n}$ if at any point $y_n = c_n b^{a_n}$.
\begin{align*}
y_0 &= \alpha \\
a_n &= \lfloor \log_b y_n \rfloor && n \ge 0 \\
c_n &= \left\lfloor \frac{y_n}{b^{a_n}} \right\rfloor && n \ge 0 \\
y_{n+1} &= \frac{c_n b^{a_n}}{y_n - c_n b^{a_n}} && n \ge 0.
\end{align*}
\end{definition}

\begin{remark}\label{rem:type2numeratorterms}
The numerator terms $c_n b^{a_n}$ are defined to match the corresponding denominator terms. Recall that in the type I case, the term $y_{n+1}$ could take any value in $(1,\infty)$, regardless of the value of $a_n$. This is no longer true, since $y_n - c_n b^{a_n} \in (0, b^{a_n})$, so $y_{n+1} \in (c_n, \infty)$. We will see later that this results in type II continued logarithms having a more complicated distribution for which we could not find a closed form. This issue was the inspiration for the definition of type III continued logarithms, where the numerator terms are $b^{a_n}$ instead of $c_n b^{a_n}$.
\end{remark}

Borwein et. al. proved that the type II continued fraction of $\alpha \in (1,\infty)$ will converge to $\alpha$, and that $\alpha \in (1,\infty)$ has a finite continued logarithm if and only if $\alpha \in \Q$ \cite[Theorems 19 and 20]{clogs} -- unlike the situation for type I.

\begin{longonly}
The following lemmas will be useful for studying the limiting distribution of the continued logarithm terms $c_n b^{a_n}$.
\end{longonly}

\begin{lemmax}\label{lem:type2equiv}
The continued logarithms
\[
y = c_0 b^{a_0} + \cFrac{c_0 b^{a_0}}{c_1 b^{a_1}} + \cFrac{c_1 b^{a_1}}{c_2 b^{a_2}} + \cFrac{c_2 b^{a_2}}{c_3 b^{a_3}} + \cdots
\]
and
\[
y_1 = c_0 b^{a_0} + \cFrac{c_0c_1^{-1}b^{a_0-a_1}}{1} + \cFrac{c_2^{-1}b^{-a_2}}{1} + \cFrac{c_3^{-1} b^{-a_3}}{1} + \cdots
\]
are equivalent. ($y_1$ is called the denominator-reduced continued logarithm for $y$.)
\end{lemmax}
\begin{proofx}
Take $d_0 = 1$ and $d_n = c_n^{-1} b^{-a_n}$ for $n \ge 1$ to satisfy the conditions of Definition~\ref{def:equiv}.
\end{proofx}

\begin{lemmax}\label{lem:type2pqlemma}
The $n$th convergent of $\alpha = [c_0 b^{a_0}, c_1 b^{a_1}, c_2 b^{a_2}, \dots]_{\text{cl}_2(b)}$ is given by
\[
x_n = \frac{p_n}{q_n}
\]
where 
\[
p_{-1} = 1, \hspace{1cm} q_{-1} = 0, \hspace{1cm} p_0 = c_0 b^{a_0}, \hspace{1cm} q_0 = 1
\]
and for $n \ge 1$,
\begin{align*}
p_n &= c_n b^{a_n} p_{n-1} + c_{n-1} b^{a_{n-1}} p_{n-2} \\
q_n &= c_n b^{a_n} q_{n-1} + c_{n-1} b^{a_{n-1}} q_{n-2}
\end{align*}
\end{lemmax}
\begin{proofx}
This follows from Fact~\ref{fact:cfracrecurrence}, where for we have $\alpha_n =  c_n b^{a_n}$ and $\beta_n = c_{n-1} b^{a_{n-1}}$.
\end{proofx}

\subsection{Distribution of Type II Continued Logarithm Terms and Type II Logarithmic Khinchine Constant}

We now look at the limiting distribution of the type II continued logarithm terms. Consider $\alpha = [c_0 b^{a_0}, c_1 b^{a_1}, c_2 b^{a_2}, \dots]_{\text{cl}_2(b)}$. Assume that $\alpha \notin \Q$, so that the continued logarithm for $\alpha$ is infinite. Furthermore, assume (without loss of generality) $a_0 = 0$ and $c_0 = 1$, so that $\alpha \in (1,2)$.

\begin{definition}\label{def:type2Dnk}
Let $n \in \N$. Let
\[
D_n(k,\ell) = \{\alpha \in (1,2) : a_n = k, c_n = \ell\}
\]
denote the $\alpha \in (1,2)$ for which the $n$th continued logarithm term is $\ell b^k$.
\end{definition}

\begin{definition}\label{def:type2zn,Mn,mn}
Let $x = [1, c_1 b^{a_1}, c_2 b^{a_2}, \dots]_{\text{cl}_2(b)} \in (1,2)$ with $n$th remainder term $r_n = r_n(x) = [c_n b^{a_n}, c_{n+1} b^{a_{n+1}}, \dots]_{\cl_2(b)}$, as in Definition~\ref{def:cfracremainderterm}. Define 
\[
z_n = z_n(x) = \frac{r_n}{c_n b^{a_n}} = [1, c_{n+1} b^{a_{n+1}}, c_{n+2} b^{a_{n+2}}, \dots]_{\cl_2(b)} \in (1,2),
\]
\[
M_n(x) = \{\alpha \in (1,2) : z_n(\alpha) < x\} \subseteq (1,2),
\]
\[
m_n(x) = \M(M_n(x) ) \in (0,1),
\]
and
\[
m(x) = \lim_{n\to\infty} m_n(x),
\]
wherever this limit exists.
\end{definition}

Notice that since $1 \le z_n(\alpha) \le 2$ for all $n \in \N$ and $\alpha \in (1,2)$, we must have $m_n(1) = 0$ and $m_n(2) = 1$ for all $n \in \N$. 

We may now derive a recursion relation for the functions $m_n$. 

\begin{theorem}\label{thm:type2mnrecurrence}
The sequence of functions $m_n$ is given by the recursive relationship
\begin{align}
m_0(x) &= x-1 \label{eq:type2m0} \\
m_n(x) &= \sum_{k=0}^\infty \sum_{\ell = 1}^{b-1} m_{n-1}(1+\ell^{-1}b^{-k}) - m_{n-1}(\max\{1+\ell^{-1}b^{-k}x^{-1}, 1+(\ell+1)^{-1}b^{-k}\}) && n \ge 1 \label{eq:type2mn}
\end{align}
for $1 \le x \le 2$.
\end{theorem}
\begin{proofx}
Notice that $r_0(\alpha) = \alpha$, $a_0 = 0$, and $c_0 = 1$, so $z_0(\alpha) = \frac{r_0}{c_0 b^{a_0}} = \alpha$ and thus $M_0(x) = \{\alpha \in (1,2) : z_0(\alpha) < x\} = \{\alpha \in (1,2) : \alpha < x\} = (1,x)$, so $m_0(x) = x-1$.

Now fix $n \ge 1$. Since $a_n \in \Z_{\ge 0}$ and $c_n \in \{1,\dots,b-1\}$, we have
\[
m_n(x) = \M\{\alpha \in (1,2) : z_n < x\} = \M \bigcup_{k=0}^\infty \bigcup_{\ell=1}^{b-1} \{\alpha \in (1,2) : z_n < x, a_n = k, c_n = \ell\}.
\]
Fix $x \in (1,2)$ and let
\[
A_{k,\ell} = \{\alpha \in (1,2) : z_n < x, a_n = k, c_n = \ell\}
\]
for $k \in \Z_{\ge 0}$ and $\ell \in \{1,\dots,b-1\}$. By Definition~\ref{def:type2zn,Mn,mn}, $z_n < x$ if and only if 
\begin{equation*}
\frac{r_n}{c_n b^{a_n}} < x
\end{equation*}
Notice that
\[
z_{n-1} = [1,c_nb^{a_n}, c_{n+1}b^{a_{n+1}}, \dots] = [1,r_n] = 1+\frac{1}{r_n}
\]
so $z_n < x$ if and only if
\[
\frac{1}{c_n b^{a_n}(z_{n-1}-1)} < x,
\]
or equivalently
\begin{equation}\label{eq:type2mnrecineq1}
z_{n-1} > 1 + c_n^{-1} b^{-a_n} x^{-1} = 1+\ell^{-1}b^{-k}x^{-1}.
\end{equation}
Additionally, in order to have $a_n = k$ and $c_n = \ell$, we must have $\ell b^k \le r_n < (\ell+1)b^k$, or equivalently,
\begin{equation}\label{eq:type2mnrecineq2}
1 + (\ell+1)^{-1}b^{-k} < z_{n-1} \le 1+\ell^{-1}b^{-k}.
\end{equation}
Unlike in the type I case, the left hand inequality in \eqref{eq:type2mnrecineq2} is not implied by \eqref{eq:type2mnrecineq1}, so $z_n < x$ with $a_n = k$ and $c_n = \ell$ if and only if
\begin{equation}\label{eq:type2mnrecineq3}
\max\{1+\ell^{-1}b^{-k}x^{-1}, 1+(\ell+1)^{-1}b^{-k}\} < z_{n-1} \le 1+\ell^{-1}b^{-k}.
\end{equation}
Thus
\begin{equation}\label{eq:type2mnrecAkl}
A_{k,\ell} = \{\alpha \in (1,2) : \max\{1+\ell^{-1}b^{-k}x^{-1}, 1+(\ell+1)^{-1}b^{-k}\} < z_{n-1} \le 1+\ell^{-1}b^{-k}\}.
\end{equation}
Now suppose $k_1,k_2 \in \Z$ and $\ell_1,\ell_2 \in \{1,\dots,b-1\}$ with $(k_1,\ell_1) \ne (k_2,\ell_2)$. We claim that $A_{k_1,\ell_1}$ and $A_{k_2,\ell_2}$ are disjoint. Consider two cases:

Case 1: $k_1 \ne k_2$. Suppose (without loss of generality) that $k_2 < k_1$, so $k_2 - k_1 \le -1$. Also note that $1 \le p,q \le b-1$, so we have
\begin{align}
1+\ell_1^{-1}b^{-k_1} &= 1+\ell_1^{-1} b^{k_2 - k_1} b^{-k_2} \le 1+ \ell_1^{-1}b^{-1}b^{-k_2} \le 1+b^{-1}b^{-k_2} \notag \\
&< 1+(\ell_2+1)^{-1} b^{-k_2} \le 1 +\max\{1+\ell_2^{-1}b^{-k_2}x^{-1}, 1+(\ell_2+1)^{-1}b^{-k_2}\} \label{eq:type2mnreccase1}
\end{align}

Case 2: $k_1 = k_2$, $\ell_1 \ne \ell_2$. Suppose (without loss of generality) that $\ell_1 > \ell_2$, so indeed $\ell_1 \ge \ell_2+1$. Then
\begin{equation}\label{eq:type2mnreccase2}
1+\ell_1^{-1}b^{-k_1} = 1+\ell_1^{-1}b^{-k_2} \le 1+(\ell_2+1)^{-1}b^{-k_2} \le 1 + \max\{1+\ell_2^{-1}b^{-k_2}x^{-1}, 1+(\ell_2+1)^{-1}b^{-k_2}\}.
\end{equation}

Now suppose $a_1 \in A_{k_1,\ell_1}$ and $a_2 \in A_{k_2, \ell_2}$. By \eqref{eq:type2mnrecAkl} and either \eqref{eq:type2mnreccase1} or \eqref{eq:type2mnreccase2}, 
\[
a_1 \le 1+\ell_1^{-1}b^{-k_1} < 1 + (\ell_2+x-1)^{-1}b^{-k_2} \le a_2
\]
so $a_1 \ne a_2$ and thus $A_{k_1,\ell_1}$ and $A_{k_2,\ell_2}$ must be disjoint. Therefore
\begin{equation} \label{eq:type2mnrecmun}
m_n(x) = \M \bigcup_{k=0}^\infty \bigcup_{\ell=1}^{b-1} A_{k,\ell} = \sum_{k=0}^\infty \sum_{\ell=1}^{b-1} \mathcal{M}(A_{k,\ell}).
\end{equation}
Finally, since $m_{n-1}(x) = \mathcal{M}\{\alpha \in (1,2) : z_{n-1} < x\}$, by \eqref{eq:type2mnrecAkl} and \eqref{eq:type2mnrecmun} we can conclude
\[
m_n(x) = \sum_{k=0}^\infty \sum_{\ell=1}^{b-1} \left( m_{n-1}(1+\ell^{-1}b^{-k}) - m_{n-1}\left(\max\{1+\ell^{-1}b^{-k}x^{-1}, 1+ (\ell+1)^{-1}b^{-k}\}\right)\right),
\]
which proves the recursion \eqref{eq:type2mn}, and completes the proof of the theorem.
\end{proofx}

We can now derive a formula for $D_n(k,\ell)$ in terms of the function $m_n$.
\begin{theorem}\label{thm:type2genprobdist}
\[
\M (D_{n+1}(k,\ell)) = m_n(1+\ell^{-1}b^{-k}) - m_n(1+(\ell+1)^{-1}b^{-k}).
\]
\end{theorem}

\begin{proofx}
Suppose that $\alpha \in D_{n+1}(k,\ell)$. Then $a_{n+1} = k$ and $c_{n+1} = \ell$, so
\[
z_n = [1, \ell b^k, r_{n+2}]_{\cl_2(b)} = 1 + \cfrac{1}{\ell b^k + \cfrac{\ell b^k}{r_{n+2}}}
\]
where $r_{n+2}$ can take any value in $(\ell,\infty)$, as noted in Remark~\ref{rem:type2numeratorterms}. Clearly $z_n$ is a monotone function of $r_{n+2}$ for fixed $k$, so the extreme values of $z_n$ on $D_{n+1}(k,\ell)$ will occur at the extreme values of $r_{n+2}$. Letting $r_n \to \ell$ gives $z_n = 1 + \frac{1}{\ell b^k + b^k} = 1+(\ell+1)^{-1}b^{-k}$, and letting $r_n \to \infty$ gives $z_n = 1 + \frac{1}{\ell b^k + 0} = 1+\ell^{-1} b^{-k}$.
Thus,
\begin{align*}
D_{n+1}(k) &= \{\alpha \in (1,b) : 1+(\ell+1)^{-1}b^{-k} < z_n(\alpha) \le 1+\ell^{-1}b^{-k} \} \\
&= M_n(1+\ell^{-1}b^{-k}) \setminus M_n(1+(\ell+1)^{-1}b^{-k}),
\end{align*}
so
\[
\M D_{n+1}(k) = m_n(1+\ell^{-1}b^{-k}) - m_n(1+(\ell+1)^{-1}b^{-k}).
\]
\end{proofx}

Thus, if the limiting distribution $m(x)$ exists, it immediately follows that
\begin{equation}\label{eq:type2limitingdist}
\lim_{n\to\infty} \M (D_n(k,\ell)) = m(1+\ell^{-1}b^{-k}) - m(1+(\ell+1)^{-1}b^{-k}).
\end{equation}

\subsection{Experimentally Determining the Type II Distribution}\label{sec:type2experimental}

Again, suppose $b$ is arbitrary. Let $\mu_b$ denote the limiting distribution function $m$ for the base $b$, assuming it exists.

We may again investigate the form of $\mu_b(x)$ by iterating the recurrence relation of Theorem \ref{thm:type2mnrecurrence} at points evenly spaced over the interval $[1,2]$, starting with $m_0(x) = x-1$. At each iteration, we fit a spline to these points, evaluating each ``infinite'' sum to 100 terms, and breaking the interval $[1,2]$ into 100 pieces. 
\begin{longonly}
This process is given by the Maple code in figure \ref{fig:type2distcode}, and the resulting functions for $b=2,3,4,5$ are shown in figure \ref{fig:type2distplot}.

\begin{figure}[ht]
\begin{mdframed}[backgroundcolor=blue!20]
\footnotesize
\begin{lstlisting}
pd_type2 := proc (b, n) 
  local u, i, j, v, res, y; 
  u := z -> z-1; 
  for i from 1 to n do 
    v := seq([x, evalf(add(add(u(1+*b^(-k)/p)
         -u(1+max(b^(-k)/(p*x), b^(-k)/(p+1))), 
         p = 1 .. b-1), k = 0 .. 100))], 
         x = seq(1+i/100, i = 0 .. 100)); 
    u := z -> Spline([v], z);
  end do;
  return [seq([x, u(x)], x = seq(1+i/100, i = 0 .. 100))]; 
end proc;
\end{lstlisting}
\end{mdframed}
\caption{Maple code for generating the type II $\mu_b$ function}
\label{fig:type2distcode}
\end{figure}

\begin{figure}[ht]
	\centering
		\includegraphics[width=0.75\textwidth]{type2dist.jpg}
	\caption{Type II $\mu_b$ for $2 \le b \le 5$ (10 iterations)}
	\label{fig:type2distplot}
\end{figure}
\end{longonly}

We find good convergence of $\mu_b(x)$ after around 10 iterations. However, we have been unable to find a closed form for $\mu_b$ for $b > 2$. It appears that $\mu_b$ is a continuous non-monotonic function that is smooth on $(1,2)$ except at $x = \frac{j+1}{j}$ for $j = 2, \dots, b-1$.

\begin{longonly}
We can still use \eqref{eq:type2limitingdist} to approximate the limiting distribution. This approximate limiting distribution is shown in figure \ref{fig:type2probplot} for $b = 2, 3, 4, 5$.

\begin{figure}[ht]
	\centering
		\includegraphics[width=0.75\textwidth]{type1prob.jpg}
	\caption{Type II approximate limiting distribution for $2 \le b \le 5$}
	\label{fig:type2probplot}
\end{figure}
\end{longonly}

If a logarithmic Khinchine constant $\KLii_b$ exists (i.e. almost every $\alpha \in (1,\infty)$ has the same limiting geometric mean of denominator terms), and if a limiting distribution $D(k,\ell) = \lim_{n \to \infty} D_n(k,\ell)$ of denominator terms exists, then
\[
\KLii_b = \prod_{k=0}^\infty \prod_{\ell = 1}^{b-1} \ell b^{k} \M D(k,\ell).
\]
This is because the limiting distribution of denominator terms (if it exists) is essentially the ``average'' distribution over all numbers $\alpha \in (1,2)$. If we then assume that almost every $\alpha \in (1,2)$ has the same limiting geometric mean of denominator terms, then this limiting geometric mean (the logarithmic Khinchine constant) must equal the limiting geometric mean of the ``average'' distribution.

However, since we do not know the limiting distribution, we can only approximate the logarithmic Khinchine constants. 
\begin{longonly}
The Maple code for these approximations is shown in figure~\ref{fig:type2klcode}, and the approximated values are shown in figure~\ref{fig:type2klresults}.
\end{longonly}

\begin{longonly}
\begin{figure}[ht]
\begin{mdframed}[backgroundcolor=blue!20]
\footnotesize
\begin{lstlisting}
mu[2] := log(2*x/(x+1))/log(4/3);
for b from 2 to 10 do
  mu[b] := Spline(pd_type2(b, 10), x);
end do;
kl2 := b -> evalf(mul(mul((p * b^k)^(eval(mu[b], x=1+1/(p*b^k)) -
           eval(mu[b], x=1+1/((p+1)*b^k))),p=1..b-1),k=0..100));
for b from 2 to 10 do 
  print(b, evalf(kl2(b), 11));
end do;
\end{lstlisting}
\end{mdframed}
\caption{Maple code for approximating Type II Logarithmic Khintchine constant}
\label{fig:type2klcode}
\end{figure}
\end{longonly}

\begin{figure}[ht]
\begin{center}
\begin{tabular}{c|c}
 $b$ & $\KLii_b$ \\ \hline
 2 & 2.656305048 \\
 3 & 3.415974174 \\
 4 & 4.064209949 \\
 5 & 4.636437895 \\
 6 & 5.152343739 \\
 7 & 5.624290253 \\
 8 & 6.060673548 \\
 9 & 6.467518102 \\
10 & 6.849326402 \\
\end{tabular}
\end{center}
\caption{Experimental type II logarithmic Khinchine constants for $2 \le b \le 10$}
\label{fig:type2klresults}
\end{figure}

This conjectured values of the type II logarithmic Khinchine constants are supported by empirical evidence, as the limiting geometric means of denominator terms for various irrational constants give the conjectured values.

\section{Type III Continued Logarithms}\label{sec:type3}

Fix an integer base $b \ge 2$. In this section, we will introduce our third generalization of base 2 continued logarithms. This appears to be the best of the three generalizations, as we will show that type III continued logarithms have guaranteed convergence, rational finiteness, and closed forms for the limiting distribution and logarithmic Khinchine constant. Additionally, type III continued logarithms `converge' to simple continued fractions if one looks at limiting behaviour as $b \to \infty$.

\subsection{Type III Definitions and Recurrences}\label{subsec:type3defrec}

We start with some definitions, notation, and lemmas related to continued logarithm recurrences.

\begin{definition}\label{def:clog}
Let $\alpha \in \R_{\ge 1}$. The type III base $b$ continued logarithm for $\alpha$ is
\[
c_0b^{a_0} + \cFrac{b^{a_0}}{c_1b^{a_1}} + \cFrac{b^{a_1}}{c_2b^{a_2}} + \cFrac{b^{a_2}}{c_3 b^{a_3}} + \cdots = [c_0b^{a_0}, c_1b^{a_1}, c_2b^{a_2}, \dots]_{\cl_3(b)}.
\]
where the terms $a_0,a_1,a_2,\dots$ and $c_0,c_1,c_2,\dots$ are determined by the recursive process below, terminating at the term $c_n b^{a_n}$ if at any point $y_n = c_n b^{a_n}$.
\begin{align*}
y_0 &= \alpha \\
a_n &= \lfloor \log_b y_n \rfloor && n \ge 0 \\
c_n &= \left\lfloor \frac{y_n}{b^{a_n}} \right\rfloor && n \ge 0 \\
y_{n+1} &= \frac{b^{a_n}}{y_n - c_nb^{a_n}} && n \ge 0
\end{align*}
\end{definition}

\begin{remark}
We can (and often will) think of the $a_n$ and $c_n$ as functions $a_0, a_1, a_2, \dots: (1, \infty) \to \Z_{\ge 0}$ and $c_0, c_1, c_2, \dots : (1,\infty) \to \{1,2,\dots,b-1\}$, since the terms $a_0, c_0, a_1, c_1, a_2, c_2, \dots$ are uniquely determined by $\alpha$. Conversely, given the complete sequences $a_0, a_1, a_2, \dots$ and $c_0, c_1, c_2, \dots$, one can recover the value of $\alpha$.
\end{remark}

\begin{remark}\label{rem:convergent,remainder}
Let $\alpha = [c_0b^{a_0}, c_1b^{a_1}, c_2b^{a_2}, \dots]_{\cl_3(b)} \in (1,\infty)$. Based on Definitions~\ref{def:cfracconv} and~\ref{def:cfracremainderterm}, the $n$th convergent and $n$th remainder term of $\alpha$ are given by
\[
x_n(\alpha) = c_0b^{a_0} + \cFrac{b^{a_0}}{c_1b^{a_1}} + \cFrac{b^{a_1}}{c_2b^{a_2}} + \cFrac{b^{a_2}}{c_3 b^{a_3}} + \cdots + \cFrac{b^{a_{n-1}}}{c_nb^{a_n}}.
\]
and
\[
r_n(\alpha) = c_nb^{a_n} + \cFrac{b^{a_n}}{c_{n+1}b^{a_{n+1}}} + \cFrac{b^{a_{n+1}}}{c_{n+2} b^{a_{n+2}}} + \cdots,
\]
respectively.

Note that the terms $r_n$ are the same as the terms $y_n$ from Definition~\ref{def:clog}.
\end{remark}

\begin{lemma}\label{lem:pqlemma}
The $n$th convergent of $\alpha = [c_0b^{a_0}, c_1b^{a_1}, c_2b^{a_2}, \dots]_{\cl_3(b)}$ is given by
\[
x_n = \frac{p_n}{q_n}
\]
where 
\[
p_{-1} = 1, \hspace{1cm} q_{-1} = 0, \hspace{1cm} p_0 = c_0b^{a_0}, \hspace{1cm} q_0 = 1,
\]
and for $n \ge 1$,
\begin{align*}
p_n &= c_nb^{a_n} p_{n-1} + b^{a_{n-1}} p_{n-2}, \\
q_n &= c_nb^{a_n} q_{n-1} + b^{a_{n-1}} q_{n-2}.
\end{align*}
\end{lemma}
\begin{proof}
This follows from Fact~\ref{fact:cfracrecurrence}, where for continued logarithms we have $\alpha_n = c_n b^{a_n}$ and $\beta_n = b^{a_{n-1}}$.
\end{proof}

\begin{lemma}\label{lem:qnbounds}
We have the following lower bounds on the denominators $q_n$:
\begin{itemize}
\item $q_n \ge 2^{(n-1)/2} > \frac12 2^{n/2}$ for $n \ge 0$,
\item $q_n \ge b^{a_1 + \cdots + a_n}$ for $n \ge 0$.
\end{itemize}
\end{lemma}
\begin{proof}
For the first bound, note that 
\[
q_n = c_nb^{a_n} q_{n-1} + b^{a_{n-1}} q_{n-2} \ge (c_nb^{a_n} + b^{a_{n-1}}) q_{n-2} \ge 2 q_{n-2}.
\]
A simple inductive argument then gives $q_n \ge 2^{n/2} q_0 = 2^{n/2} > 2^{(n-1)/2}$ for even $n$ and $q_n \ge 2^{(n-1)/2} q_1 \ge 2^{(n-1)/2}$ for odd $n$.

For the second bound, note that $q_n = c_nb^{a_n} q_{n-1} + b^{a_{n-1}} q_{n-2} \ge b^{a_n} q_{n-1}$ from which another simple inductive argument gives $q_n \ge b^{a_n + a_{n-1} + \cdots + a_1} q_0 = b^{a_1 + \cdots + a_n}$.
\end{proof}

\begin{lemma}\label{lem:pnqnm1-qnpnm1}
For $n \ge 0$,
\[
p_n q_{n-1} - q_n p_{n-1} = (-1)^{n-1} b^{a_0 + \cdots + a_{n-1}}.
\]
\end{lemma}
\begin{proof}
For $n = 0$, we have
\[
p_0 q_{-1} - q_0 p_{-1} = c_0b^{a_0}(0) - 1(1) = -1 = (-1)^{-1} b^0.
\]
Now suppose that the statement is true for some $n \ge 0$. Then by Lemma~\ref{lem:pqlemma},
\begin{align*}
p_{n+1}q_n - q_{n+1}p_n &= (c_{n+1}b^{a_{n+1}}p_n + b^{a_n}p_{n-1})q_n - (c_{n+1}b^{a_{n+1}}q_n + b^{a_n}q_{n-1})p_n \\
&= -b^{a_n} (p_n q_{n-1} - q_n p_{n-1}) = -b^{a_n} (-1)^{n-1} b^{a_0 + \cdots + a_{n-1}} \\
&= (-1)^n b^{a_0 + \cdots + a_n}.
\end{align*}
so the result follows by induction.
\end{proof}

The following lemma is equivalent to Lemma~\ref{lem:pqlemma}, and will be used to prove Theorem~\ref{thm:remaindertheorem}.

\begin{lemma}\label{lem:pqmatrix}
Let $a_{-1} = 0$. Then for all $n \ge 0$,
\[
\mat{p_n & p_{n-1} \\ q_n & q_{n-1}} = \prod_{j=0}^n \mat{c_jb^{a_j} & 1 \\ b^{a_{j-1}} & 0}.
\]
\end{lemma}
\begin{proof}
For $n=0$, we have
\[
\prod_{j=0}^0 \mat{c_jb^{a_j} & 1 \\ b^{a_{j-1}} & 0} = \mat{c_0b^{a_0} & 1 \\ b^{a_{-1}} & 0} = \mat{c_0b^{a_0} & 1 \\ 1 & 0} = \mat{p_0 & p_{-1} \\ q_0 & q_{-1}}.
\]
Now suppose for induction that
\[
\prod_{j=0}^{n-1} \mat{c_jb^{a_j} & 1 \\ b^{a_{j-1}} & 0} = \mat{p_{n-1} & p_{n-2} \\ q_{n-1} & q_{n-2}}.
\]
Then by Lemma~\ref{lem:pqlemma},
\begin{align*}
\prod_{j=0}^n \mat{c_jb^{a_j} & 1 \\ b^{a_{j-1}} & 0} &= \mat{p_{n-1} & p_{n-2} \\ q_{n-1} & q_{n-2}} \mat{c_nb^{a_n} & 1 \\ b^{a_{n-1}} & 0} = \mat{c_nb^{a_n} p_{n-1} + b^{a_{n-1}}p_{n-2} & p_{n-1} \\ c_nb^{a_n}q_{n-1} + b^{a_{n-1}} q_{n-2} & q_{n-1}} = \mat{p_n & p_{n-1} \\ q_n & q_{n-1}},
\end{align*}
as asserted.
\end{proof}

\begin{theorem}\label{thm:remaindertheorem}
For arbitrary $1 \le k \le n$,
\[
[c_0b^{a_0}, c_1b^{a_1}, \dots, c_nb^{a_n}]_{\cl_3(b)} = \frac{p_{k-1}r_k + p_{k-2}b^{a_{k-1}}}{q_{k-1}r_k + q_{k-2}b^{a_{k-1}}}.
\]
\end{theorem}
\begin{proof}
First notice that $r_k = [c_kb^{a_k}, \dots, c_nb^{a_n}]_{\cl_3(b)} = \frac{p_k'}{q_k'}$, where
\[
\mat{p_k' \\ q_k'} = \mat{c_kb^{a_k} & 1 \\ 1 & 0} \prod_{j=k+1}^n \mat{c_jb^{a_j} & 1 \\ b^{a_{j-1}} & 0} \mat{1\\0}.
\]
Also note that
\[
\mat{c_kb^{a_k} & 1 \\ b^{a_{k-1}} & 0} = \mat{1 & 0 \\ 0 & b^{a_{k-1}}} \mat{c_kb^{a_k} & 1 \\ 1 & 0}.
\]
Then
\begin{align*}
\mat{p_n \\ q_n} &= \prod_{j=0}^n \mat{c_jb^{a_j} & 1 \\ b^{a_{j-1}} & 0} \mat{1 \\ 0} = \prod_{j=0}^{k-1} \mat{c_jb^{a_j} & 1 \\ b^{a_{j-1}} & 0} \mat{c_kb^{a_k} & 1 \\ b^{a_{k-1}} & 0} \prod_{j=k+1}^n \mat{c_jb^{a_j} & 1 \\ b^{a_{j-1}} & 0} \mat{1\\0} \\
&= \mat{p_{k-1} & p_{k-2} \\ q_{k-1} & q_{k-2}} \mat{1 & 0 \\ 0 & b^{a_{k-1}}} \mat{c_kb^{a_k} & 1 \\ 1 & 0} \prod_{j=k+1}^n \mat{c_jb^{a_j} & 1 \\ b^{a_{j-1}} & 0} \mat{1\\0} \\
&= \mat{p_{k-1} & p_{k-2} b^{a_{k-1}}  \\ q_{k-1} & q_{k-2} b^{a_{k-1}}} \mat{p_k' \\ q_k'} = \mat{p_{k-1}p_k' + p_{k-2} b^{a_{k-1}} q_k' \\ q_{k-1}p_k' + q_{k-2} b^{a_{k-1}} q_k'}.
\end{align*}
Thus
\begin{align*}
[c_0b^{a_0}, \dots, c_nb^{a_n}]_{\cl_3(b)} &= \frac{p_n}{q_n} = \frac{p_{k-1}p_k' + p_{k-2} b^{a_{k-1}} q_k'}{q_{k-1}p_k' + q_{k-2} b^{a_{k-1}} q_k'} = \frac{p_{k-1} \frac{p_k'}{q_k'} + p_{k-2} b^{a_{k-1}}}{q_{k-1} \frac{p_k'}{q_k'} + q_{k-2} b^{a_{k-1}}} 
= \frac{p_{k-1} r_k + p_{k-2} b^{a_{k-1}}}{q_{k-1} r_k + q_{k-2} b^{a_{k-1}}},
\end{align*}
as required.
\end{proof}

\subsection{Convergence and Rational Finiteness of Type III Continued Logarithms}\label{subsec:type3convratfin}

\begin{theorem}
The type III continued logarithm for a number $x \ge 1$ converges to $x$.
\end{theorem}
\begin{proof}
Suppose that the continued logarithm for $x = [c_0 b^{a_0}, c_1 b^{a_1}, \dots, c_n b^{a_n}]_{\cl_3(b)}$ is finite. From the construction, we have $x = y_0$ where
\[
y_k = c_k b^{a_k} + \frac{b^{a_k}}{y_{k+1}}
\]
for $0 \le k \le n-1$. From Definition~\ref{def:clog}, since the continued logarithm terminates, we have $y_n = c_n b^{a_n}$, at which point we simply have
\[
x = c_0 b^{a_0} + \cFrac{b^{a_0}}{c_1 b^{a_1}} + \cFrac{b^{a_1}}{c_2 b^{a_2}} + \cfrac{b^{a_2}}{c_3 b^{a_3}} + \cdots + \cfrac{b^{a_{n-1}}}{c_n b^{a_n}}.
\]
This shows convergence in the case of finite termination. If the continued logarithm for $x$ does not terminate, then convergence follows from Fact~\ref{fact:convcrit}, since
\[
\sum_{n=1}^\infty \frac{\alpha_n \alpha_{n+1}}{\beta_{n+1}} = \sum_{n=1}^\infty \frac{c_n b^{a_n} c_{n+1} b^{a_{n+1}}}{b^{a_n}} = \sum_{n=1}^\infty c_n c_{n+1} b^{a_{n+1}} = \infty,
\]
while all terms are positive as required.
\end{proof}

\begin{lemma}\label{lem:equivforratfinite}
If
\begin{align*}
y &= c_0 b^{a_0} + \cFrac{b^{a_0}}{c_1 b^{a_1}} + \cFrac{b^{a_1}}{c_2 b^{a_2}} + \cFrac{b^{a_2}}{c_3 b^{a_3}} + \cdots, \\
y_1 &= c_0 b^{a_0} + \cFrac{c_1^{-1}b^{a_0-a_1}}{1} + \cFrac{c_1^{-1}c_2^{-1}b^{-a_2}}{1} + \cFrac{c_2^{-1}c_3^{-1}b^{-a_3}}{1} + \cdots,
\end{align*}
then $y$ and $y_1$ are equivalent. (The form $y_1$ is called the denominator-reduced continued logarithm for $y$.)
\end{lemma}
\begin{proof}
Take $d_0 = 1$ and $d_n = c_n^{-1} b^{-a_n}$ for $n \ge 1$ to satisfy the conditions of Definition~\ref{def:equiv}.
\end{proof}

\begin{theorem}
The type III continued logarithm for a number $x \ge 1$ will terminate finitely if and only if $x \in \Q$.
\end{theorem}
\begin{proof}
Clearly, if the continued logarithm for $x$ terminates finitely, then $x \in \Q$. Conversely, suppose
\[
x = c_0 b^{a_0} + \cFrac{b^{a_0}}{c_1b^{a_1}} + \cFrac{b^{a_1}}{c_2 b^{a_2}} + \cdots
\]
is rational. By Lemma~\ref{lem:equivforratfinite}, we can write
\[
x = c_0 b^{a_0} \left(1 + \cFrac{c_0^{-1}c_1^{-1}b^{-a_1}}{1} + \cFrac{c_1^{-1} c_2^{-1} b^{-a_2}}{1} + \cdots \right).
\]
Let $y_n$ denote the $n$th tail of the continued logarithm, that is,
\[
y_n = 1 + \cFrac{c_n^{-1}c_{n+1}^{-1}b^{-a_{n+1}}}{1} + \cFrac{c_{n+1}^{-1} c_{n+2}^{-1} b^{-a_{n+2}}}{1} + \cdots.
\]
Notice that
\[
y_n = 1 + \frac{c_n^{-1} c_{n+1}^{-1} b^{-a_{n+1}}}{y_{n+1}},
\]
so
\[
y_{n+1} = \frac{c_n^{-1} c_{n+1}^{-1} b^{-a_{n+1}}}{y_n - 1}.
\]
Since each $y_n$ is rational, write $y_n = \frac{u_n}{v_n}$ for positive relatively prime integers $u_n$ and $v_n$.  Hence
\[
\frac{u_{n+1}}{v_{n+1}} = y_{n+1} = \frac{c_n^{-1} c_{n+1}^{-1} b^{-a_{n+1}}}{\frac{u_n - v_n}{v_n}} = \frac{v_n}{c_n c_{n+1} b^{a_{n+1}} (u_n - v_n)},
\]
or equivalently,
\[
c_n c_{n+1} b^{a_{n+1}} (u_n - v_n) u_{n+1} = v_n v_{n+1}.
\]
Notice that since $y_n \ge 1$ for all $n$, $u_n - v_n \ge 0$, so each multiplicative term in the above equation is a nonnegative integer. Since $u_{n+1}$ and $v_{n+1}$ are relatively prime, we must have $u_{n+1} \mid v_n$, so $u_{n+1} \le v_n \le u_n$. If at any point we have $u_{n+1} = v_n = u_n$, then $y_n = \frac{u_n}{v_n} = 1$ and the continued logarithm terminates. Otherwise, $u_{n+1} < u_n$, so $(u_n)$ is a strictly decreasing sequence of nonnegative integers, so the process must terminate, again giving a finite continued logarithm.
\end{proof}

\subsection{Using Measure Theory to Study the Type III Continued Logarithm Terms}\label{subsec:type3measuretheory}

We now look at the relative frequency of the continued logarithm terms. Specifically, the main theorem of this section places bounds on the measure of the set 
\[
\left\{x \in (1,2): \begin{array}{lllll} a_1 = k_1, & a_2 = k_2, & \dots, & a_n = k_n, & a_{n+1} = k \\ c_1 = \ell_1, & c_2 = \ell_2, & \dots,&  c_n = \ell_n, & c_{n+1} = \ell \end{array} \right\}
\]
in terms of the measure of the set 
\[
\left\{x \in (1,2): \begin{array}{llll} a_1 = k_1, & a_2 = k_2, & \dots, & a_n = k_n \\ c_1 = \ell_1, & c_2 = \ell_2, & \dots,&  c_n = \ell_n \end{array} \right\}
\]
and the value of $k$ and $\ell$. From that, we can get preliminary bounds on the measure of $\{x \in (1,2) : a_n = k, c_n = \ell\}$ in terms of $k$ and $\ell$.

Consider $\alpha = [c_0b^{a_0}, c_1b^{a_1}, c_2b^{a_2}, \dots]_{\cl_3(b)}$. Assume that $\alpha \notin \Q$, so that the continued logarithm for $\alpha$ is infinite. Furthermore, assume $a_0 = 1$ and $c_0 = 1$, so that $\alpha \in (1,2)$. Notice that in order to have $a_1 = k_1$ and $c_1 = \ell_1$, we must have $1 + (\ell_1 + 1)b^{-k_1} < \alpha \le 1+\ell_1 b^{-k_1}$. Thus we can partition $(1,2)$ into countably many intervals $J_1\mat{0\\1}, J_1\mat{0\\2}, \dots, J_1\mat{0\\b-1}, J_1\mat{1\\1}, J_1\mat{1\\2}, \dots$ such that $a_1 = k_1$ and $c_1 = \ell_1$ for all $\alpha \in J_1\mat{k_1\\\ell_1}$. This gives, in general,
\[
J_1 \mat{k_1\\\ell_1} = \left(1+\frac{1}{(\ell_1 + 1) b^{k_1}}, 1 + \frac{1}{\ell_1b^{k_1}} \right].
\]
We call these intervals the intervals of first rank.

Now fix some interval of first rank, $J_1\mat{k_1 \\ \ell_1}$, and consider the values of $a_2$ and $c_2$ for $\alpha \in J_1\mat{k_1 \\ \ell_1}$. One can show that we have $a_1 = k_1$, $c_1 = \ell_1$, $a_2 = k_2$, and $c_2 = \ell_2$ on the interval
\[
J_2\mat{k_1, & k_2 \\ \ell_1, & \ell_2} = \left[1 + \frac{1}{\ell_1 b^{k_1} + \frac{b^{k_1}}{\ell_2 b^{k_2}}}, 1 + \frac{1}{\ell_1 b^{k_1} + \frac{b^{k_1}}{(\ell_2+1) b^{k_2}}}\right).
\]
These are the intervals of second rank. We may repeat this process indefinitely to get the intervals of $n$th rank, noting that each interval of rank $n$ is just a subinterval of an interval of rank $n-1$.

\begin{definition}\label{def:Jnk1...kn}
Let $n \in \N$. The intervals of $n$th rank are the intervals of the form
\[
J_n\mat{k_1, & k_2, & \dots, & k_n \\ \ell_1, & \ell_2, & \dots, & \ell_n} = \left\{\alpha \in (1,2) : \begin{array}{cccc} a_1 = k_1, & a_2 = k_2, & \dots, & a_n = k_n \\ c_1 = \ell_1, & c_2 = \ell_2, & \dots, & c_n = \ell_n \end{array} \right\},
\]
where $k_1, k_2, \dots, k_n \in \Z_{\ge 0}$ and $\ell_1, \ell_2, \dots, \ell_n \in \{1,2,\dots,b-1\}$.
\end{definition}

\begin{remark}
The intervals of $n$th rank will be half-open intervals that are open on the left if $n$ is odd and open on the right if $n$ is even. However, for simplicity, we will ignore what happens at the endpoints and treat these intervals as open intervals. This will not affect the main theorems of this paper, as the set of endpoints is a set of measure zero.
\end{remark}

\begin{definition}
Suppose $m,n \in \N$ with $m \ge n$. Let $a_{n+1},\dots,a_m \in \Z_{\ge 0}$ and $c_{n+1},\dots,c_m \in \{1,\dots,b-1\}$. Let $f$ be a function that maps intervals of rank $m$ to real numbers. Then we define
\[
\sum^{(n)} f\left(J_m\mat{a_1, & a_2, & \dots, & a_m \\ c_1, & c_2, & \dots, & c_m}\right) = \sum_{a_1 = 1}^\infty \sum_{c_1 = 1}^{b-1} \cdots \sum_{a_n = 1}^\infty \sum_{c_n = 1}^{b-1} f\left(J_m\mat{a_1, & a_2, & \dots, & a_m \\ c_1, & c_2, & \dots, & c_m}\right),
\]
and similarly we define
\[
\bigcup^{(n)} J_m\mat{a_1, & a_2, & \dots, & a_m \\ c_1, & c_2, & \dots, & c_m} = \bigcup_{a_1 = 1}^\infty \bigcup_{c_1 = 1}^{b-1} \cdots \bigcup_{a_n = 1}^\infty \bigcup_{c_n = 1}^{b-1} J_m\mat{a_1, & a_2, & \dots, & a_m \\ c_1, & c_2, & \dots, & c_m}.
\]
\end{definition}

\begin{definition}\label{def:Dnk}
Let $n \in \N$. Let
\[
D_n(k,\ell) = \{\alpha \in (1,2) : a_n = k, c_n = \ell\}
\]
denote the set of points where the $n$th continued logarithm term is $\ell b^k$.
\end{definition}

\begin{remark}
$D_n(k,\ell)$ is a countable union of intervals of rank $n$, specifically,
\[
D_n (k,\ell) = \bigcup^{(n-1)} J_n \mat{a_1, & a_2, & \dots, & a_{n-1}, & k \\ c_1, & c_2, & \dots, & c_{n-1}, & \ell}.
\]
\end{remark}

\begin{lemma}\label{lem:endpoints}
Let $J_n\mat{a_1, & a_2, & \dots, & a_n \\ c_1, & c_2, & \dots, & c_n}$ be an interval of rank $n$. The endpoints of $J_n$ are
\[
\frac{p_n}{q_n} \hspace{1cm} \text{and} \hspace{1cm} \frac{p_n + p_{n-1}b^{a_n}}{q_n + q_{n-1}b^{a_n}}.
\]
\end{lemma}
\begin{proof}
Let $\alpha \in J_n\mat{a_1, & a_2, & \dots, & a_n \\ c_1, & c_2, & \dots, & c_n}$ be arbitrary. Note that $\alpha = [1, c_1b^{a_1}, \dots, c_nb^{a_n}, r_{n+1}]_{\cl_3(b)}$, where $r_{n+1}$ can take any real value in $[1,\infty)$. From Theorem~\ref{thm:remaindertheorem}, we have
\[
\alpha = \frac{p_n r_{n+1} + p_{n-1} b^{a_n}}{q_n r_{n+1} + q_{n-1} b^{a_n}}.
\]
Notice that
\begin{align*}
\alpha - \frac{p_n}{q_n} = \frac{p_n r_{n+1} + p_{n-1} b^{a_n}}{q_n r_{n+1} + q_{n-1} b^{a_n}} = \frac{(q_n p_{n-1} - p_n q_{n-1})b^{a_n}}{q_n(q_n r_{n+1} + q_{n-1} b^{a_n})},
\end{align*}
and on $J_n\mat{a_1, & a_2, & \dots, & a_n \\ c_1, & c_2, & \dots, & c_n}$, all of $p_n, q_n, p_{n-1}, q_{n-1}, a_n$ are fixed. Thus $\alpha$ is a monotonic function of $r_{n+1}$, so the extreme values of $\alpha$ on $J_n\mat{a_1, & a_2, & \dots, & a_n \\ c_1, & c_2, & \dots, & c_n}$ will occur at the extreme values of $r_{n+1}$. Taking $r_{n+1} = 1$ gives $\alpha = \frac{p_n + p_{n-1} b^{a_n}}{q_n + q_{n-1} b^{a_n}}$, and letting $r_{n+1} \to \infty$ gives $\alpha = \frac{p_n}{q_n}$. Thus the endpoints of $J_n\mat{a_1, & a_2, & \dots, & a_n \\ c_1, & c_2, & \dots, & c_n}$ are
\[
\frac{p_n}{q_n} \hspace{1cm} \text{and} \hspace{1cm} \frac{p_n + p_{n-1}b^{a_n}}{q_n + q_{n-1}b^{a_n}},
\]
as claimed.
\end{proof}

\begin{theorem}\label{thm:mjn1bounds}
Suppose $n \in \N$, $a_1, a_2, \dots, a_n, k \in \Z_{\ge 0}$, and $c_1, c_2, \dots, c_n, \ell \in \{1,\dots,b-1\}$. Let $\a = (a_1, \dots, a_n)$ and $\c = (c_1,\dots,c_n)$. Then
\[
\frac{1}{4 \ell(\ell+1) b^k} \M J_n\mat{\a \\ \c} \le \M J_{n+1}\mat{\a, & k \\ \c, & \ell} \le \frac{2}{\ell(\ell+1) b^k} \M J_n\mat{\a \\ \c}.
\]
\end{theorem}
\begin{proof}
From Lemma~\ref{lem:endpoints}, we know that the endpoints of $J_n\mat{\a \\ \c}$ are
\[
\frac{p_n}{q_n} \hspace{1cm} \text{and} \hspace{1cm} \frac{p_n + p_{n-1}b^{a_n}}{q_n + q_{n-1}b^{a_n}},
\]
Now in order to be in $J_{n+1}\mat{\a, & k \\ \c, & \ell}$, we must have $a_{n+1} = k$ and $c_{n+1} = \ell$, so $\ell b^k \le r_{n+1} \le (\ell+1)b^k$. Thus the endpoints of $J_{n+1}\mat{\a, & k \\ \c, & \ell}$ will be
\[
\frac{p_n \ell b^k + p_{n-1} b^{a_n}}{q_n \ell b^k + q_{n-1} b^{a_n}} \hspace{1cm} \text{and} \hspace{1cm} \frac{p_n (\ell+1)b^k + p_{n-1} b^{a_n}}{q_n (\ell+1)b^k + q_{n-1} b^{a_n}}.
\]
Thus
\begin{align*}
\M J_n\mat{\a \\ \c} &= \left| \frac{p_n}{q_n} - \frac{p_n + p_{n-1}b^{a_n}}{q_n + q_{n-1}b^{a_n}}\right| = \left|\frac{p_n q_{n-1} b^{a_n} - p_{n-1} q_n b^{a_n}}{q_n(q_n + q_{n-1}b^{a_n})}\right| \\
&= \frac{b^{a_1 + \cdots + a_n}}{q_n(q_n + q_{n-1}b^{a_n})} = \frac{b^{a_1 + \cdots + a_n}}{q_n^2\left(1+\frac{q_{n-1} b^{a_n}}{q_n}\right)},
\end{align*}
and
\begin{align*}
\M J_{n+1}\mat{\a, & k \\ \c, & \ell} &= \left|\frac{p_n \ell b^k + p_{n-1} b^{a_n}}{q_n \ell b^k + q_{n-1} b^{a_n}} - \frac{p_n (\ell+1)b^k + p_{n-1} b^{a_n}}{q_n (\ell+1)b^k + q_{n-1} b^{a_n}} \right| \\
&= \left|\frac{p_n q_{n-1} \ell b^{a_n + k} + p_{n-1}q_n(\ell+1)b^{a_n + k} - p_nq_{n-1}(\ell+1)b^{a_n + k} - p_{n-1}q_n \ell b^{a_n + k}}{(q_n \ell b^k + q_{n-1}b^{a_n})(q_n (\ell+1)b^k + q_{n-1} b^{a_n}}\right| \\
&= \left|\frac{b^{a_1+\cdots+a_n+k}}{\ell(\ell+1) b^{2k} q_n^2 \left(1+\frac{q_{n-1}b^{a_n}}{q_n \ell b^k}\right)\left(1 + \frac{q_{n-1}b^{a_n}}{q_n(\ell+1)b^k}\right)}\right| \\
&= \frac{b^{a_1 + \cdots + a_n}}{\ell(\ell+1) b^k q_n^2\left(1+\frac{q_{n-1}b^{a_n}}{q_n \ell b^k}\right) \left(1+\frac{q_{n-1}b^{a_n}}{q_n(\ell+1)b^k}\right)},
\end{align*}
so
\[
\frac{\M J_{n+1}\mat{\a, & k \\ \c, & \ell}}{\M J_n\mat{\a \\ \c}} = \frac{1+\frac{q_{n-1}b^{a_n}}{q_n}}{\ell(\ell+1) b^k\left(1+\frac{q_{n-1}b^{a_n}}{q_n \ell b^k}\right)\left(1+\frac{q_{n-1}b^{a_n}}{q_n(\ell+1)b^k}\right)}.
\]
Now notice that $q_n = c_nb^{a_n} q_{n-1} + b^{a_{n-1}} q_{n-2} \ge b^{a_n} q_{n-1}$, so $0 \le \frac{q_{n-1} b^{a_n}}{q_n} \le 1$, $0 \le \frac{q_{n-1} b^{a_n}}{q_n \ell b^k} \le 1$, and $0 \le \frac{q_{n-1}b^{a_n}}{q_n(\ell+1)b^k} \le 1$, and thus
\[
\frac14 \le \frac{1+\frac{q_{n-1}b^{a_n}}{q_n}}{\left(1+\frac{q_{n-1}b^{a_n}}{q_n \ell b^s}\right)\left(1+\frac{q_{n-1}b^{a_n}}{q_n(\ell+1)b^k}\right)} \le 2.
\]
Therefore
\[
\frac{1}{4 \ell(\ell+1) b^{k}} \M J_n\mat{\a \\ \c} \le \M J_{n+1}\mat{\a, & k \\ \c, & \ell} \le \frac{2}{\ell(\ell+1) b^k} \M J_n\mat{\a \\ \c},
\]
and we are done.
\end{proof}

\begin{corollary}\label{cor:weakMDnbound}
Let $n \in \N$, $k \in \Z_{\ge 0}$, and $\ell \in \{1,\dots,b-1\}$. Then
\[
\frac{1}{4\ell(\ell+1) b^k} \le \M( D_{n+1}(k,\ell)) \le \frac{2}{\ell(\ell+1) b^k}.
\]
\end{corollary}
\begin{proof}
Note that any two distinct intervals of rank $n$ are disjoint. Thus we can add up the above inequality over all intervals of rank $n$, noting that
\[
\bigcup^{(n)} J_n\mat{a_1, & \dots, & a_n \\ c_1, & \dots, & c_n} = (1,2),
\]
so
\[
\sum^{(n)} \M J_n\mat{a_1, & \dots, & a_n \\ c_1, & \dots, & c_n} = \M(1,2) = 1,
\]
and that
\[
\bigcup^{(n)} J_{n+1}\mat{a_1, & \dots, & a_n, & k \\ c_1, & \dots, & c_n, & \ell} = D_{n+1}(k,\ell),
\]
so
\[
\sum^{(n)} \M J_{n+1}\mat{a_1, & \dots, & a_n, & k \\ c_1, & \dots, & c_n, & \ell} = \M D_{n+1}(k,\ell).
\]
This gives
\[
\frac{1}{4\ell(\ell+1) b^k} \le \M D_{n+1} \mat{k \\ \ell} \le \frac{2}{\ell(\ell+1) b^k},
\]
as needed.
\end{proof}

\subsection{Distribution of Type III Continued Logarithm Terms}\label{subsec:type3dist}

\begin{definition}\label{def:zn,Mn,mn}
Let $x = [1, c_1b^{a_1}, c_2b^{a_2}, \dots]_{\cl_3(b)} \in (1,2)$ and $r_n = r_n(x) = [c_nb^{a_n}, c_{n+1}b^{a_{n+1}}, \dots]_{\cl_3(b)}$, as per Remark~\ref{rem:convergent,remainder}. Define 
\[
z_n = z_n(x) = \frac{r_n}{b^{a_n}} - c_n + 1 = [1, c_{n+1}b^{a_{n+1}}, c_{n+2}b^{a_{n+2}}, \dots]_{\cl_3(b)} \in (1,2),
\]
\[
M_n(x) = \{\alpha \in (1,2) : z_n(\alpha) < x\} \subseteq (1,2),
\]
\[
m_n(x) = \M M_n(x) \in (0,1),
\]
and
\[
m(x) = \lim_{n\to\infty} m_n(x),
\]
wherever this limit exists.
\end{definition}

We now get a recursion relation for the sequence of functions $m_n$.

\begin{theorem}\label{thm:mnrecurrence}
The sequence of functions $m_n$ is given by the recursive relationship
\begin{align}
m_0(x) &= x-1 \label{eq:m0} \\
m_n(x) &= \sum_{k=0}^\infty \sum_{\ell = 1}^{b-1} m_{n-1}(1+\ell^{-1}b^{-k}) - m_{n-1}(1+(x+\ell-1)^{-1}b^{-k}) && n \ge 1 \label{eq:mn}
\end{align}
for $1 \le x \le 2$.
\end{theorem}
\begin{proof}
Notice that $r_0(\alpha) = \alpha$, $a_0 = 0$, and $c_0 = 1$, so $z_0(\alpha) = \frac{r_0}{b^{a_0}} - c_0 + 1 = \alpha$ and thus 
\[
M_0(x) = \{\alpha \in (1,2) : z_0(\alpha) < x\} = \{\alpha \in (1,2) : \alpha < x\} = (1,x),
\] 
so $m_0(x) = x-1$. Now fix $n \ge 1$. Since $a_n \in \Z_{\ge 0}$ and $c_n \in \{1,\dots,b-1\}$, we have
\[
m_n(x) = \M\{\alpha \in (1,2) : z_n < x\} = \M \bigcup_{k=0}^\infty \bigcup_{\ell=1}^{b-1} \{\alpha \in (1,2) : z_n < x, a_n = k, c_n = \ell\}.
\]
Fix $x \in (1,2)$ and let
\[
A_{k,\ell} = \{\alpha \in (1,2) : z_n < x, a_n = k, c_n = \ell\}
\]
for $k \in \Z_{\ge 0}$ and $\ell \in \{1,\dots,b-1\}$. By Definition~\ref{def:zn,Mn,mn}, $z_n < x$ if and only if 
\begin{equation*}
\frac{r_n}{b^{a_n}} - c_n + 1 < x.
\end{equation*}
Notice that
\[
z_{n-1} = [1,c_nb^{a_n}, c_{n+1}b^{a_{n+1}}, \dots]_{\cl_3(b)} = [1,r_n]_{\cl_3(b)} = 1+\frac{1}{r_n},
\]
so $z_n < x$ if and only if
\[
\frac{1}{b^{a_n}(z_{n-1}-1)} - c_n + 1 < x,
\]
or equivalently
\begin{equation}\label{eq:mnrecineq1}
z_{n-1} > 1 + (x+c_n-1)^{-1} b^{-a_n} = 1+(x+\ell-1)^{-1}b^{-k}.
\end{equation}
Additionally, in order to have $a_n = k$ and $c_n = \ell$, we must have $\ell b^k \le r_n < (\ell+1)b^k$, or equivalently,
\begin{equation}\label{eq:mnrecineq2}
1 + (\ell+1)^{-1}b^{-k} < z_{n-1} \le 1+\ell^{-1}b^{-k}.
\end{equation}
Now notice that since $x < 2$,
\[
1+(\ell+1)^{-1}b^{-k} < 1+(x+\ell-1)^{-1}b^{-k},
\]
and thus the left hand inequality in \eqref{eq:mnrecineq2} is implied by \eqref{eq:mnrecineq1}. Therefore $z_n < x$ with $a_n = k$ and $c_n = \ell$ if and only if
\begin{equation}\label{eq:mnrecineq3}
1+(x+\ell-1)^{-1}b^{-k} < z_{n-1} \le 1+\ell^{-1}b^{-k}.
\end{equation}
Thus
\begin{equation}\label{eq:mnrecAkl}
A_{k,\ell} = \{\alpha \in (1,2) : 1+(x+\ell-1)^{-1}b^{-k} < z_{n-1} \le 1+\ell^{-1}b^{-k}\}.
\end{equation}
Now suppose $k_1,k_2 \in \Z$ and $\ell_1,\ell_2 \in \{1,\dots,b-1\}$ with $(k_1,\ell_1) \ne (k_2,\ell_2)$. We claim that $A_{k_1,\ell_1}$ and $A_{k_2,\ell_2}$ are disjoint. Consider two cases:

Case 1: $k_1 \ne k_2$. Suppose (without loss of generality) that $k_2 < k_1$, so $k_2 - k_1 \le -1$. Also note that $1 \le \ell_2 \le b-1$ and $x < 2$ so $\ell_2+x-1 < b$. Then we have
\begin{equation}\label{eq:mnreccase1}
1+\ell_1^{-1}b^{-k_1} = 1+\ell_1^{-1} b^{k_2 - k_1} b^{-k_2} \le 1+ \ell_1^{-1}b^{-1}b^{-k_2} \le 1+b^{-1}b^{-k_2}  < 1+(\ell_2+x-1)^{-1}b^{-k_2}.
\end{equation}

Case 2: $k_1 = k_2$, $\ell_1 \ne \ell_2$. Suppose (without loss of generality) that $\ell_1 > \ell_2$, so indeed $\ell_1 \ge \ell_2+1$. Then since $x - 1 < 1$,
\begin{equation}\label{eq:mnreccase2}
1+\ell_1^{-1}b^{-k_1} = 1+\ell_1^{-1}b^{-k_2} \le 1+(\ell_2+1)^{-1}b^{-k_2} < 1 + (\ell_2+x-1)^{-1}b^{-k_2}.
\end{equation}

Now suppose $a_1 \in A_{k_1,\ell_1}$ and $a_2 \in A_{k_2, \ell_2}$. By \eqref{eq:mnrecAkl} and either \eqref{eq:mnreccase1} or \eqref{eq:mnreccase2}, 
\[
a_1 \le 1+\ell_1^{-1}b^{-k_1} < 1 + (\ell_2+x-1)^{-1}b^{-k_2} \le a_2,
\]
so $a_1 \ne a_2$ and thus $A_{k_1,\ell_1}$ and $A_{k_2,\ell_2}$ must be disjoint. Therefore
\begin{equation} \label{eq:mnrecmun}
m_n(x) = \M \bigcup_{k=0}^\infty \bigcup_{\ell=1}^{b-1} A_{k,\ell} = \sum_{k=0}^\infty \sum_{\ell=1}^{b-1} \mathcal{M}(A_{k,\ell}).
\end{equation}
Finally, since $m_{n-1}(x) = \mathcal{M}\{\alpha \in (1,2) : z_{n-1} < x\}$, by \eqref{eq:mnrecAkl} and \eqref{eq:mnrecmun} we can conclude
\[
m_n(x) = \sum_{k=0}^\infty \sum_{\ell=1}^{b-1} \left( m_{n-1}(1+\ell^{-1}b^{-k}) - m_{n-1}\left(1+(x+\ell-1)^{-1}b^{-k}\right)\right),
\]
which proves the recursion \eqref{eq:mn}, and completes the proof of the theorem.
\end{proof}

\newcommand{\distthmtext}{
There exist constants $A, \lambda > 0$ such that
\[
\left|m_n(x) - \frac{\log \frac{bx}{x+b-1}}{\log \frac{2b}{b+1}} \right| < A e^{-\lambda \sqrt{n}}
\]
for all $n \ge 0$ and $x \in (1,2)$.
}

\begin{theorem}\label{thm:mnlimbd}
\distthmtext
\end{theorem}

\newtheorem*{distthm}{Theorem \ref{thm:mnlimbd}}
\newcommand{\restatedistthm}{
\begin{distthm}[Restated]
\distthmtext
\end{distthm}}

The proof of this theorem, which is based on the proof in Section 15 of \cite{khinchine}, is lengthy and somewhat technical. It is provided in appendix A, and the following corollary immediately follows.

\begin{corollary}\label{cor:mndist} We have
\[
m(x) = \cfrac{\log \cfrac{bx}{x+b-1}}{\log \cfrac{2b}{b+1}}
\]
for all $x \in (1,2)$.
\end{corollary}

\begin{longonly}
This limiting distribution is shown in figure~\ref{fig:type3distplot} for bases $b = 2, 3, 4, 5$. Notice that the limiting distributions for the different bases are much more similar than they were for type I and type II continued logarithms. (c.f. figures~\ref{fig:type1distplot} and~\ref{fig:type2distplot}).

\begin{figure}[ht]
	\centering
		\includegraphics[width=0.75\textwidth]{type3dist.jpg}
	\caption{Type III $m(x)$ for $2 \le b \le 5$}
	\label{fig:type3distplot}
\end{figure}
\end{longonly}

\begin{theorem}\label{thm:genprobdist} We have
\[
\M (D_{n+1}(k,\ell)) = m_n(1+(\ell+1)^{-1}b^{-k}) - m_n(1+\ell^{-1}b^{-k}).
\]
\end{theorem}

\begin{proof}
Suppose that $\alpha \in D_{n+1}(k,\ell)$. Then $a_{n+1} = k$ and $c_{n+1} = \ell$, so
\[
z_n = [1, \ell b^k, r_{n+2}] = 1 + \cfrac{1}{\ell b^k + \cfrac{b^k}{r_{n+2}}},
\]
where $r_{n+2}$ can take any value in $(1,\infty)$. Clearly $z_n$ is a monotonic function of $r_{n+2}$ for fixed $k,\ell$, so the extreme values of $z_n$ on $D_{n+1}(k,\ell)$ will occur at the extreme values of $r_{n+2}$. Letting $r_n \to 1$ gives $z_n = 1+\frac{1}{\ell b^k + b^k} = 1+(\ell+1)^{-1} b^{-k}$ and letting $r_n \to \infty$ gives $z_n = 1+\frac{1}{\ell b^k + 0} = \ell^{-1}b^{-k}$. Thus
\begin{align*}
D_{n+1}(k,\ell) &= \{\alpha \in (1,2) : 1+(\ell+1)^{-1} b^{-k} < z_n(\alpha) \le 1+\ell^{-1}b^{-k} \} \\
&= M_n(1+\ell^{-1}b^{-k}) \setminus M_n(1+(\ell+1)^{-1} b^{-k}),
\end{align*}
so
\[
\M D_{n+1}(k,\ell) = m_n(1+\ell^{-1}b^{-k}) - m_n(1+(\ell+1)^{-1} b^{-k}).
\]
\end{proof}

\newcommand{\MDnboundthmtext}{
There exist constants $A, \lambda > 0$ such that
\[
\left| \M ( D_n (k,\ell )) - \frac{\log \frac{(\ell b^k + 1)((\ell +1)b^{k+1} + 1)}{(\ell b^{k+1} + 1)((\ell +1)b^k + 1)}}{\log \frac{2b}{b+1}} \right| < \frac{A e^{-\lambda \sqrt{n-1}}}{\ell (\ell +1)b^k}
\]
for all $k \in \Z_{\ge 0}, \ell \in \{1,2,\dots,b-1\}$ and $n \in \Z_{\ge 0}$.
}

\begin{theorem}\label{thm:MDnbound}
\MDnboundthmtext
\end{theorem}

\newtheorem*{MDnboundthm}{Theorem \ref{thm:MDnbound}}
\newcommand{\restateMDnboundthm}{
\begin{MDnboundthm}[Restated]
\MDnboundthmtext
\end{MDnboundthm}}

We then immediately get the following limiting distribution. 
\begin{longonly}
This distribution is shown in figure~\ref{fig:type3probplot} for $b = 2, 3, 4, 5$. 
\end{longonly} 
Notice that like with type II continued fractions, the distribution is non-monotonic. This is due to the gaps in possible denominator terms. For example, for base 4, the possible denominator terms are $1,2,3,4,8,12,\dots$. The jump from 4 to 8 causes a spike in the limiting distribution.

\begin{corollary}\label{cor:limDn} We have
\[
\lim_{n \to \infty}\M( D_n (k,\ell))) = \frac{\log \frac{(\ell b^k + 1)((\ell +1)b^{k+1} + 1)}{(\ell b^{k+1} + 1)((\ell +1)b^k + 1)}}{\log \frac{2b}{b+1}}
\]
for $k \in \Z_{\ge 0}$ and $\ell \in \{1,2,\dots,b-1\}$.
\end{corollary}

\begin{longonly}
\begin{figure}[ht]
	\centering
	\begin{tabular}{cc}
	\includegraphics[width=6cm]{type3base2prob.jpg} & \includegraphics[width=6cm]{type3base3prob.jpg} \\
	\includegraphics[width=6cm]{type3base4prob.jpg} & \includegraphics[width=6cm]{type3base5prob.jpg}
	\end{tabular}
	\caption{Type III limiting distribution for $2 \le b \le 5$}
	\label{fig:type3probplot}
\end{figure}
\end{longonly}

\subsection{Type III Logarithmic Khinchine Constant}\label{subsec:type3khinchine}

We now extend the Khinchine constant to type III continued logarithms. Note that we only gave an overview for type I and type II, but here we will be much more rigourous.

\begin{definition}\label{def:Pkl}
Let $\alpha \in (1,\infty)$ have type III continued logarithm $[c_0b^{a_0}, c_1b^{a_1}, c_2b^{a_2}, \dots]_{\cl_3(b)}$. Let $k \in \Z_{\ge 0}$ and $\ell \in \{1,2,\dots,b-1\}$. We define
\[
P_\alpha(k, \ell) = \lim_{N \to \infty} \frac{|\{n \in \N : a_n = k, c_n = \ell\}|}{N}
\]
to be the limiting proportion of continued logarithm terms of $\alpha$ that have $a_n = k$ and $c_n = \ell$, if this limit exists.
\end{definition}

Note that for the theorems which follow, we will restrict our study to $(1,2)$ instead of $(1,\infty)$. The results can be easily extended to $(1,\infty)$ by noting that every $\alpha \in (1,\infty)$ corresponds to an $\alpha' \in (1,2)$ in the sense that the continued logarithm of $\alpha'$ is just the continued logarithm of $\alpha$ with the first term replace by 1. Since we are looking at limiting behaviour over all terms, changing the first term will have no impact.

\newcommand{\khinchinedistthmtext}{
For almost every $\alpha \in (1,2)$ with continued logarithm $[1, c_1b^{a_1}, c_2b^{a_2}, \dots]_{\cl_3(b)}$ we have
\[
P_\alpha(k,\ell) = \frac{\log\frac{(1+\ell^{-1}b^{-k})(b+(\ell+1)^{-1}b^{-k})}{(b+\ell^{-1}b^{-k})(1+(\ell+1)^{-1}b^{-k})}}{\log\frac{2b}{b+1}}
\]
for all $k \in \Z_{\ge 0}$ and $\ell \in \{1,2,\dots,b-1\}$.
}

\newcommand{\khinchinethmtext}{
For almost every $\alpha \in (1,2)$ with continued logarithm $[1, c_1b^{a_1}, c_2b^{a_2}, \dots]_{\cl_3(b)}$ we have
\[
\lim_{N \to \infty} \left(\prod_{n=1}^N (c_n b^{a_n})\right)^{1/N} = b^{\mathcal{A}_b},
\]
where
\[
\mathcal{A}_b = \frac{1}{\log b \log \frac{b+1}{2b}} \sum_{\ell=2}^b \log\left(1-\frac1\ell\right) \log \left(1+\frac1\ell\right).
\] 
}

The following two theorems are proved in Appendix B. The proofs are based on the analogous proofs for simple continued fractions that are presented in Sections 15 and 16 of \cite{khinchine}.

\begin{theorem}\label{thm:khinchinedist}
\khinchinedistthmtext
\end{theorem}

\begin{theorem}\label{thm:logkhinchine}
\khinchinethmtext
\end{theorem}

\newtheorem*{khinchinedistthm}{Theorem \ref{thm:khinchinedist}}
\newcommand{\restatekhinchinedistthm}{
\begin{khinchinedistthm}[Restated]
\khinchinedistthmtext
\end{khinchinedistthm}}

\newtheorem*{khinchinethm}{Theorem \ref{thm:logkhinchine}}
\newcommand{\restatekhinchinethm}{
\begin{khinchinethm}[Restated]
\khinchinethmtext
\end{khinchinethm}}

The values of the Khinchine constant given by the above formula for $2 \le b \le 10$ are shown in Figure~\ref{fig:type3klresults}.

\begin{figure}[ht]
\begin{center}
\begin{tabular}{c|c} 
 $b$ & $\KLiii_b$ \\ \hline
2 & 2.656305058 \\
3 & 2.666666667 \\
4 & 2.671738848 \\
5 & 2.674705520 \\
6 & 2.676638451 \\
7 & 2.677992355 \\
8 & 2.678991102 \\
9 & 2.679757051 \\
10 & 2.680362475 \\ 
\end{tabular}
\end{center}
\caption{Type III logarithmic Khinchine constants for $2 \le b \le 10$}
\label{fig:type3klresults}
\end{figure}

\begin{remark}\label{rem:distvskhinchine}
Notice that Theorem~\ref{thm:khinchinedist} is similar to Corollary~\ref{cor:limDn}. However, Corollary~\ref{cor:limDn} is about the limiting proportion of numbers $\alpha \in (1,2)$ that have $a_n = k$ and $c_n = \ell$, whereas Theorem~\ref{thm:khinchinedist} is about the limiting proportion of terms of a number $\alpha \in (1,2)$ for which $a_n = k$ and $c_n = \ell$. The fact that these two limits are the same is not a coincidence: one can show that Corollary~\ref{cor:limDn} is a consequence of Theorem~\ref{thm:khinchinedist}.

Based on Theorem~\ref{thm:logkhinchine}, we denote
\[
\KLiii_b = b^{\mathcal{A}_b},
\]
where $\mathcal{A}_b$ is as in Theorem~\ref{thm:logkhinchine}.
\end{remark}

\subsection{Type III Continued Logarithms and Simple Continued Fractions}\label{subsec:type3clogsandcfracs}

Now suppose $b$ is no longer fixed. Let $\mu_b$ denote the limiting distribution for a given base $b$, as shown in Corollary~\ref{cor:mndist}. That is,
\[
\mu_b(x) = \frac{\log\frac{bx}{x+b-1}}{\log\frac{2b}{b+1}}.
\]
Furthermore, let $\KLiii_b$ denote the base $b$ logarithmic Khinchine constant, as in Remark~\ref{rem:distvskhinchine}, and let $\mathcal{K}$ denote the Khinchine constant for simple continued fractions, as in Section~\ref{sec:khinchine}.

We now have an interesting relationship between these logarithmic Khinchine constants and the Khinchine constant for simple continued fractions, based on the following lemma.

\begin{lemma}[\cite{onKhinchine}, Lemma 1(c)]
\[
\sum_{\ell=2}^\infty \log\left(1-\frac1\ell\right) \log\left(1+\frac1\ell\right) = -\log \mathcal{K} \log 2.
\]
\end{lemma}

\begin{theorem}
\[
\lim_{b \to \infty} \KLiii_b = \mathcal{K}.
\]
\end{theorem}
\begin{proof}
We will show that $\lim_{b \to \infty} \log \KLiii_b = \log \mathcal{K}$, from which the desired limit immediately follows.
\begin{align*}
\lim_{b \to \infty} \log \KLiii_b &= \lim_{b \to \infty} \log b^{\mathcal{A}_b} = \lim_{b\to\infty} (\log b) \mathcal{A}_b \\
&= \lim_{b \to \infty} \frac{\log b}{\log b \log \frac{b+1}{2b}}\sum_{k=2}^{b} \log \left(1-\frac1k\right) \log \left(1+\frac1k\right) \\
&= \lim_{b \to \infty} \frac{1}{\log \frac{b+1}{2b}} \sum_{k=2}^b \log\left(1-\frac1k\right) \log \left(1+\frac1k\right) \\
&= \frac{1}{\lim_{b \to \infty} \log \left(\frac12 \frac{b+1}{b}\right)} \sum_{k=2}^\infty \log \left(1-\frac1k\right)\log\left(1+\frac1k\right) \\
&= - \frac{1}{\log 2}\sum_{k=2}^\infty \log \left(1-\frac1k\right)\log\left(1+\frac1k\right) = \log \mathcal{K}.
\end{align*}
\end{proof}

Furthermore, as $b \to \infty$, the distribution function $\mu_b$ approaches the appropriately shifted continued fraction distribution $\mu_{\text{cl}}$. The continued fraction distribution function is given by
\[
\mu_{\text{cl}}(x) = \log_2(1+x) \hspace{1.5cm} x \in (0,1).
\]
(See Section 3.4 of \cite{nef}.) Since the continued fraction for a number will be unchanged (except for the first term) when adding an integer, we can shift this distribution to the right and think of it as a distribution over $(1,2)$ instead of $(0,1)$, in order to compare it to $\mu_b$. We define the shifted continued fraction distribution
\[
\mu_{\text{cl}}^*(x) = \mu_{\text{cl}}(x-1) = \log_2 x \hspace{1.5cm} x \in (1,2).
\]
We then have
\begin{align*}
\lim_{b \to \infty} \mu_b(x) &= \lim_{b \to \infty} \frac{\log \frac{x+b-1}{bx}}{\log \frac{b+1}{2b}} = \lim_{b \to \infty} \frac{\log\left(\frac1b + \frac1x - \frac1{bx}\right)}{\log \left(\frac12 \frac{b+1}{b}\right)} = \frac{\log \frac1x}{\log \frac12} = \frac{-\log x}{- \log 2} = \log_2(x) = \mu_{\text{cf}}^*(x).
\end{align*}

This shows that, in some sense, as we let $b \to \infty$ for type III continued logarithms, we get in the limit simple continued fractions.

\section{Generalizing Beyond Continued Logarithms}\label{sec:generalizing}

A natural question that arises is how one can define something more general than continued logarithms. Consider the following definition of generalized continued fractions.

\begin{definition}\label{def:gcfrac}
Let $(c_n)_{n=0}^\infty$ be an increasing sequence of natural numbers with $c_0 = 1$. Let $\alpha \in (1,\infty)$. The generalized continued fraction for $\alpha$ determined by $(c_n)_{n=0}^\infty$ is
\[
a_0 + \cFrac{b_0}{a_1} + \cFrac{b_1}{a_2} + \cFrac{b_2}{a_3} + \cdots = [a_0,a_1,a_2,\dots]_{\text{gcf}},
\]
where the the terms $a_0, a_1, \dots$ and $b_0, b_1, \dots$ are determined by the following recursive process, terminating at the term $a_n$ if $y_n = a_n$.
\begin{align*}
y_0 &= \alpha \\
j_n &= \max\{j : c_j \le y_n\} && n \ge 0 \\
a_n &= c_{j_n}  && n \ge 0 \\
b_n &= c_{j_n + 1} - c_{j_n} && n \ge 0 \\
y_{n+1} &= \frac{b_n}{y_n-a_n} = \frac{c_{j_n+1} - c_{j_n}}{y_n - c_{j_n}} && n \ge 0.
\end{align*}
\end{definition}

\begin{remark}
This is a generalization of simple continued fractions, and of type I and type III continued logarithms. Indeed, for simple continued fractions, the term sequence $(c_n)_{n=0}^\infty$ consists of the natural numbers. For type I continued logarithms, the term sequence consists of the powers $b^0, b^1, b^2, \dots$. For type III continued logarithms, the term sequence consists of terms of the form $\ell b^k$, where $k \in \Z_{\ge 0}$ and $\ell \in \{1,\dots,b-1\}$.

Recall from Remark~\ref{rem:type2numeratorterms} that type II continued logarithms did not have the property that $y_{n+1}$ could take any value in $(1,\infty)$, regardless of the values of $a_n, c_n$. This is a desirable property to have, since it uniquely determines the numerator terms based on the corresponding denominator terms. We have defined generalized continued logarithms so that they have this property, and for that reason they are not a generalization of type II continued logarithms.
\end{remark}

\begin{remark}
As per Definitions~\ref{def:cfracconv} and~\ref{def:cfracremainderterm}, the $n$th convergent and $n$th remainder term are given by
\[
x_n = [a_0, a_1, \dots, a_n]_{\text{gcf}} \hspace{1cm} \text{and} \hspace{1cm} r_n = [a_n, a_{n+1}, a_{n+2}, \dots]_{\text{gcf}},
\]
respectively. Note that the remainder terms $r_n$ and the terms $y_n$ from Definition~\ref{def:gcfrac} are in fact the same.
\end{remark}

\begin{shortonly}
We can derive various results for generalized continued fractions that are similar to those for continued logarithms. Most notably, we get the following sufficient criteria for guaranteed convergence and rational finiteness.
\end{shortonly}

\begin{longonly}
We now present various preliminary results about these generalized continued fractions. We start with a few basic lemmas.
\end{longonly}

\begin{lemmax}
The $n$th convergent of $\alpha = [a_0,a_1,a_2,\dots]_{\text{gcf}}$ is given by $x_n = \frac{p_n}{q_n}$, where $p_{-1} = 1$, $q_{-1} = 0$, $p_0 = a_0 = c_{j_0}$, $q_0 = 1$, and for $n \ge 0$,
\begin{align*}
p_n &= a_n p_{n-1} + b_{n-1} p_{n-2} = c_{j_n} p_{n-1} + (c_{j_{n-1}+1} - c_{j_{n-1}}) p_{n-2} \\
q_n &= a_n q_{n-1} + b_{n-1} q_{n-2} = c_{j_n} q_{n-1} + (c_{j_{n-1}+1} - c_{j_{n-1}}) q_{n-2}.
\end{align*}
\end{lemmax}

\begin{proofx}
This follows from Fact~\ref{fact:cfracrecurrence}, where $\alpha_n = a_n = c_{j_n}$ and $\beta_n = b_{n-1} = c_{j_{n-1}+1} - c_{j_{n-1}}$.
\end{proofx}

\begin{lemmax}
For $n \ge 0$,
\[
p_n q_{n-1} - q_n p_{n-1} = (-1)^{n-1} \prod_{k=0}^{n-1} b_k = (-1)^{n-1} \prod_{k=0}^{n-1} (c_{j_k+1}-c_{j_k}).
\]
\end{lemmax}

\begin{proofx}
For $n = 0$,
\[
p_0 q_{-1} - q_0 p_{-1} = a_0(0) - 1(1) = -1.
\]
Now suppose the statement is true for $n$. Then
\begin{align*}
p_{n+1} q_n - q_{n+1} p_n &= (a_{n+1} p_n + b_n p_{n-1})q_n - (a_{n+1} q_n + b_n q_{n-1})p_n = -b_n(p_n q_{n-1} - q_n p_{n-1}) \\
&= -b_n(-1)^{n-1} \prod_{k=0}^{n-1} b_k = (-1)^n \prod_{k=0}^n b_k,
\end{align*}
so the result follows by induction.
\end{proofx}

\begin{lemmax}\label{lem:gcfpqmatrix}
Let $b_{-1} = 1$. Then for all $n \ge 0$,
\[
\mat{p_n & p_{n-1} \\ q_n & q_{n-1}} = \prod_{k=0}^n \mat{a_k & 1 \\ b_{k-1} & 0} = \prod_{k=0}^n \mat{c_{j_k} & 1 \\ c_{j_{k-1}+1} - c_{j_{k-1}} & 0}.
\]
\end{lemmax}
\begin{proofx}
For $n=0$, we have
\[
\prod_{j=0}^0 \mat{a_k & 1 \\ b_{k-1} & 0} = \mat{a_0 & 0 \\ b_{-1} & 0} = \mat{a_0 & 1 \\ 1 & 0} = \mat{p_0 & p_{-1} \\ q_0 & q_{-1}}.
\]
Now suppose for induction that
\[
\prod_{j=0}^{n-1} \mat{a_k & 1 \\ b_{k-1} & 0} = \mat{p_{n-1} & p_{n-2} \\ q_{n-1} & q_{n-2}}.
\]
Then by Lemma~\ref{lem:pqlemma},
\begin{align*}
\prod_{j=0}^n \mat{a_k & 1 \\ b_{k-1} & 0} &= \mat{p_{n-1} & p_{n-2} \\ q_{n-1} & q_{n-2}} \mat{a_n & 1 \\ b_{n-1} & 0} = \mat{a_n p_{n-1} + b_{n-1}p_{n-2} & p_{n-1} \\ a_nq_{n-1} + b_{n-1} q_{n-2} & q_n} = \mat{p_n & p_{n-1} \\ q_n & q_{n-1}}.
\end{align*}
\end{proofx}

\begin{theoremx}\label{thm:gcfremaindertheorem}
For arbitrary $1 \le k \le n$,
\[
[a_0,a_1,\dots,a_n]_{\text{gcf}} = \frac{p_{k-1}r_k + p_{k-2}b_{k-1}}{q_{k-1}r_k + q_{k-2}b_{k-1}}.
\]
\end{theoremx}
\begin{proofx}
First notice that $r_k = [a_k, \dots, a_n]_{\text{gcf}} = \frac{p_k'}{q_k'}$, where
\[
\mat{p_k' \\ q_k'} = \mat{a_k & 1 \\ 1 & 0} \prod_{i=k+1}^n \mat{a_i & 1 \\ b_{i-1} & 0} \mat{1\\0}.
\]
Also note that
\[
\mat{a_k & 1 \\ b_{k-1} & 0} = \mat{1 & 0 \\ 0 & b_{k-1}} \mat{a_k & 1 \\ 1 & 0}.
\]
Then
\begin{align*}
\mat{p_n \\ q_n} &= \prod_{i=0}^n \mat{a_i & 1 \\ b_{i-1} & 0} \mat{1 \\ 0} = \prod_{i=0}^{k-1} \mat{a_i & 1 \\ b_{i-1} & 0} \mat{a_k & 1 \\ b_{k-1} & 0} \prod_{i=k+1}^n \mat{a_i & 1 \\ b_{i-1} & 0} \mat{1\\0} \\
&= \mat{p_{k-1} & p_{k-2} \\ q_{k-1} & q_{k-2}} \mat{1 & 0 \\ 0 & b_{k-1}} \mat{a_k & 1 \\ 1 & 0} \prod_{i=k+1}^n \mat{a_i & 1 \\ b_{i-1} & 0} \mat{1\\0} \\
&= \mat{p_{k-1} & p_{k-2} b_{k-1}  \\ q_{k-1} & q_{k-2} b_{k-1}} \mat{p_k' \\ q_k'} = \mat{p_{k-1}p_k' + p_{k-2} b_{k-1} q_k' \\ q_{k-1}p_k' + q_{k-2} b_{k-1} q_k'}.
\end{align*}
Thus
\begin{align*}
[a_0, \dots, a_n]_{\text{gcf}} &= \frac{p_n}{q_n} = \frac{p_{k-1}p_k' + p_{k-2} b_{k-1} q_k'}{q_{k-1}p_k' + q_{k-2} b_{k-1} q_k'} = \frac{p_{k-1} \frac{p_k'}{q_k'} + p_{k-2} b_{k-1}}{q_{k-1} \frac{p_k'}{q_k'} + q_{k-2} b_{k-1}} 
= \frac{p_{k-1} r_k + p_{k-2} b_{k-1}}{q_{k-1} r_k + q_{k-2} b_{k-1}}.
\end{align*}
\end{proofx}

\begin{longonly}
We can now prove some sufficient conditions in order to have guaranteed convergence for all continued logarithms and sufficient conditions for rational finiteness. 
\end{longonly}

\begin{theorem}
Suppose there is a constant $M > 0$ such that $c_{j+1} - c_j < Mc_j$ for all $j$. Then every infinite continued fraction with term sequence $(c_n)_{n=0}^\infty$ will converge.
\end{theorem}
\begin{proofx}
By Fact~\ref{fact:convcrit}, the continued fraction $[a_0, a_1,\dots]_{\text{gcf}}$ will converge if
\[
\sum_{n=1}^\infty \frac{a_n a_{n+1}}{b_n} = \infty.
\]
Let $s_{j} = \frac{c_{j+1} - c_j}{c_j} < M$, so that $b_n = c_{j_n+1}-c_{j_n} = c_{j_n} s_{j_n} = a_n s_{j_n}$. Then
\[
\sum_{n=1}^\infty \frac{a_n a_{n+1}}{b_n} = \sum_{n=1}^\infty \frac{a_n a_{n+1}}{a_n s_{j_n}} >  \frac{1}{M} \sum_{n=1}^\infty a_{n+1} = \infty,
\]
so the continued fraction will converge.
\end{proofx}

\begin{theorem}
Suppose $(c_{n+1} - c_n) \mid c_n$ for all $n \ge 1$. Then for every $\alpha > 1$, the continued fraction of $\alpha$ is finite if and only if $\alpha \in \Q$.
\end{theorem}

\begin{proofx}
Clearly if the continued fraction of $\alpha$ is finite then $\alpha \in \Q$. Conversely, suppose $\alpha \in \Q$ with
\[
\alpha = a_0 + \cFrac{b_0}{a_1} + \cFrac{b_1}{a_2} + \cdots.
\]
We can write
\begin{equation}\label{eq:gcfratfiniteequiv}
\alpha = a_0 \left(1 + \cFrac{a_0^{-1}a_1^{-1}b_0}{1} + \cFrac{a_1^{-1}a_2^{-1}b_1}{1} + \cFrac{a_2^{-1}a_3^{-1}b_2}{1} + \cdots\right).
\end{equation}
Let $y_n$ denote the $n$th tail of the continued fraction in \eqref{eq:gcfratfiniteequiv}, that is,
\[
y_n = 1 + \cFrac{a_n^{-1} a_{n+1}^{-1} b_n}{1} + \cFrac{a_{n+1}^{-1} a_{n+2}^{-1} b_{n+1}}{1} + \cdots.
\]
Notice that $y_n = 1+\frac{a_n^{-1} a_{n+1}^{-1} b_n}{y_{n+1}}$, so $y_{n+1} = \frac{a_n^{-1} a_{n+1}^{-1} b_n}{y_n - 1}$. Clearly, since $\alpha$ is rational, each $y_n$ must be rational, so write $y_n = \frac{u_n}{v_n}$ for positive relatively prime integers $u_n, v_n$. Since $(c_{n+1} - c_n) \mid c_n$, write $c_n = d_n (c_{n+1}-c_n)$ so $a_n = c_{j_n} = d_{j_n} (c_{j_n+1} - c_{j_n}) = d_{j_n} b_n$. We then have
\[
\frac{u_{n+1}}{v_{n+1}} = y_{n+1} = \frac{a_n^{-1} a_{n+1}^{-1} b_n}{\frac{u_n-v_n}{v_n}} = \frac{v_n b_n}{a_n a_{n+1} (u_n - v_n)} = \frac{v_n}{d_{j_n} a_{n+1} (u_n - v_n)},
\]
or equivalently,
\[
d_{j_n} a_{n+1} (u_n-v_n) u_{n+1} = v_n v_{n+1}.
\]
Now since $y_n \ge 1$, $u_n - v_n \ge 0$ so each multiplicative term in the above equation is a nonnegative integer. Since $u_{n+1}$ and $v_{n+1}$ are relatively prime, we must have $u_{n+1} \mid v_n$, so $u_{n+1} \le v_n \le u_n$. If at any point we have $u_{n+1} = v_n = u_n$, then $y_n = \frac{u_n}{v_n} = 1 = c_0$, so the continued fraction terminates. Otherwise, $u_{n+1} < u_n$ so $(u_n)$ is a decreasing sequence of positive integers, which cannot be infinite, so we get a finite continued fraction.
\end{proofx}

\begin{shortonly}
We are also able to extend some of the measure-theoretic results to generalized continued fractions, though details are not provided here. We conjecture that the main results that we derived for the distribution and Khinchine constant of continued logarithms would extend (likely with some additional restrictions on the sequence $(c_n)_{n=0}^\infty$) to our generalized continued fractions.
\end{shortonly}

\begin{longonly}
We can also extend some of the measure-theoretic definitions and results. We assume $a_0 = 1$, so that $\alpha \in [1, c_1)$.
\end{longonly}

\begin{definitionx}
Let $n \in \N$ and $k_1, k_2, \dots, k_n \in \Z_{\ge 0}$. The intervals of rank $n$ are intervals of the form
\[
J_n(k_1,\dots,k_n) = \{x \in (1,c_1) : j_1 = k_1, \dots, j_n = k_n\}.
\]
\end{definitionx}

\begin{definitionx}\label{def:gcfDnk}
Let $n \in \N$ and $k \in \Z_{\ge 0}$. Define
\[
D_n(k) = \{\alpha \in (1,2) : j_n = k\}
\]
to be the set of points where the $n$th continued logarithm term is $c_k$.
\end{definitionx}

\begin{lemmax}\label{lem:gcfendpoints}
The endpoints of $J_n(k_1, \dots, k_n)$ are
\[
\frac{p_n}{q_n} \hspace{1cm} \text{and} \hspace{1cm} \frac{p_n + p_{n-1} b_n}{q_n + q_{n-1} b_n}.
\]
\end{lemmax}

\begin{proofx}
Consider an arbitrary $\alpha \in J_n(k_1,\dots,k_n)$. Note that $\alpha = [1,k_1,\dots,k_n, r_{n+1}]_{\text{gcf}}$ where $r_{n+1}$ can take any real value in $[1,\infty)$. By Theorem~\ref{thm:gcfremaindertheorem}, 
\[
\alpha = \frac{p_n r_{n+1} + p_{n-1} b_n}{q_n r_{n+1} + q_{n-1} b_n}.
\]
Notice that
\[
\alpha - \frac{p_n}{q_n} = \frac{p_n r_{n+1} + p_{n-1} b_n}{q_n r_{n+1} + q_{n-1} b_n} - \frac{p_n}{q_n} = \frac{b_n(p_{n-1}q_n - p_n q_{n-1})}{q_n (q_n r_{n+1} + q_{n-1}b_n)},
\]
and on $J_n(k_1,\dots,k_n)$, all of $b_n, p_{n-1}, q_{n-1}, p_n, q_n$ are fixed. Thus $\alpha$ is a monotonic function of $r_{n+1}$, so the extreme values of $\alpha$ on $J_n(k_1,\dots,k_n)$ will occur at the extreme values of $r_{n+1}$. Taking $r_{n+1} = 1$ gives $\alpha = \frac{p_n + p_{n-1} b_n}{q_n + q_{n-1} b_n}$ and letting $r_{n+1} \to \infty$ gives $\alpha = \frac{p_n}{q_n}$, and thus these are the endpoints of $J_n(k_1,\dots,k_n)$.
\end{proofx}

\begin{theoremx}\label{thm:gcfmjn1bounds}
Suppose $n \in \N$ and $a_1, a_2, \dots, a_n, s \in \Z_{\ge 0}$. Also suppose that there is a constant $M$ such that $c_{j+1} - c_j \le Mc_j$ for all $j$. Let $\a = (a_1, \dots, a_n)$. Then
\[
\frac{c_{k+1}-c_k}{(M+1)^2 c_k c_{k+1}} \M J_n(\a) \le \M J_{n+1}(\a,k) \le \frac{(M+1)(c_{k+1}-c_k)}{c_k c_{k+1}} \M J_n(\a).
\]
\end{theoremx}
\begin{proofx}
From Lemma~\ref{lem:gcfendpoints}, we know that the endpoints of $J_n(\a)$ are
\[
\frac{p_n}{q_n} \hspace{1cm} \text{and} \hspace{1cm} \frac{p_n + p_{n-1}b_n}{q_n + q_{n-1}b_n},
\]
Now in order to be in $J_{n+1}(\a,k)$, we must have $j_{n+1} = k$, so $a_{n+1} = c_k$, and thus $c_k \le r_{n+1} < c_{k+1}$. Thus the endpoints of $J_{n+1}(\a,k)$ will be
\[
\frac{p_n c_k + p_{n-1} b_n}{q_n c_k + q_{n-1} b_n} \hspace{1cm} \text{and} \hspace{1cm} \frac{p_n c_{k+1} + p_{n-1} b_n}{q_n c_{k+1} + q_{n-1} b_n}.
\]
Thus
\begin{align*}
\M J_n(\a) &= \left| \frac{p_n}{q_n} - \frac{p_n + p_{n-1}b_n}{q_n + q_{n-1}b_n}\right| = \left|\frac{p_n q_{n-1} b_n - p_{n-1} q_n b_n}{q_n(q_n + q_{n-1}b_n)}\right|
= \frac{\prod_{j=1}^n b_j}{q_n(q_n + q_{n-1}b_n)} = \frac{\prod_{j=1}^n b_j}{q_n^2\left(1+\frac{q_{n-1} b_n}{q_n}\right)},
\end{align*}
and
\begin{align*}
\M J_{n+1}(\a,k) &= \left|\frac{p_n c_k + p_{n-1} b_n}{q_n c_k + q_{n-1} b_n} - \frac{p_n c_{k+1} + p_{n-1} b_n}{q_n c_{k+1} + q_{n-1} b_n} \right| \\
&= \left|\frac{p_n q_{n-1} b_n c_k + p_{n-1} q_n b_n c_{k+1} - p_{n-1}q_n c_k b_n - p_n q_{n-1} b_n c_{k+1}}{(q_n c_k + q_{n-1} b_n)(q_n c_{k+1} + q_{n-1}b_n)} \right| \\
&= \left|\frac{b_n c_k ( p_n q_{n-1} - p_{n-1} q_n ) + b_n c_{k+1} ( p_{n-1} q_n - p_n q_{n-1} )}{q_n^2 c_k c_{k+1} \left(1+\frac{q_{n-1}b_n}{q_n c_k}\right)\left(1+\frac{q_{n-1} b_n}{q_n c_{k+1}}\right)}\right| \\
&= \left|\frac{(-1)^{n-1} (c_{k+1}-c_k) \prod_{j=1}^n b_j}{q_n^2 c_k c_{k+1} \left(1+\frac{q_{n-1}b_n}{q_n c_k}\right)\left(1+\frac{q_{n-1} b_n}{q_n c_{k+1}}\right)} \right| \\
&= \frac{(c_{k+1}-c_k \prod_{j=1}^n b_j}{q_n^2 c_k c_{k+1} \left(1+\frac{q_{n-1}b_n}{q_n c_k}\right)\left(1+\frac{q_{n-1} b_n}{q_n c_{k+1}}\right)},
\end{align*}
so
\[
\frac{\M J_{n+1}(\a,k)}{\M J_n(\a)} = \frac{(c_{k+1} - c_k) \left(1+\frac{q_{n-1}b_n}{q_n}\right)}{c_k c_{k+1} \left(1+\frac{q_{n-1}b_n}{q_n c_k}\right) \left(1+\frac{q_{n-1}b_n}{q_n c_{k+1}}\right)}.
\]
Now notice that 
\[
q_n = a_n q_{n-1} + b_{n-1} q_{n-2} \ge c_{j_n} q_{n-1} \ge \frac{1}{M} (c_{j_{n+1}}-c_{j_n}) q_{n-1} = \frac{1}{M} b_n q_{n-1},
\]
so $0 \le \frac{q_{n-1}b_n}{q_n} \le M$, $0 \le \frac{q_{n+1} b_n}{q_n c_k} \le M$, and $0 \le \frac{q_{n+1} b_n}{q_n c_{k+1}} \le M$. Additionally, since $c_k \ge 1$, $\frac{1+\frac{q_{n-1}b_n}{q_n}}{1+\frac{q_{n-1}b_n}{q_n c_k}} \le 1$, so
\[
\frac{1}{M+1} \le \frac{\left(1+\frac{q_{n-1}b_n}{q_n}\right)}{\left(1+\frac{q_{n-1}b_n}{q_n c_k}\right) \left(1+\frac{q_{n-1}b_n}{q_n c_{k+1}}\right)} \le M+1,
\]
from which the result follows.
\end{proofx}

\begin{corollaryx}
Let $n \in \N$ and $s \in \Z_{\ge 0}$. Suppose there exists a constant $M \ge 1$ such that $c_{j+1} - c_j \le M c_j$ for all $j$. Then
\[
\frac{(c_{k+1}-c_k)(c_1-1)}{(M+1)c_kc_{k+1}} \le \M D_{n+1}(k) \le \frac{(M+1)(c_{k+1}-c_k)(c_1-1)}{c_k c_{k+1}}.
\]
\end{corollaryx}

\begin{proofx}
Note that distinct intervals of the same rank are disjoint. Thus we can add up the inequality in Theorem~\ref{thm:gcfmjn1bounds} over all intervals of rank $n$, noting that
\[
\bigcup^{(n)} J_n(a_1, \dots, a_n) = (1,c_1),
\]
so
\[
\sum^{(n)} \M J_n(a_1,\dots,a_n) = c_1-1,
\]
and
\[
\bigcup^{(n)} J_n(a_1, \dots, a_n, k) = D_{n+1}(k),
\]
so
\[
\sum^{(n)} \M J_n(a_1,\dots,a_n,k) = \M D_{n+1}(k).
\]
\end{proofx}

\section*{Acknowledgements}

We would like to thank Andrew Mattingly for his input and assistance. This research was initiated at and supported by the Priority Research Centre for Computer-Assisted Research Mathematics and its Applications at the University of Newcastle.

\addcontentsline{toc}{section}{Appendix A: Proof of the Type III Continued Logarithm Distribution}
\section*{Appendix A: Proof of the Type III Continued Logarithm Distribution}

This appendix is devoted to proving Theorems~\ref{thm:mnlimbd} and~\ref{thm:MDnbound}, restated below:

\restatedistthm

\restateMDnboundthm

These proofs are based extensively on the proof presented in Section 15 of \cite{khinchine}, which proves similar statements for simple continued fractions.

\begin{lemma}\label{lem:1/x(x+b-1)sum}
For $x > 1$,
\[
\sum_{k=0}^\infty \sum_{\ell=1}^{b-1} \frac{b^{-k}}{(x+\ell-1)^2} \frac{1}{(1+b^{-k}(x+\ell-1)^{-1})(b+b^{-k}(x+\ell-1)^{-1})} = \frac{1}{x(x+b-1)}.
\]
\end{lemma}
\begin{proof}
\begin{align*}
\sum_{k=0}^\infty &\sum_{\ell=1}^{b-1} \frac{b^{-k}}{(x+\ell-1)^2} \frac{1}{(1+b^{-k}(x+\ell-1)^{-1})(b+b^{-k}(x+\ell-1)^{-1})} \\
&= \sum_{k=0}^\infty \sum_{\ell=1}^{b-1} \frac{b^k}{(b^k(x+\ell-1)+1)(b^{k+1}(x+\ell-1)+1)} \\
&= \frac{1}{1-b} \sum_{\ell=1}^{b-1} \sum_{k=0}^\infty \frac{b^k}{b^k(x+\ell-1)+1} - \frac{b^{k+1}}{b^{k+1}(x+\ell-1)+1} \\
&= \frac{1}{1-b} \sum_{\ell=1}^{b-1}\left[\frac{1}{x+\ell} - \lim_{k\to\infty} \frac{b^k}{b^k(x+\ell-1)+1} \right] \\
&= \frac{1}{1-b} \sum_{\ell=1}^{b-1} \left[\frac{1}{x+\ell} - \frac{1}{x+\ell-1}\right] \\
&= \frac{1}{1-b}\left[\frac{1}{x+b-1}-\frac{1}{x}\right] = \frac{1}{1-b}\left[\frac{1-b}{x(x+b-1)}\right] = \frac{1}{x(x+b-1)}.
\end{align*}
\end{proof}

\begin{theorem}\label{thm:mn'recurrence}
The sequence of functions $m_n'(x) = \frac{\text{d}}{\text{d}x} m_n(x)$ is given by the recursive relationship
\begin{align}
m_0'(x) &= 1 \label{eq:m0'} \\
m_n'(x) &= \sum_{k=0}^\infty \sum_{\ell=1}^{b-1} b^{-k} (x+\ell-1)^{-2} m_{n-1}'(1+b^{-k}(x+\ell-1)^{-1}) && n \ge 1 \label{eq:mn'}.
\end{align}
for $1 \le x \le 2$.
\end{theorem}
\begin{proof}
Equation~\eqref{eq:m0'} follows immediately from \eqref{eq:m0}. Notice that \eqref{eq:mn'} is the result of differentiating both sides of \eqref{eq:mn}. In general, if $m_{n+1}'$ is bounded and continuous for some $n$, then the series on the right hand side of \eqref{eq:mn'} will converge uniformly on $(1,2)$. Thus the sum of the series will be bounded and continuous and will equal $m_{n+1}'$, so \eqref{eq:mn'} follows by induction, since $m_0'$ is clearly bounded and continuous.
\end{proof}

We will now prove a number of lemmas and theorems about the following classes of sequences of functions, to which $(m_n')_{n=0}^\infty$ belongs.

\begin{definition}\label{def:An*}
Let $f_0, f_1, \dots$ be a sequence of functions on $(1,2)$. We will say $(f_n)_{n=0}^\infty \in A^*$ if for all $x \in (1,2)$ and $n \ge 0$,
\begin{equation} \label{eq:frec}
f_{n+1}(x) = \sum_{k=0}^\infty \sum_{\ell=1}^{b-1} \frac{b^{-k}}{(x+\ell-1)^2} f_{n}\left(1+\frac{b^{-k}}{x+\ell-1}\right).
\end{equation}
Furthermore, we say that $(f_n)_{n=0}^\infty \in A^{**}$ if $(f_n)_{n=0}^\infty \in A^*$ and there exist constants $M, \mu > 0$ such that for all $x \in (1,2)$, we have $0 < f_0(x) < M$ and $|f_0'(x)| < \mu$.
\end{definition}

\begin{lemma}\label{lem:ranknsum=1}
\[
\sum^{(n)} \frac{b^{a_0 + \cdots + a_n}}{q_n(q_n + b^{a_n}q_{n-1})} = 1.
\]
\end{lemma}

\begin{proof}
Since the intervals of rank $n$ are disjoint and 
\[
\bigcup^{(n)} J_n\mat{a_1, & \dots, & a_n \\ c_1, & \dots, & c_n} = (1,2), 
\]
we have that
\[
\sum^{(n)} \M J_n\mat{a_1, & \dots, & a_n \\ c_1, & \dots, & c_n} = \M (1,2) = 1.
\]
Now notice that by Lemma~\ref{lem:endpoints} and Lemma~\ref{lem:pnqnm1-qnpnm1}, 
\begin{align*}
\M J_n\mat{a_1, & \dots, & a_n \\ c_1, & \dots, & c_n} &= \left| \frac{p_n}{q_n} - \frac{p_n + b^{a_n}p_{n-1}}{q_n + b^{a_n} q_{n-1}} \right| = \left|\frac{b^{a_n} (p_n q_{n-1} - q_n p_{n-1})}{q_n(q_n + b^{a_n} q_{n-1})}\right| \\
&= \left|\frac{(-1)^{n-1} b^{a_0 + \cdots + a_n}}{q_n(q_n + b^{a_n} q_{n-1})}\right| = \frac{b^{a_0+\cdots+a_n}}{q_n(q_n + b^{a_n}q_{n-1})},
\end{align*}
and thus
\[
\sum^{(n)} \frac{b^{a_0+\cdots+a_n}}{q_n(q_n + b^{a_n}q_{n-1})} = \sum^{(n)} \M J_n\mat{a_1, & \dots, & a_n \\ c_1, & \dots, & c_n} = 1.
\]
\end{proof}

\begin{lemma}\label{lem:fnf0}
If $(f_n)_{n=0}^\infty \in A^*$ then for $n \ge 0$,
\begin{equation}\label{eq:fnf0}
f_n(x) = \sum^{(n)} f_0\left(\frac{p_n + b^{a_n}p_{n-1}(x-1)}{q_n + b^{a_n} q_{n-1}(x-1)}\right) \frac{b^{\sum_{j=0}^n a_j}}{(q_n + b^{a_n} q_{n-1} (x-1))^2}.
\end{equation}
\end{lemma}

\begin{proof}
For $n = 0$, we just have a single interval, so
\begin{align*}
\sum^{(0)} &f_0\left(\frac{p_0 + b^{a_0} p_{-1} (x-1)}{q_0 + b^{a_0} q_{-1}(x-1)}\right) \frac{b^{a_0}}{(q_0 + b^{a_0}q_{-1}(x-1))^2} \\
&= f_0\left(\frac{1+(1)(1)(x-1)}{1+(1)(0)(x-1)}\right) \frac{1}{(1+(1)(0)(x-1))^2} = f_0(x).
\end{align*}
Now suppose \eqref{eq:fnf0} holds for $n$. Then
\begin{align*}
f&_{n+1}(x)\\
&= \sum_{k=0}^\infty \sum_{\ell=1}^{b-1} \frac{b^{-k}}{(x+\ell-1)^2} f_n\left(1+\frac{b^{-k}}{x+\ell-1}\right) \\
&=\sum_{k=0}^\infty \sum_{\ell = 1}^{b-1} \frac{b^{-k}}{(x+\ell-1)^2} \sum^{(n)} f_0\left(\frac{p_n + b^{a_n}p_{n-1}(1+\frac{b^{-k}}{x+\ell-1})}{q_n + b^{a_n}q_{n-1}(1+\frac{b^{-k}}{x+\ell-1})}\right) \frac{b^{\sum_{j=0}^n a_j}}{(q_n + b^{a_n}q_{n-1}(1+\frac{b^{-k}}{x+\ell-1}))^2} \\
&= \sum^{(n)} \sum_{k=0}^\infty \sum_{\ell=1}^{b-1} f_0\left(\frac{p_n b^k (x+\ell-1) + b^{a_n} p_{n-1}}{q_n b^k (x+\ell-1) + b^{a_n} q_{n-1}}\right)\frac{b^{\sum_{j=0}^n a_j} b^k}{(q_n b^k (x+\ell-1) + b^{a_n} q_{n-1})^2} \\
&= \sum^{(n)} \sum_{k=0}^\infty \sum_{\ell=1}^{b-1} f_0\left(\frac{\ell b^k p_n + b^{a_n} p_{n-1} + b^k p_n (x-1)}{\ell b^k q_n + b^{a_n} q_{n-1} + b^k q_n (x-1)}\right)\frac{b^{\sum_{j=0}^n a_j} b^k}{(\ell b^k q_n + b^{a_n} q_{n-1} + b^k q_n (x-1))^2} \\
&= \sum^{(n+1)} f_0\left(\frac{c_{n+1} b^{a_{n+1}} p_n + b^{a_n} p_{n-1} + b^{a_{n+1}} p_n (x-1)}{c_{n+1} b^{a_{n+1}} q_n + b^{a_n} q_{n-1} + b^{a_{n+1}} q_n (x-1)}\right)\frac{b^{\sum_{j=0}^n a_j} b^{a_{n+1}}}{(c_{n+1} b^{a_{n+1}} q_n + b^{a_n} q_{n-1} + b^{a_{n+1}} q_n (x-1))^2} \\
&= \sum^{(n+1)} f_0\left(\frac{p_{n+1} + b^{a_{n+1}} p_n (x-1)}{q_{n+1} + b^{a_{n+1}} q_n (x-1)}\right)\frac{b^{\sum_{j=0}^{n+1} a_j}}{(q_{n+1} + b^{a_{n+1}} q_n (x-1))^2},
\end{align*}
so the result follows by induction.
\end{proof}

\begin{lemma}\label{lem:fn'bound}
If $(f_n)_{n=0}^\infty \in A^{**}$, then for $n \ge 0$,
\[
|f_n'(x)| \le \frac{3\mu}{2^{n/2}} + 4M.
\]
\end{lemma}
\begin{proof}
Differentiate \eqref{eq:fnf0} termwise, letting $u = \frac{p_n + b^{a_n} p_{n-1}(x-1)}{q_n + b^{a_n} q_{n-1}(x-1)}$, to get
\begin{equation}\label{eq:fn'}
f_n'(x) = \sum^{(n)} f_0'(u) \frac{(-1)^{n-1} b^{2\sum_{j=0}^n a_j}}{(q_n + b^{a_n} q_{n-1}(x-1))^4} - 2\sum^{(n)} f_0(u) \frac{b^{a_n} q_{n-1} b^{\sum_{j=0}^n a_j}}{(q_n + b^{a_n} q_{n-1} (x-1))^3}.
\end{equation}
The validity of termwise differentiation follows from the uniform convergence of both sums on the right hand side for $1 \le x \le 2$. Notice that
\begin{equation}\label{eq:leftbound}
\left|\frac{(-1)^{n-1} b^{2\sum_{j=0}^n a_j}}{(q_n + b^{a_n} q_{n-1}(x-1))^4}\right| \le \frac{q_n b^{\sum_{j=0}^n a_j}}{q_n^4} \le \frac{2 b^{\sum_{j=0}^n a_j}}{2^{(n-1)/2} q_n (q_n + b^{a_n} q_{n-1})}
\end{equation}
by Lemma~\ref{lem:qnbounds}, Lemma~\ref{lem:pqlemma}, and the fact that $q_n + b^{a_{n-1}} q_{n-1} \le 2 q_n$. Additionally,
\begin{equation}\label{eq:rightbound}
\frac{b^{a_n} q_{n-1} b^{\sum_{j=0}^n a_j}}{(q_n + b^{a_n} q_{n-1} (x-1))^3} \le \frac{b^{a_n} q_{n-1} b^{\sum_{j=0}^n a_j}}{q_n^3} \le \frac{2 b^{\sum_{j=0}^n a_j}}{q_n(q_n + b^{a_n} q_{n-1})}
\end{equation}
since $b^{a_n} q_{n-1} \le q_n$ and $q_n + b^{a_n} q_{n-1} \le 2q_n$. Since $(f_n)_{n=0}^\infty \in A^{**}$, we have by Definition~\ref{def:An*} that $|f_0(x)| < M$ and $|f_0'(x)| < \mu$ for all $x \in (1,2)$. Thus we have by \eqref{eq:fn'}, \eqref{eq:leftbound}, \eqref{eq:rightbound}, and Lemma~\ref{lem:ranknsum=1},
\begin{align*}
|f_n'(x)| &\le \sum^{(n)} |f_0'(u)| \left|\frac{(-1)^{n-1} b^{2\sum_{j=0}^n a_j}}{(q_n + b^{a_n} q_{n-1} (x-1))^4}\right| + 2 \sum^{(n)} |f_0(u)| \left|\frac{b^{a_n} q_{n-1} b^{\sum_{j=0}^n a_j}}{(q_n + b^{a_n} q_{n-1} (x-1))^3}\right| \\
&\le \frac{2\mu}{2^{(n-1)/2}} \sum^{(n)} \frac{b^{\sum_{j=0}^n a_j}}{q_n(q_n + b^{a_n} q_{n-1})} + 4M \sum^{(n)} \frac{b^{\sum_{j=0}^n a_j}}{q_n(q_n + b^{a_n} q_{n-1})} \\
&= \frac{2\mu}{2^{(n-1)/2}} + 4M = \frac{2 \sqrt 2 \mu}{2^{n/2}} + 4M < \frac{3\mu}{2^{n/2}} + 4M.
\end{align*}
\end{proof}

\begin{lemma}\label{lem:fnbound->fn+1bound}
If $(f_n)_{n=0}^\infty \in A^*$ and for some constants $T > t > 0$,
\[
\frac{t}{x(x+b-1)} < f_n(x) < \frac{T}{x(x+b-1)} \hspace{1cm} \forall x \in (1,2),
\]
then
\[
\frac{t}{x(x+b-1)} < f_{n+1}(x) < \frac{T}{x(x+b-1)} \hspace{1cm} \forall x \in (1,2).
\]
\end{lemma}
\begin{proof}
By \eqref{eq:frec} and Lemma~\ref{lem:1/x(x+b-1)sum} we have
\begin{align*}
f_{n+1}(x) &= \sum_{k=0}^\infty \sum_{\ell = 1}^{b-1} \frac{b^{-k}}{(x+\ell-1)^2} f_n\left(1+\frac{b^{-k}}{x+\ell-1}\right) \\
&> \sum_{k=0}^\infty \sum_{\ell=1}^{b-1} \frac{b^{-k}}{(x+\ell-1)^2} \frac{t}{(1+(x+\ell-1)^{-1}b^{-k})(b+(x+\ell-1)^{-1}b^{-k})} \\
&= \frac{t}{x(x+b-1)},
\end{align*}
and a similar derivation shows
\[
f_{n+1}(x) < \frac{T}{x(x+b-1)},
\]
from which the result follows.
\end{proof}

\begin{lemma}\label{lem:intfnintf0}
If $(f_n)_{n=0}^\infty \in A^*$ then for all $n \ge 0$,
\[
\int_1^2 f_n(z) \, \text{d}z = \int_1^2 f_0(z) \, \text{d}z.
\]
\end{lemma}
\begin{proof}
Notice that 
\[
(1,2] = \bigcup_{k=0}^\infty \bigcup_{\ell=1}^{b-1} (1+(\ell+1)^{-1}b^{-k}, 1+\ell^{-1} b^{-k}],
\]
where the intervals are pairwise disjoint. We then have, by \eqref{eq:frec} that
\begin{align*}
\int_1^2 f_n(z)\d z &= \int_1^2 \sum_{k=0}^\infty \sum_{\ell = 1}^{b-1} \frac{b^{-k}}{(z+\ell-1)^2} f_{n-1} \left(1+\frac{b^{-k}}{z+\ell-1}\right) \d z \\
&= \sum_{k=0}^\infty \sum_{\ell = 1}^{b-1} \int_{1+\ell^{-1}b^{-k}}^{1+(\ell+1)^{-1}b^{-k}} -f_{n-1}(u)\d u \\
&= \sum_{k=0}^\infty \sum_{\ell = 1}^{b-1} \int_{1+(\ell+1)^{-1}b^{-k}}^{1+\ell^{-1}b^{-k}} f_{n-1}(u)\d u \\
&= \int_1^2 f_{n-1}(u)\,du,
\end{align*}
from which the result follows by induction.
\end{proof}

\begin{lemma}\label{lem:gcompressing}
Suppose $(f_n)_{n=0}^\infty \in A^{**}$ and there are constants $g, G > 0$ such that for all $x \in (1,2)$,
\begin{equation}\label{eq:f0bound}
\frac{g}{x(x+b-1)} < f_0(x) < \frac{G}{x(x+b-1)}.
\end{equation}
Then there exist $n \in \N$ and $g_1, G_1 > 0$ such that
\begin{equation}\label{eq:resultfnbound}
\frac{g_1}{x(x+b-1)} < f_n(x) < \frac{G_1}{x(x+b-1)},
\end{equation}
\begin{equation}\label{eq:resultgorder}
g < g_1 < G_1 < G,
\end{equation}
and
\begin{equation}\label{eq:resultgbound}
G_1 - g_1 < (G-g)\delta + 2^{-n/2}(\mu + G),
\end{equation}
where $\delta = 1 - \frac{1}{4(b-1)} \log \frac{2b}{b+1}$.
\end{lemma}

\begin{proof}
First define
\begin{equation}\label{eq:phipsidef}
\varphi_n(x) = f_n(x) - \frac{g}{x(x+b-1)}, \hspace{15mm} \psi_n(x) = \frac{G}{x(x+b-1)} - f_n(x),
\end{equation}
which are both positive functions by \eqref{eq:f0bound} and Lemma~\ref{lem:fnbound->fn+1bound}. Notice that for the functions $h(x) = \frac{g}{x(x+b-1)}$ and $H(x) = \frac{G}{x(x+b-1)}$, we have by Lemma~\ref{lem:1/x(x+b-1)sum} that
\begin{align*}
h(x) &= \frac{g}{x(x+b-1)} = \sum_{k=0}^\infty \sum_{\ell=1}^{b-1} \frac{b^{-k}}{(x+\ell-1)^2} \frac{g}{(1+b^{-k}(x+\ell-1)^{-1})(b+b^{-k}(x+\ell-1)^{-1})} \\
&= \sum_{k=0}^\infty \sum_{\ell=1}^{b-1} \frac{b^{-k}}{(x+\ell-1)^2} h\left(1+\frac{b^{-k}}{x+\ell-1}\right),
\end{align*}
and similarly
\[
H(x) = \sum_{k=0}^\infty \sum_{\ell=1}^{b-1} \frac{b^{-k}}{(x+\ell-1)^2} H\left(1+\frac{b^{-k}}{x+\ell-1}\right).
\]
Thus for $n \ge 1$, we have
\begin{align*}
\varphi_{n+1}(x) &= f_{n+1}(x) - h(x) \\
&= \sum_{k=0}^\infty \sum_{\ell=1}^{b-1} \frac{b^{-k}}{(x+\ell-1)^2} f_n\left(1+\frac{b^{-k}}{x+\ell-1}\right) - \sum_{k=0}^\infty \sum_{\ell=1}^{b-1} \frac{b^{-k}}{(x+\ell-1)^2} h\left(1+\frac{b^{-k}}{x+\ell-1}\right) \\
&= \sum_{k=0}^\infty \sum_{\ell=1}^{b-1} \frac{b^{-k}}{(x+\ell-1)^2} \varphi_n \left(1+\frac{b^{-k}}{x+\ell-1}\right),
\end{align*}
and similarly
\[
\psi_{n+1}(x) = \sum_{k=0}^\infty  \sum_{\ell=1}^{b-1} \frac{b^{-k}}{(x+\ell-1)^2} \psi_n \left(1+\frac{b^{-k}}{x+\ell-1}\right). 
\]
Thus $(\varphi_n)_{n=0}^\infty, (\psi_n)_{n=0}^\infty \in A^*$, so by Lemma~\ref{lem:fnf0}, setting $u = \frac{p_n + b^{a_n} p_{n-1}(x-1)}{q_n + b^{a_n} q_{n-1}(x-1)}$, we get
\begin{equation}\label{eq:philowerbound}
\varphi_n(x) = \sum^{(n)} \varphi_0(u) \frac{b^{\sum_{j=0}^n a_j}}{(q_n + b^{a_n} q_{n-1}(x-1))^2} \ge \frac14 \sum^{(n)} \varphi_0(u) \frac{b^{\sum_{j=0}^n a_j}}{q_n^2},
\end{equation}
and similarly
\begin{equation}\label{eq:psilowerbound}
\psi_n(x) = \ge \frac14 \sum^{(n)} \psi_0(u) \frac{b^{\sum_{j=0}^n a_j}}{q_n^2},
\end{equation}
since
\[
q_n + b^{a_n} q_{n-1} (x-1) \le 2q_n \hspace{1cm} \forall x \in (1,2).
\]
On the other hand, the mean value theorem gives
\begin{equation}\label{eq:phiintmvt}
\frac14 \int_1^2 \varphi_0(z) \d z = \frac14 \sum^{(n)} \varphi_0(u_1) \frac{b^{\sum_{j=0}^n a_j}}{q_n(q_n + b^{a_n}q_{n-1})},
\end{equation}
and
\begin{equation}\label{eq:psiintmvt}
\frac14 \int_1^2 \psi_0(z) \d z = \frac14 \sum^{(n)} \psi_0(u_2) \frac{b^{\sum_{j=0}^n a_j}}{q_n(q_n + b^{a_n}q_{n-1})},
\end{equation}
where for each interval $\left(\frac{p_n}{q_n}, \frac{p_n + b^{a_n}p_{n-1}}{q_n + b^{a_n} q_{n-1}}\right)$ of rank $n$, $u_1$ and $u_2$ are points in the interval and the length of the interval is $ \frac{b^{\sum_{j=0}^n a_j}}{q_n(q_n + b^{a_n}q_{n-1})}$.
From \eqref{eq:philowerbound} and \eqref{eq:phiintmvt} we then get
\begin{equation}\label{eq:phi-intbound}
\varphi_n(x) - \frac14 \int_1^2 \varphi_0(z) \d z \ge \frac14 \sum^{(n)} \left[\varphi_0(u) - \varphi_0(u_1)\right] \frac{b^{\sum_{j=0}^n a_j}}{q_n(q_n+b^{a_n}q_{n-1})},
\end{equation}
and from \eqref{eq:psilowerbound} and \eqref{eq:psiintmvt}, we get
\begin{equation}\label{eq:psi-intbound}
\psi_n(x) - \frac14 \int_1^2 \psi_0(z) \d z \ge \frac14 \sum^{(n)} \left[\psi_0(u) - \psi_0(u_2)\right] \frac{b^{\sum_{j=0}^n a_j}}{q_n(q_n+b^{a_n}q_{n-1})}.
\end{equation}
Now for $1 \le x \le 2$, we have $|\varphi_0'(x)| \le |f_0'(x)| + g \le \mu + g$ and $|\psi_0'(x)| \le |f_0'(x)| + G \le \mu + G$, so it follows by Lemma~\ref{lem:qnbounds} that
\begin{equation}\label{eq:deltaphibound}
|\varphi_0(u_1) - \varphi_0(u)| \le (\mu + g) |u_1 - u| \le (\mu + g) \frac{b^{\sum_{j=0}^n a_j}}{q_n(q_n+b^{a_n}q_{n-1})} \le \frac{\mu + g}{q_n} \le 2 \frac{\mu + g}{2^{n/2}},
\end{equation}
and similarly
\begin{equation}\label{eq:deltapsibound}
|\psi_0(u_2) - \psi_0(u)| \le 2 \frac{\mu + G}{2^{n/2}}.
\end{equation}
Then by Lemma~\ref{lem:ranknsum=1}, \eqref{eq:phi-intbound} and \eqref{eq:deltaphibound} give
\begin{align*}
\varphi_n(x) &> \frac14 \int_1^2 \varphi_0(z) \d z - \frac14 \sum^{(n)} \left[\varphi_0(u_1) - \varphi_0(u)\right] \frac{b^{\sum_{j=0}^n a_j}}{q_n(q_n + b^{a_n}q_{n-1})} \\
&\ge \ell - \frac14 \sum^{(n)} |\varphi_0(u_1)-\varphi_0(u)| \frac{b^{\sum_{j=0}^n a_j}}{q_n(q_n+b^{a_n}q_{n-1})} \\
&\ge \ell - \frac12 \frac{\mu + g}{2^{n/2}} \sum^{(n)} \frac{b^{\sum_{j=0}^n a_j}}{q_n(q_n+b^{a_n}q_{n-1})} = \ell - \frac12 \frac{\mu + g}{2^{n/2}} = \ell - \frac{\mu + g}{2^{n/2+1}},
\end{align*}
where $\ell = \frac14 \int_1^2 \varphi_0(z) \d z$. Similarly, \eqref{eq:psi-intbound} and \eqref{eq:deltapsibound} give
\[
\psi_n(x) \ge L - \frac{G + \mu}{2^{n/2+1}},
\]
where $L = \frac14 \int_1^2 \psi_0(z) \d z$. Now by \eqref{eq:phipsidef}, we have
\begin{align}
f_n(x) &= \frac{g}{x(x+b-1)} + \varphi_n(x) > \frac{g}{x(x+b-1)} + \ell - \frac{\mu + g}{2^{n/2+1}} \notag \\ 
&> \frac{g + \ell - 2^{-n/2-1}(\mu + g)}{x(x+b-1)} = \frac{g_1}{x(x+b-1)}, \label{eq:fnlowerbound}
\end{align}
where $g_1 = g + \ell - 2^{-n/2}(\mu + g)$, and
\begin{align}
f_n(x) &= \frac{G}{x(x+b-1)} - \psi_n(x) < \frac{G}{x(x+b-1)} - L + \frac{\mu + G}{2^{n/2+1}} \notag \\
&< \frac{G - \ell + 2^{-n/2-1}(\mu + G)}{x(x+b-1)} = \frac{G_1}{x(x+b-1)}, \label{eq:fnupperbound}
\end{align}
where $G_1 = G - L + 2^{-n/2-1}(\mu + G)$. Now since $\ell, L > 0$, we can choose $n$ sufficiently large so that $2^{-n/2-1}(\mu + g) < \ell$ and $2^{-n/2-1}(\mu + G) < L$, so that we get
\begin{equation}\label{eq:gorder}
g < g_1 < G_1 < G.
\end{equation}
Thus by \eqref{eq:fnlowerbound}, \eqref{eq:fnupperbound}, and \eqref{eq:gorder}, we have found $g_1, G_1$, and $n$ that satisfy \eqref{eq:resultfnbound} and \eqref{eq:resultgorder}. Notice that we also have
\begin{equation}\label{eq:gbound1}
G_1 - g_1 = G-g-(L+\ell)+2^{-n/2-1}(2\mu + g + G) < G-g -(L+\ell)+ 2^{-n/2}(\mu + G).
\end{equation}
Now since
\[
\ell + L = \frac14 \int_1^2 \frac{G-g}{x(x+b-1)} \d x = (G-g) \frac{1}{4(b-1)} \log \frac{2b}{b+1},
\]
\eqref{eq:gbound1} becomes
\[
G_1 - g_1 < \left(1-\frac{1}{4(b-1)}\log\frac{2b}{b+1}\right)(G-g) + 2^{-n/2}(\mu+G) = \delta(G-g) + 2^{-n/2}(\mu + G),
\]
so we see that $g_1, G_1$, and $n$ also satisfy \eqref{eq:resultgbound}, completing the proof.
\end{proof}

\begin{remark}\label{rem:samen}
Notice that the value of $n$ chosen depends only on the values of $\mu$ and $G$, and that if we make $0 < \mu_1 < \mu$ and $0 < G_1 < G$, the value of $n$ chosen for $\mu$ and $G$ will also work for $\mu_1$ and $G_1$. In other words, we can make $\mu$ and $G$ smaller without having to increase $n$. This will be useful in the proof of the Theorem~\ref{thm:mainthm}.
\end{remark}

\begin{theorem}\label{thm:mainthm}
Suppose $(f_n)_{n=0}^\infty \in A^{**}$. Then there exist constants $\lambda, A > 0$ such that for all $n \ge 0$ and $x \in (1,2)$,
\[
\left|f_n(x) - \frac{a}{x(x+b-1)} \right| < A e^{-\lambda \sqrt{n}},
\]
where
\[
a = \frac{b-1}{\log \frac{2b}{b+1}} \int_1^2 f_0(z) \d z.
\]
\end{theorem}

\begin{proof}
By assumption, $f_0$ is differentiable and continuous on $[1,2]$, so there is some constant $m > 0$ such that $m < f_0(x) < M$ for all $x \in [1,2]$. Then since $\frac1{2(b+1)} < \frac{1}{x(x+b-1)} < \frac{1}{b}$ for all $x \in (1,2)$, we have
\[
\frac{bm}{x(x+b-1)} < f_0(x) < \frac{2(b+1)M}{x(x+b-1)} \hspace{1cm} \forall x \in (1,2).
\]
Thus let $g = bm$ and $G = 2(b+1)M$ and apply Lemma~\ref{lem:gcompressing} to $f_0, g$, and $G$, to get $g_1, G_1$, and $n$ such that
\[
\frac{g_1}{x(x+b-1)} < f_n(x) < \frac{G_1}{x(x+b-1)} \hspace{1cm} \forall x \in (1,2),
\]\[
g < g_1 < G_1 < G,
\]
and
\[
G_1-g_1 < \delta(G-g) +  2^{-n/2}(\mu + G).
\]
By Lemma~\ref{lem:fn'bound}, $|f_n'(x)| < \mu_1 = \frac{3\mu}{2^{n/2}} + 4M$, and we can arrange to have $\mu_1 < \mu$ by making $\mu$ and $n$ sufficiently large. (By Remark~\ref{rem:samen}, the results above are still valid for the new values of $\mu$ and $n$.) We can then apply Lemma~\ref{lem:gcompressing} again with $f_n$, $g_1$, and $G_1$ instead of $f_0$, $g$, and $G$. This gives us new constants $g_2$ and $G_2$ such that (again due to Remark~\ref{rem:samen}), 
\[
\frac{g_2}{x(x+b-1)} < f_{2n}(x) < \frac{G_2}{x(x+b-1)} \hspace{1cm} \forall x \in (1,2),
\]\[
g < g_1 < g_2 < G_2 < G_1 < G,
\]
and
\[
G_2-g_2 < \delta(G_1-g_1) + 2^{-2n/2}(\mu_1 + G_1).
\]
Repeating this in a similar fashion gives, in general, constants $g_r, G_r$ such that
\[
\frac{g_r}{x(x+b-1)} < f_{nr}(x) < \frac{G_r}{x(x+b-1)} \hspace{1cm} \forall x \in (1,2),
\]\[
g < g_1 < \cdots < g_{r-1} < g_r < G_r < G_{r-1} < \cdots < G_1 < G,
\]
and
\[
G_r-g_r < \delta(G_{r-1}-g_{r-1}) +  2^{-rn/2}(\mu_{r-1} + G_{r-1}),
\]
where $\mu_{r-1}$ is a constant such that $|f_{n(r-1)}'(x)| < \mu_{r-1}$ for all $x \in (1,2)$. By Lemma~\ref{lem:fn'bound}, we can take $\mu_r = \frac{3\mu}{2^{nr/2}} + 4M$, and then can choose $r_0 \in \N$ such that $\mu_{r-1} < 5M$ for all $r \ge r_0$. Then since $G_r < G = 2(b+1)M$, we have
\begin{equation}\label{eq:Gr-gr}
G_r - g_r < \delta(G_{r-1} - g_{r-1}) + (2b+7)M2^{-nr/2} = \delta(G_{r-1} - g_{r-1}) + M_1 2^{-nr/2},
\end{equation}
for all $r \ge r_0$ where $M_1 = (2b+7)M$. We now claim that for all $k \ge 0$,
\begin{equation}\label{eq:Grrecursion}
G_{r_0 + k} - g_{r_0 + k} < \delta^k (G-g) + \delta^k M_1 2^{-nr_0/2} \sum_{j=0}^k (2^{-nj/2} \delta^{-j}).
\end{equation}
For $k=0$, from \eqref{eq:Gr-gr}, we have
\begin{align*}
G_{r_0}-g_{r_0} &< \delta (G_{r_0-1} - g_{r_0-1}) + M_1 2^{-nr_0/2} < (G-g) + M_1 2^{-nr_0/2} \\
&= \delta^0 (G-g) + M_1 \delta^0 2^{-nr_0/2} \sum_{j=0}^0 2^{-nj/2} \delta^{-j}.
\end{align*}
Now suppose \eqref{eq:Grrecursion} holds for $k$. Notice that
\[
M_1 2^{-n(r_0+k+1)/2} = M_1 \delta^{k+1} 2^{-nr_0/2} 2^{-n(k+1)/2} \delta^{-(k+1)},
\]
so by \eqref{eq:Gr-gr},
\begin{align*}
G_{r_0+k+1}-g_{r_0+k+1}
&< \delta(G_{r_0+k} - g_{r_0+k}) + M_1 2^{-n(r_0+k+1)/2} \\
&< \delta\left(\delta^k(G-g) + M_1 \delta^k 2^{-nr_0/2} \sum_{j=0}^k (2^{-nj/2} \delta^{-j})\right) + M_1 2^{-n(r_0+k+1)/2} \\
&= \delta^{k+1}(G-g) + M_1 \delta^{k+1} 2^{-nr_0/2} \left(\sum_{j=0}^k (2^{-nj/2} \delta^{-j}) + 2^{-n(k+1)/2}\delta^{-(k+1)}\right) \\
&= \delta^{k+1}(G-g) + M_1 \delta^{k+1} 2^{-nr_0/2} \sum_{j=0}^{k+1} 2^{-nj/2} \delta^{-j},
\end{align*}
so \eqref{eq:Grrecursion} follows by induction.

Now notice that for $k > 0$,
\[
\sum_{j=0}^k 2^{-nj/2} \delta^{-j} < \sum_{j=0}^\infty (2^{n/2} \delta)^{-j} \le \sum_{j=0}^\infty (2^{1/2} \delta)^{-j} = \gamma < \infty,
\]
since $2^{1/2}\delta = \sqrt{2}(1-\frac{1}{4(b-1)} \log \frac{2b}{b+1}) > \sqrt{2}(1-\frac14 \log 2) > 1$. \eqref{eq:Grrecursion} then becomes 
\[
G_{r_0+k} - g_{r_0+k} < \delta^k(G-g+ M_1 2^{-nr_0/2} \gamma) = \delta^k c,
\]
where $c>0$ is a constant. Then for $r \ge r_0$, we have
\[
G_r - g_r < \delta^{r-r_0} c = \delta^r (\delta^{-r_0} c) = \delta^r d,
\]
where again, $d > 0$ is a constant. Finally, since $\delta < 1$, we can choose $B, \lambda > 0$ such that $G_r - g_r < B e^{-\lambda r}$. Thus there is clearly some common limit
\[
a = \lim_{r\to\infty} g_r = \lim_{r\to\infty} G_r,
\] 
and we have (setting $r = n$) that
\begin{equation}\label{eq:fn2bound}
\left|f_{n^2}(x) - \frac{a}{x(x+b-1)}\right| < Be^{-\lambda n} \hspace{1cm} \forall x \in (1,2).
\end{equation}
Thus we have
\[
\lim_{n\to\infty} \int_1^2 f_{n^2}(z) \d z = \int_1^2 \frac{a}{x(x+b-1)} = \frac{a}{b-1} \log \frac{2b}{b+1},
\]
so by Lemma~\ref{lem:intfnintf0}, $\int_1^2 f_0(z) \d z = \frac{a}{b-1} \log \frac{2b}{b+1}$ and thus 
\[
a = \frac{b-1}{\log \frac{2b}{b+1}} \int_1^2 f_0(z) \d z.
\] 
Now for arbitrary $N \ge r_0^2$, we can choose $n \ge r_0$ such that $n^2 \le N < (n+1)^2$. We then have, by \eqref{eq:fn2bound},
\[
\frac{a - 2(b+1)Be^{-\lambda n}}{x(x+b-1)} < \frac{a}{x(x+b-1)} - Be^{-\lambda n} < f_{n^2}(x) < \frac{a}{x(x+b-1)} + Be^{-\lambda n} < \frac{a+2(b+1)Be^{-\lambda n}}{x(x+b-1)},
\]
for all $x \in (1,2)$. Then by Lemma~\ref{lem:fnbound->fn+1bound},
\[
\frac{a-2(b+1)Be^{-\lambda n}}{x(x+b-1)} < f_N(x) < \frac{a+2(b+1)e^{-\lambda n}}{x(x+b-1)},
\]
so
\[
\left|f_N(x) - \frac{a}{x(x+b-1)} \right| < \frac{2(b+1)Be^{-\lambda n}}{x(x+b-1)} < 2(b+1)Be^{-\lambda n} = 2(b+1)Be^\lambda e^{-\lambda(n+1)} < A' e^{-\lambda \sqrt{N}},
\]
where $A' = 2(b+1)Be^{\lambda}$ is a constant. Now for $0 \le N < r_0^2$, note that each $f_N$ is continuous (since $f_0$ is differentiable and thus continuous and $f_{N+1}$ is an absolutely convergent sum of continuous transformations of $f_N$). Thus we can choose $A_0, A_1, \dots, A_{r_0^2-1}$ such that for $0 \le N \le r_0^2 - 1$
\[
\left|f_N(x) - \frac{a}{x(x+b-1)}\right| < A_n e^{-\lambda \sqrt{N}} \hspace{1cm} \forall x\in (1,2), \hspace{5mm} \forall N \in \{0,1,\dots,r_0^2-1\}
\]
for all $x \in (1,2)$. Finally, take $A = \max\{A_0, A_1, \dots, A_{r_0-1}, A'\}$, so we have
\[
\left|f_N(x) - \frac{a}{x(x+b-1)} \right| < A e^{-\lambda \sqrt{N}} \hspace{1cm} \forall x \in (1,2) \hspace{5mm} \forall N \in \Z_{\ge 0},
\]
proving the theorem.
\end{proof}

\begin{corollary}\label{cor:mn'dist}
There exist constants $\lambda, A > 0$ such that for all $n \ge 0$ and $x \in (1,2)$,
\[
\left|m_n'(x) - \frac{a}{x(x+b-1)}\right| < A e^{-\lambda \sqrt{n}},
\]
where 
\[
a = \frac{b-1}{\log \frac{2b}{b+1}}.
\]
\end{corollary}

\begin{proof}
By Theorem~\ref{thm:mn'recurrence}, $(m_n')_{n=0}^\infty \in A^{**}$. Then Theorem~\ref{thm:mainthm} gives constants $A, \lambda > 0$ such that
\[
\left|m_n'(x) - \frac{a}{x(x+b-1)}\right| < A e^{-\lambda \sqrt{n}},
\]
where
\[
a = \frac{b-1}{\log\frac{2b}{b+1}} \int_1^2 m_0'(z) \d z = \frac{b-1}{\log\frac{2b}{b+1}} \int_1^2 1 \d z = \frac{b-1}{\log\frac{2b}{b+1}},
\]
proving the corollary.
\end{proof}

Our main goals, Theorems~\ref{thm:mnlimbd} and~\ref{thm:MDnbound}, follow easily from Corollary~\ref{cor:mn'dist}.

\restatedistthm

\begin{proof}
First note that since $m_n(1) = 0$ for all $n$, so by the Fundamental Theorem of Calculus,
\[
m_n(x) = m_n(1) + \int_1^x m_n'(z)\d z = \int_1^x m_n'(z) \d z.
\]
Thus
\begin{equation}\label{eq:mnlimequiv}
\int_1^x m_n'(z)\d z - \frac{b-1}{\log \frac{2b}{b+1}} \int_1^x \frac{1}{z(z+b-1)} \d z = m_n(x) - \frac{b-1}{\log \frac{2b}{b+1}} \frac{\log \frac{bx}{x+b-1}}{b-1} = m_n(x) - \frac{\log \frac{bx}{x+b-1}}{\log \frac{2b}{b+1}}.
\end{equation}
Then by Theorem~\ref{thm:mainthm}, we have
\begin{align*}
\left| m_n(x) - \frac{\log \frac{bx}{x+b-1}}{\log \frac{2b}{b+1}} \right| &= \left|\int_1^x m_n'(z)  - \frac{b-1}{\log \frac{2b}{b+1}} \frac{1}{z(z+b-1)} \d z \right| \\
&\le \int_1^x \left|m_n'(z) - \frac{b-1}{\log \frac{2b}{b+1}} \frac{1}{z(z+b-1)}\right| \d z < \int_1^x A e^{-\lambda \sqrt{n}} \\
&= (x-1) Ae^{-\lambda \sqrt{n}} < A e^{-\lambda \sqrt{n}}.
\end{align*}
\end{proof}

\restateMDnboundthm

\begin{proof}

By Theorem~\ref{thm:genprobdist},
\[
\M D_n(k,\ell) = m_{n-1}(1+\ell^{-1}b^{-k}) - m_{n-1} (1+(\ell+1)^{-1}b^{-k}) = \int_{1+(\ell+1)^{-1}b^{-k}}^{1+\ell^{-1}b^{-k}} m_{n-1}'(z) \d z.
\]
Then by Corollary~\ref{cor:mn'dist}, it follows that there is are constants $A, \lambda > 0$ such that
\begin{align*}
\Bigg|\int_{1+(\ell+1)^{-1}b^{-k}}^{1+\ell^{-1}b^{-k}} &m_{n-1}'(z) \d z - \frac{b-1}{\log \frac{2b}{b+1}} \int_{1+(\ell+1)^{-1}b^{-k}}^{1+\ell^{-1}b^{-k}} \frac{1}{z(z+b-1)} \d z \Bigg| \\
&\le \int_{1+(\ell+1)^{-1}b^{-k}}^{1+\ell^{-1}b^{-k}} \left| m_{n-1}'(z) - \frac{b-1}{\log \frac{2b}{b+1}} \frac{1}{z(z+b-1)} \right| \d z \\
&< \int_{1+(\ell+1)^{-1}b^{-k}}^{1+\ell^{-1}b^{-k}} Ae^{-\lambda \sqrt{n-1}} \d z = (\ell^{-1}b^{-k} - (\ell+1)^{-1}b^{-k}) Ae^{-\lambda \sqrt{n-1}} = \frac{A e^{-\lambda \sqrt{n-1}}}{\ell(\ell+1)b^k}.
\end{align*}
Finally, since
\begin{align*}
\frac{b-1}{\log \frac{2b}{b+1}} \int_{1+(\ell+1)^{-1}b^{-k}}^{1+\ell^{-1}b^{-k}} \frac{1}{z(z+b-1)} \d z &= \frac{b-1}{\log \frac{2b}{b+1}} \frac{\log \frac{(1+\ell^{-1}b^{-k})(1+(\ell+1)^{-1}b^{-(k+1)})}{(1+\ell^{-1}b^{-(k+1)})(1+(\ell+1)^{-1}b^{-k})}}{b-1} \\
&= \frac{\log \frac{(\ell b^k + 1)((\ell +1)b^{k+1} + 1)}{(\ell b^{k+1} + 1)((\ell +1)b^k + 1)}}{\log \frac{2b}{b+1}},
\end{align*}
we have
\begin{align*}
\left| \M D_n (k,\ell) - \frac{\log \frac{(\ell b^k + 1)((\ell +1)b^{k+1} + 1)}{(\ell b^{k+1} + 1)((\ell +1)b^k + 1)}}{\log \frac{2b}{b+1}} \right| &= \left|\int_{1+(\ell+1)^{-1}b^{-k}}^{1+\ell^{-1}b^{-k}} m_{n-1}'(z) - \frac{b-1}{\log\frac{2b}{b+1}} \frac{1}{z(z+b-1)} \d z\right| \\
&< \frac{A e^{-\lambda \sqrt{n-1}}}{\ell(\ell+1)b^k}.
\end{align*}
\end{proof}

\addcontentsline{toc}{section}{Appendix A: Proof of the Type III Logarithmic Khinchine Constant}
\section*{Appendix B: Proof of the Type III Logarithmic Khinchine Constant}

This appendix is devoted to proving Theorems~\ref{thm:khinchinedist} and~\ref{thm:logkhinchine}, restated below. Note that the proofs in this appendix rely on certain results from Appendix A.

\restatekhinchinedistthm

\restatekhinchinethm

\begin{definition}\label{def:E}
Let $n \in \N$, $j_1, j_2, \dots, j_n \in \N$ be distinct, $k_1, k_2, \dots, k_n \in \Z_{\ge 0}$, and $\ell_1, \ell_2, \dots, \ell_n \in \{1,2,\dots,b-1\}$. Define
\[
E \mat{j_1, & j_2, & \dots, & j_n \\ k_1, & k_2, & \dots, & k_n \\ \ell_1, & \ell_2, & \dots, & \ell_n} = \left\{ \alpha \in (1,2) : \begin{array}{cccc} a_{j_1} = k_1, & a_{j_2} = k_2, & \dots, & a_{j_n} = k_n \\ c_{j_1} = \ell_1, & c_{j_2} = \ell_2, & \dots, & c_{j_n} = \ell_n \end{array} \right\}.
\]
\end{definition}

\begin{remark}\label{rem:Eassumptions}
We will always assume that $j_1 < j_2 < \cdots < j_n$, in which case $E \mat{j_1, & \dots, & j_n \\ k_1, & \dots, & k_n \\ \ell_1, & \dots, & \ell_n}$ is a countable union of intervals of rank $j_n$.
\end{remark}

\begin{theorem}\label{thm:Eratiobound}
There exist constants $A, \lambda > 0$ such that for arbitrary $m \in \N$, $j_1 < \dots < j_m < j \in \N$, $k_1, \dots, k_m, k \in \Z_{\ge 0}$, and $\ell_1, \dots, \ell_m, \ell \in \{1,\dots,b-1\}$, we have
\[
\left| \frac{\M E \mat{ j_1, & \dots, & j_m, & j \\ k_1, & \dots, & k_m, & k \\ \ell_1, & \dots, & \ell_m, & \ell}}{\M E \mat{ j_1, & \dots, & j_m \\ k_1, & \dots, & k_m \\ \ell_1, & \dots, & \ell_m}} - \frac{\log \frac{(1+\ell^{-1}b^{-k})(b+(\ell+1)^{-1}b^{-k})}{(b + \ell^{-1}b^{-k})(1+(\ell+1)^{-1}b^{-k})}}{\log\frac{2b}{b+1}} \right| < \frac{Ae^{-\lambda \sqrt{j-j_m-1}}}{\ell (\ell+1)b^{k}}.
\]
\end{theorem}

\begin{proof}
First fix some interval $J = J_n\mat{k_1, & \dots, & k_m \\ \ell_1, & \dots, & \ell_m}$ of rank $m$. Let
\[
M_n(x) = \M \left\{ \alpha \in J : z_{m+n} < x \right\}.
\]
In order to have $\alpha \in M_n(x)$ with $a_{m+n} = k$ and $c_{m+n} = \ell$, we must have $1+(x+\ell-1)^{-1}b^{-k} < z_{m+n-1} \le 1+\ell^{-1}b^{-k}$ (similar to in \eqref{eq:mnrecineq3}). It follows that
\[
M_n(x) = \sum_{k=0}^\infty \sum_{\ell = 1}^{b-1} M_{n-1} (1+\ell^{-1}b^{-k}) - M_{n-1}(1+(x+\ell-1)^{-1}b^{-k}),
\]
so that $(M_n')_{n=0}^{\infty} \in A^*$. Now by Lemma~\ref{lem:endpoints}, an arbitrary $\alpha \in J$ can be written as
\[
\alpha = \frac{p_m r_{m+1} + b^{a_m} p_{m-1}}{q_m r_{m+1} + b^{a_m} q_{m-1}},
\]
or since $r_{m+1} = \frac{1}{z_m - 1}$,
\[
\alpha = \frac{p_m + b^{a_m} p_{m-1} (z_m - 1)}{q_m + b^{a_m} q_{m-1} (z_m - 1)}.
\]
To have $1 < z_m < x$, we must have
\[
\alpha \in \left(\frac{p_m}{q_m}, \frac{p_m + b^{a_m} p_{m-1} (x-1)}{q_m + b^{a_m} q_{m-1} (x-1)} \right).
\]
Thus
\begin{equation}\label{eq:M0x}
M_0(x) = \left|\frac{p_m}{q_m} - \frac{p_m + b^{a_m} p_{m-1} (x-1)}{q_m + b^{a_m} q_{m-1} (x-1)} \right| = \frac{b^{\sum_{j=0}^m a_m} (x-1)}{q_m(q_m + b^{a_m} q_{m-1} (x-1))}.
\end{equation}
Now define
\[
\chi_n(x) = \frac{M_n(x)}{\M J},
\]
and note that $(\chi_n')_{n=0}^\infty \in A^*$, since $(M_n')_{n=0}^\infty \in A^*$ and
\begin{equation}\label{eq:MJ}
\M J = \left|\frac{p_m}{q_m} - \frac{p_m + b^{a_m} p_{m-1}}{q_m + b^{a_m} q_{m-1}} \right| = \frac{b^{\sum_{j=0}^m a_j}}{q_m(q_m + b^{a_m} q_{m-1})}
\end{equation}
is a constant. Now by \eqref{eq:M0x} and \eqref{eq:MJ}, we have
\begin{align*}
\chi_0(x) &= \frac{(q_m + b^{a_m} q_{m-1})(x-1)}{q_m + b^{a_m} q_{m-1} (x-1)}, \\
\chi_0'(x) &= \frac{q_m(q_m + b^{a_m} q_{m-1})}{(q_m + b^{a_m} q_{m-1} (x-1))^2}, \\
\chi_0''(x) &= -\frac{2 q_m b^{a_m} q_{m-1} (q_m + b^{a_m} q_{m-1})}{(q_m + b^{a_m} q_{m-1}(x-1))^3}.
\end{align*}
Thus for $1 \le x \le 2$, we have $\chi_0'(x) < \frac{2q_m^2}{q_m^2} = 2$, $\chi_0'(x) > \frac{q_m^2}{(2q_m)^2} = \frac14$, and $|\chi_0''(x)| < \frac{4 q_m^3}{q_m^3} = 4$, so $(\chi_n')_{n=0}^\infty \in A^{**}$. It then follows from Theorem~\ref{thm:mainthm} that there are constants $A, \lambda > 0$ such that 
\[
\left|\chi_n'(x) - \frac{a}{x(x+b-1)} \right| < Ae^{-\lambda \sqrt{n}},
\]
for all $n \ge 0$ and $x \in (1,2)$, or equivalently there exist functions $\theta_n: (1,2) \to (-1,1)$ such that
\[
\chi_n'(x) = \frac{a}{x(x+b-1)} + \theta_n(x) Ae^{-\lambda \sqrt{n}}
\]
for all $n \ge 0$ and $x \in (1,2)$. We then have, for $k \in \Z_{\ge 0}$ and $\ell \in \{1, \dots, b-1\}$, that
\begin{align*}
\chi_n(1+&\ell^{-1}b^{-k}) - \chi_n(1+(\ell+1)^{-1}b^{-k}) \\
&= \int_{1+(\ell+1)^{-1}b^{-k}}^{1+\ell^{-1}b^{-k}} \chi'_n(x) \d x \\
&= \int_{1+(\ell+1)^{-1}b^{-k}}^{1+\ell^{-1}b^{-k}} \frac{a}{x(x+b-1)} + \theta_n(x) A e^{-\lambda \sqrt{n}} \d x \\
&= \frac{a}{b-1} \log \frac{(1+\ell^{-1}b^{-k})(b+(\ell+1)^{-1}b^{-k})}{(b+\ell^{-1}b^{-k})(1+(\ell+1)^{-1}b^{-k})} + Ae^{-\lambda \sqrt{n}} \int_{1+(\ell+1)^{-1}b^{-k}}^{1+\ell^{-1}b^{-k}} \theta_n(x) \d x.
\end{align*}
Now
\[
\left| \int_{1+(\ell+1)^{-1}b^{-k}}^{1+\ell^{-1}b^{-k}} \theta_n(x) \d x \right| \le \int_{1+(\ell+1)^{-1}b^{-k}}^{1+\ell^{-1}b^{-k}} |\theta_n(x)| \d x < \int_{1+(\ell+1)^{-1}b^{-k}}^{1+\ell^{-1}b^{-k}} 1 \d x = \frac{1}{\ell(\ell+1)b^k},
\]
so there exist functions $\gamma_n: (1,2) \to (-1,1)$ such that
\[
\int_{1+(\ell+1)^{-1}b^{-k}}^{1+\ell^{-1}b^{-k}} \theta_n(x) \d x = \frac{\gamma_n(x)}{\ell(\ell+1)b^{k}}.
\]
Then since $\M E \mat{1, & \dots, & m, & m+n \\ k_1, & \dots, & k_m, & k_{m+n} \\ \ell_1, & \dots, & \ell_m, & \ell_{m+n}} = M_{n-1}(1+\ell^{-1}b^{-k}) - M_{n-1}(1+(\ell+1)^{-1}b^{-k})$,
\begin{multline*}
\M E\mat{1, & \dots, & m, & m+n \\ k_1, & \dots, & k_m, & k_{m+n} \\ \ell_1, & \dots, & \ell_m, & \ell_{m+n}} \\= \left(\frac{\log \frac{(1+\ell^{-1}b^{-k})(b+(\ell+1)^{-1}b^{-k})}{(b+\ell^{-1}b^{-k})(1+(\ell+1)^{-1}b^{-k})}}{\log \frac{2b}{b+1}} + \frac{\gamma_n(x) A e^{-\lambda \sqrt{n-1}}}{\ell (\ell+1) b^k}\right) \M E\mat{1, \dots, m \\ k_1, \dots, k_m \\ \ell_1, \dots, \ell_m}.
\end{multline*}
Now we can sum this relationship for $k_j$ from 0 to $\infty$ and $\ell_j$ from 1 to $b-1$ for certain indices $j \le m$. The indices we sum over will cancel from both sides, and we are left with an arbitrary sequence of subscripts $1 \le j_1 < j_2 < \cdots < j_t = m$. Then if we let $j = m+n$, we get
\[
\left| \frac{\M E \mat{ j_1, & \dots, & j_m, & j \\ k_1, & \dots, & k_m, & k \\ \ell_1, & \dots, & \ell_m, & \ell}}{\M E \mat{ j_1, & \dots, & j_m \\ k_1, & \dots, & k_m \\ \ell_1, & \dots, & \ell_m}} - \frac{\log \frac{(1+\ell^{-1}b^{-k})(b+(\ell+1)^{-1}b^{-k})}{(b + \ell^{-1}b^{-k})(1+(\ell+1)^{-1}b^{-k})}}{\log\frac{2b}{b+1}} \right| < \frac{Ae^{-\lambda \sqrt{j-j_m-1}}}{\ell (\ell+1)b^{k}},
\]
completing the proof.
\end{proof}

\begin{theorem}\label{thm:limavgf}
Suppose $f: \Z_{\ge 0} \times \{1, \dots, b-1\} \to \R$ is a positive function for which there exist constants $C, \delta > 0$ such that
\[
f(s,t) < C(tb^s)^{\frac12 - \delta}
\]
for all $s \in \Z_{\ge 0}$ and $t \in \{1,\dots,b-1\}$. Then for almost every $\alpha \in (1,2)$,
\[
\lim_{N \to \infty} \frac{1}{N} \sum_{n=1}^N f(a_n, c_n) = \sum_{k=0}^\infty \sum_{\ell = 1}^{b-1} f(k,\ell) \frac{\log \frac{(1+\ell^{-1}b^{-k})(b+(\ell+1)^{-1}b^{-k})}{(b + \ell^{-1}b^{-k})(1+(\ell+1)^{-1}b^{-k})}}{\log\frac{2b}{b+1}}.
\]
\end{theorem}
\begin{proof}
First define
\begin{align*}
u_k &= \int_1^2 f(a_k, c_k) \d \alpha, &
b_k &= \int_1^2 (f(a_k,c_k) - u_k)^2 \d \alpha, \\
g_{ik} &= \int_1^2 (f(a_i,c_i) - u_i)(f(a_k,c_k)-u_k) \d\alpha, &
S_n(\alpha) &= \sum_{k=1}^n (f(a_k,c_k) - u_k).
\end{align*}
Notice that the integral $u_k$ is finite for all $k$, since
\begin{align*}
u_k &= \int_1^2 f(a_k, c_k) \d \alpha = \sum_{s = 0}^\infty \sum_{t = 1}^{b-1} f(s,t) \M D_n(s,t) \\
&< \sum_{s=0}^{\infty} \sum_{t=1}^{b-1} C(tb^b)^{\frac12-\delta}(2t^{-2}b^{-s}) = 2C \sum_{s=0}^{\infty} \sum_{t = 1}^{b-1} t^{-1} (tb^s)^{-1-\delta} < \infty.
\end{align*}
Furthermore,
\begin{align*}
\int_1^2 f_n(a_k, c_k)^2 \d \alpha &= \sum_{s=0}^\infty \sum_{t=1}^{b-1} f(s,t)^2 \M D_n (s,t)  < \sum_{s=0}^{\infty} \sum_{t=1}^{b-1} C^2 (tb^s)^{1-2\delta} (2t^{-2}b^{-s}) \\
&= 2C^2 \sum_{s=0}^{\infty} \sum_{t=1}^{b-1} t^{-1}(tb^s)^{-2\delta} = C_1 < \infty,
\end{align*}
so
\begin{equation}\label{eq:bkbound}
b_k = \int_1^2 (f(a_k, c_k) - u_k)^2 \d \alpha = \int_1^2 f(a_k, c_k)^2 \d \alpha - 2u_k \int_1^2 f(a_k, c_k) \d \alpha + u_k^2 < C_1 - u_k \le C_1 < \infty,
\end{equation}
and by the Cauchy-Schwarz Inequality, 
\begin{equation}\label{eq:ukbound}
u_k = \int_1^2 f(a_k,c_k) \d\alpha < \sqrt{\int_1^2 f(a_k,c_k)^2\d\alpha} < \sqrt{C_1}.
\end{equation}
Furthermore, for $k > i$, we have
\begin{equation}\label{eq:gikequiv}
g_{ik} = \int_1^2 f(a_i, c_i) f(a_k, c_k) \d\alpha - u_i u_k = \sum_{s_1=0}^{\infty} \sum_{t_1=1}^{b-1} \sum_{s_2=0}^{\infty} \sum_{t_2 =1}^{b-1} f(s_1,t_1) f(s_2,t_2) \M E\mat{i & k \\ s_1 & s_2 \\ t_1 & t_2} - u_i u_k.
\end{equation}
Now by Theorem~\ref{thm:Eratiobound} and Corollary~\ref{cor:weakMDnbound},
\begin{align}
\left| \M E\mat{i & k \\ s_1 & s_2 \\ t_1 & t_2} - \frac{\log\frac{(1+t_2^{-1}b^{-s_2})(b+(t_2+1)^{-1}b^{-s_2})}{(b+t_2^{-1}b^{-s_2})(1+(t_2+1)^{-1}b^{-s_2})}}{\log\frac{2b}{b+1}} \M E\mat{i\\ s_1 \\ t_1} \right| &< \frac{Ae^{-\lambda\sqrt{k-i-1}}}{t_2(t_2+1)b^{s_2}} \M E\mat{i \\ s_1 \\ t_1} \notag \\
&< 4A e^{-\lambda \sqrt{k-i-1}} \M E\mat{i \\ s_1 \\ t_1} \M E\mat{k \\ s_2 \\ t_2}, \label{eq:MEikbound}
\end{align}
and by Theorem~\ref{thm:MDnbound} and Corollary~\ref{cor:weakMDnbound},
\begin{equation}\label{eq:MEkbound}
\left| \M E\mat{k \\ s_2 \\ t_2} - \frac{\log\frac{(1+t_2^{-1}b^{-s_2})(b+(t_2+1)^{-1}b^{-s_2})}{(b+t_2^{-1}b^{-s_2})(1+(t_2+1)^{-1}b^{-s_2})}}{\log\frac{2b}{b+1}} \right| < \frac{A e^{-\lambda \sqrt{k-1}}}{t_2(t_2+1)b^{s_2}}.
\end{equation}
Now by \eqref{eq:MEikbound} and \eqref{eq:MEkbound}, letting $v = \frac{\log\frac{(1+t_2^{-1}b^{-s_2})(b+(t_2+1)^{-1}b^{-s_2})}{(b+t_2^{-1}b^{-s_2})(1+(t_2+1)^{-1}b^{-s_2})}}{\log\frac{2b}{b+1}}$, we get
\begin{align}
\left| \M E \mat{i & k \\ s_1 & s_2 \\ t_1 & t_2}\right. &- \left. \M E \mat{i \\ s_1 \\ t_1} \M E \mat{k \\ s_2 \\ t_2} \right| \notag \\
&\le \left| \M E \mat{i & k \\ s_1 & s_2 \\ t_1 & t_2} - v \M E \mat{i \\ s_1 \\ t_1}  \right| + \left| v \M E \mat{i \\ s_1 \\ t_1} - \M E \mat{i \\ s_1 \\ t_1} \M E \mat{k \\ s_2 \\ t_2} \right| \notag \\
&< (4 A e^{-\lambda \sqrt{k-i-1}} + 4 A e^{-\lambda \sqrt{k-1}}) \M E \mat{i \\ s_1 \\ t_1} \M E \mat{k \\ s_2 \\ t_2} \notag \\
&\le 8 A e^{-\lambda \sqrt{k-i-1}} \M E \mat{i \\ s_1 \\ t_1} \M E \mat{k \\ s_2 \\ t_2}. \label{eq:MEik-MEiMEk}
\end{align}
Then by \eqref{eq:gikequiv} and \eqref{eq:MEik-MEiMEk}, we get
\begin{align}
\bigg| g_{ik} &- \sum_{s_1=0}^{\infty} \sum_{t_1=1}^{b-1} \sum_{s_2=0}^{\infty} \sum_{t_2 =1}^{b-1} f(s_1,t_1) f(s_2,t_2) \M E \mat{i \\ s_1 \\ t_1} \M E \mat{k \\ s_2 \\ t_2} + u_i u_k \bigg| \notag \\
&< 8 A e^{-\lambda \sqrt{k-i-1}} \sum_{s_1=0}^{\infty} \sum_{t_1=1}^{b-1} \sum_{s_2=0}^{\infty} \sum_{t_2 =1}^{b-1} f(s_1,t_1) f(s_2,t_2) \M E\mat{i \\ s_1 \\ t_1} \M E \mat{k \\ s_2 \\ t_2} \notag \\
&= 8 A e^{-\lambda \sqrt{k-i-1}} u_i u_k. \label{eq:gikbound1}
\end{align}
But since 
\[
\sum_{s_1=0}^{\infty} \sum_{t_1=1}^{b-1} \sum_{s_2=0}^{\infty} \sum_{t_2 =1}^{b-1} f(s_1,t_1) f(s_2,t_2) \M E \mat{i \\ s_1 \\ t_1} \M E \mat{k \\ s_2 \\ t_2} = u_i u_k,
\]
\eqref{eq:gikbound1} is just
\begin{equation}\label{eq:gikbound2}
|g_{ik}| < 8 A e^{-\lambda \sqrt{k-i-1}} u_i u_k < 8 A C_1 e^{-\lambda \sqrt{k-i-1}}.
\end{equation}
From \eqref{eq:bkbound} and \eqref{eq:gikbound2}, we have for $n > m > 0$,
\begin{align}
\int_1^2 (S_n&(\alpha) - S_m(\alpha))^2 \d \alpha \notag \\
&= \int_1^2 \left[\sum_{k=m+1}^n f(a_k, c_k) - u_k\right]^2 \d\alpha \notag \\
&= \sum_{k=m+1}^n (f(a_k,c_k)-u_k)^2\d\alpha + 2 \sum_{i=m+1}^{n-1} \sum_{k=i+1}^n \int_1^2 (f(a_i,c_i) - u_i)(f(a_k,c_k) - u_k) \d\alpha \notag \\
&= \sum_{k=m+1}^n b_k + 2 \sum_{i=m+1}^{n-1} \sum_{k=i+1}^n g_{ik} < C_1(n-m) + 16AC_1 \sum_{i=m+1}^{n-1} \sum_{k=i+1}^n e^{-\lambda \sqrt{k-i-1}} \notag \\
&< C_1(n-m) + 16 AC_1 \sum_{i=m+1}^n \sum_{j=0}^\infty e^{-\lambda \sqrt{j}} = C_1(n-m) + 16AC_1 (n-m) \sum_{j=0}^\infty e^{-\lambda \sqrt{j}} \notag \\
&= C_2 (n-m), \label{eq:IntSdiffbound}
\end{align}
where $C_2 = C_1 + 16 A C_1 \sum_{j=0}^\infty e^{-\lambda \sqrt{j}}$ is a constant. Now let $\e > 0$ and define
\[
e_n = \{\alpha \in (1,2) : |S_n(\alpha)| \ge \e n\}.
\]
Clearly
\[
\int_1^2 S_n(\alpha)^2 \d \alpha \ge \int_{e_n} S_n(\alpha)^2 \d \alpha \ge \e^2 n^2 \M e_n,
\]
so that if we let $m = 0$ in \eqref{eq:IntSdiffbound} we get
\[
\M e_{n^2} \le \frac{\int_1^2 S_{n^2}(\alpha)^2 \d \alpha}{\e^2 n^4} < \frac{C_2}{\e^2 n^3}.
\]
Thus the series $\sum_{n=1}^\infty \M e_{n^2}$ converges, so almost every $\alpha \in (1,2)$ belongs to $e_{n^2}$ for only finitely many $n \in \N$. Therefore for almost every $\alpha \in (1,2)$ and for sufficiently large $n$,
\[
\frac{S_{n^2}(\alpha)}{n^2} < \e.
\]
Now since $\e > 0$ was arbitrary, we can conclude that
\begin{equation}\label{eq:limSn2}
\lim_{n \to \infty} \frac{S_{n^2}(\alpha)}{n^2} = 0
\end{equation}
for almost every $\alpha \in (1,2)$.

Now let $N \in \N$ be arbitrary and choose $n$ such that $n^2 \le N < (n+1)^2$, so that
\[
\int_1^2 (S_N(\alpha) - S_{n^2}(\alpha))^2 \d\alpha < C_2 (N-n^2) < C_2((n+1)^2-n^2) = C_2(2n+1) \le 3C_2n.
\]
Let $\e > 0$ and define
\[
e_{n,N} = \{\alpha \in (1,2) : |S_N(\alpha) - S_{n^2}(\alpha)| \ge \e n^2 \}
\]
and
\[
E_n = \bigcup_{N=n^2}^{(n+1)^2-1} e_{n,N}.
\]
We then have for $n^2 \le N < (n+1)^2$ that
\[
\int_1^2 (S_N(\alpha) - S_{n^2}(\alpha))^2 \d \alpha \ge \int_{e_{n,N}} (S_N(\alpha) - S_{n^2}(\alpha))^2 > \e^2 n^4 \M e_{n,N},
\]
and
\[
\M e_{n,N} < \frac{\int_1^2 (S_N(\alpha)^2 - S_{n^2}(\alpha))^2}{\e^2 n^4} < \frac{3C_2}{\e^2 n^3},
\]
so
\[
\M E_n \le \sum_{N=n^2}^{(n+1)^2-1} \M e_{n,N} < ((n+1)^2-n^2) \frac{3C_2}{\e^2 n^3} \le \frac{9 C_2}{\e^2 n^2}.
\]
Thus the series $\sum_{n=1}^\infty \M E_n$ converges, so almost every $\alpha \in (1,2)$ belongs to $E_n$ for only finitely many $n \in \N$. In other words, for almost every $\alpha$, sufficiently large $N$, and $n = \lfloor \sqrt{N} \rfloor$, we have
\[
\frac{|S_N(\alpha) - S_{n^2}(\alpha)|}{n^2} < \e.
\]
Since $\e > 0$ was arbitrary, we can conclude
\[
\lim_{N \to \infty} \frac{S_N(\alpha)}{n^2} - \frac{S_{n^2}}{n^2} = 0
\]
for almost every $\alpha \in (1,2)$, where $n = \lfloor \sqrt{N} \rfloor$. By \eqref{eq:limSn2},
\[
\lim_{N \to \infty} \frac{S_N(\alpha)}{n^2} = 0,
\]
where $n = \lfloor \sqrt{N} \rfloor$. Now since $0 < \frac{S_N(\alpha)}{N} < \frac{S_N(\alpha)}{n^2}$, it follows that $\frac{S_N(\alpha)}{N} \to 0$ as $n \to \infty$. Equivalently, by the definition of $S_N$,
\begin{equation}\label{eq:limavgf-u}
\frac{1}{N} \sum_{k=1}^N f(a_k, c_k) - \frac{1}{N} \sum_{k=1}^N u_k \to 0
\end{equation}
as $N \to \infty$. Now by Theorem~\ref{thm:MDnbound},
\begin{align*}
\Bigg|u_n - &\sum_{k=0}^\infty \sum_{\ell=1}^{b-1} f(k,\ell) \frac{\log\frac{(1+\ell^{-1}b^{-k})(b+(\ell+1)^{-1}b^{-k})}{(b+\ell^{-1}b^{-k})(1+(\ell+1)^{-1}b^{-k})}}{\log\frac{2b}{b+1}} \Bigg| \\
&= \sum_{k=0}^{\infty} \sum_{\ell=1}^{b-1} f(k,\ell) \left|\M D_n\mat{k \\ \ell} - \frac{\log\frac{(1+\ell^{-1}b^{-k})(b+(\ell+1)^{-1}b^{-k})}{(b+\ell^{-1}b^{-k})(1+(\ell+1)^{-1}b^{-k})}}{\log\frac{2b}{b+1}} \right| \\
&< Ae^{-\lambda \sqrt{n-1}} \sum_{k=0}^\infty \sum_{\ell = 1}^{b-1} \frac{f(k,\ell)}{\ell(\ell+1)b^k} < A e^{-\lambda \sqrt{n-1}} \sum_{k=0}^\infty \sum_{\ell=1}^{b-1} \frac{C(\ell b^k)^{\frac12-\delta}}{\ell(\ell+1)b^k} \\
&= ACe^{-\lambda\sqrt{n-1}} \sum_{k=0}^\infty \sum_{\ell=1}^{b-1} \frac{1}{(\ell+1)(\ell b^k)^{\frac12 + \delta}} < A_1 e^{-\lambda \sqrt{n}}
\end{align*}
for some constant $A_1$. Thus for almost every $\alpha \in (1,2)$,
\[
\lim_{n \to \infty} u_n = \sum_{k=0}^\infty \sum_{\ell = 1}^{b-1} f(k,\ell) \frac{\log\frac{(1+\ell^{-1}b^{-k})(b+(\ell+1)^{-1}b^{-k})}{(b+\ell^{-1}b^{-k})(1+(\ell+1)^{-1}b^{-k})}}{\log\frac{2b}{b+1}},
\]
so indeed,
\[
\lim_{N \to \infty} \frac{1}{N} \sum_{n=1}^N u_n = \sum_{k=0}^\infty \sum_{\ell = 1}^{b-1} f(k,\ell) \frac{\log\frac{(1+\ell^{-1}b^{-k})(b+(\ell+1)^{-1}b^{-k})}{(b+\ell^{-1}b^{-k})(1+(\ell+1)^{-1}b^{-k})}}{\log\frac{2b}{b+1}},
\]
at which point \eqref{eq:limavgf-u} gives
\[
\lim_{N \to \infty} \frac{1}{N} \sum_{n=1}^{N} f(a_n,c_n) = \sum_{k=0}^\infty \sum_{\ell = 1}^{b-1} f(k,\ell) \frac{\log\frac{(1+\ell^{-1}b^{-k})(b+(\ell+1)^{-1}b^{-k})}{(b+\ell^{-1}b^{-k})(1+(\ell+1)^{-1}b^{-k})}}{\log\frac{2b}{b+1}},
\]
for almost every $\alpha \in (1,2)$.
\end{proof}

We can now prove the desired theorems.

\restatekhinchinedistthm

\begin{proof}
Fix $k \in \Z_{\ge 0}$ and $\ell \in \{1,2,\dots,b-1\}$. Let
\[
f(s,t) = \begin{cases} 1 & s = k, t = \ell \\ 0 & \text{otherwise} \end{cases}.
\]
Clearly $f(s,t) < 2 < 3(t b^s)^{1/4}$ so $f$ satisfies the conditions of theorem~\ref{thm:limavgf}. Now
\[
\lim_{N \to \infty} \frac{1}{N} \sum_{n=1}^N f(a_n,c_n) = \lim_{N \to \infty} \frac{|\{n \in \N : a_n = k, c_n = \ell\}|}{N},
\]
so Theorem~\ref{thm:limavgf} immediately gives, for almost every $\alpha \in (1,2)$ that
\[
P_\alpha(k,\ell) = \lim_{N \to \infty} \frac{|\{n \in \N : a_n = k, c_n = \ell\}|}{N} = \frac{\log\frac{(1+\ell^{-1}b^{-k})(b+(\ell+1)^{-1}b^{-k})}{(b+\ell^{-1}b^{-k})(1+(\ell+1)^{-1}b^{-k})}}{\log\frac{2b}{b+1}},
\]
proving the theorem.
\end{proof}

\restatekhinchinethm

\begin{proof}
Define $f(s,t) = \log_b(tb^s) = s+\log_b t$. Notice that we can choose $C > 0$ such that $\log_b(x) < C x^{1/3}$ for all $x \ge 1$. Then if we take $\delta = \frac16$, we get
\[
f(s,t) = \log_b(tb^s) < C(tb^s)^{1/3} = C(tb^s)^{\frac12-\delta},
\]
so $f$ satisfies the conditions of Theorem~\ref{thm:limavgf}. We then get that for almost every $\alpha \in (1,2)$.
\begin{equation}\label{eq:limlogcba}
\lim_{N \to \infty} \frac{1}{N} \log_b(c_n b^{a_n}) = \sum_{k=0}^\infty \sum_{\ell=1}^{b-1} \log_b(\ell b^k) \frac{\log\frac{(1+\ell^{-1}b^{-k})(b+(\ell+1)^{-1}b^{-k})}{(b+\ell^{-1}b^{-k})(1+(\ell+1)^{-1}b^{-k})}}{\log\frac{2b}{b+1}}.
\end{equation}
Now let $u(k,\ell) = \log(1+\ell^{-1}b^{-k})$ and $v(k) = u(k,\ell) - u(k,\ell+1)$. Notice that $u(k,b) = u(k+1,1)$. Then
\begin{align*}
\sum_{k=0}^\infty &\sum_{\ell=1}^{b-1} \log_b(\ell b^k) \frac{\log\frac{(1+\ell^{-1}b^{-k})(b+(\ell+1)^{-1}b^{-k})}{(b+\ell^{-1}b^{-k})(1+(\ell+1)^{-1}b^{-k})}}{\log\frac{2b}{b+1}} \\
&= \frac{1}{\log \frac{2b}{b+1}} \sum_{k=0}^\infty \sum_{\ell = 1}^{b-1} (k+\log_b \ell) [u(k,\ell) + u(k+1,\ell+1) - u(k+1,\ell) - u(k,\ell+1)] \\
&= \frac{1}{\log\frac{2b}{b+1}} \sum_{\ell=1}^{b-1} \sum_{k=0}^\infty (k+\log_b\ell) [v(k) - v(k+1)] = \frac{1}{\log\frac{2b}{b+1}} \left(A+\frac{B}{\log b}\right),
\end{align*}
where
\begin{align*}
A &= \sum_{\ell = 1}^{b-1} \sum_{k=0}^{\infty} k[v(k)-v(k+1)] = \sum_{\ell=1}^{b-1} \sum_{k=1}^{\infty} v(k) = \sum_{k=1}^{\infty} \sum_{\ell=1}^{b-1} u(k,\ell) - u(k,\ell+1) \\
&= \sum_{k=1}^\infty u(k,1) - u(k,b) = \sum_{k=1}^\infty u(k,1) - u(k+1,1) = u(1,1) - \lim_{k \to \infty} u(k,1) = \log\left(1+\frac1b\right),
\end{align*}
and
\begin{align*}
B &= \sum_{\ell = 1}^{b-1} \log \ell \sum_{k=0}^\infty v(k)-v(k+1) = \sum_{\ell=1}^{b-1} \log \ell (v(0) - \lim_{k \to \infty} v(k)) \\
&= \sum_{\ell=1}^{b-1} \log \ell \left[ \log \frac{1+\ell^{-1}}{1+(\ell+1)^{-1}} - \lim_{k \to \infty} \log \frac{1+\ell^{-1}b^{-k}}{1+(\ell+1)^{-1}b^{-k}} \right] \\
&= \sum_{\ell=1}^{b-1} \log \ell \left[ \log \left(1+\frac{1}{\ell}\right) - \log\left(1+\frac{1}{\ell+1}\right)\right] \\
&=\sum_{\ell=1}^{b-1} \log \ell \log \left(1+\frac1\ell\right) - \sum_{\ell=2}^{b} \log(\ell-1) \log\left(1+\frac1\ell\right) \\
&= \log 1 \log 2 - \sum_{\ell=2}^{b-1} \left[\log(\ell-1)-\log\ell\right] \log\left(1+\frac1\ell\right) - \log(b-1) \log\left(1+\frac1b\right) \\
&= -\sum_{\ell=2}^{b-1} \log\left(1-\frac1\ell\right) \log\left(1+\frac1\ell\right) - \log(b-1)\log\left(1+\frac1b\right).
\end{align*}
Thus we have
\begin{align*}
\sum_{k=0}^\infty &\sum_{\ell=1}^{b-1} \log_b(\ell b^k) \frac{\log\frac{(1+\ell^{-1}b^{-k})(b+(\ell+1)^{-1}b^{-k})}{(b+\ell^{-1}b^{-k})(1+(\ell+1)^{-1}b^{-k})}}{\log\frac{2b}{b+1}} \\
&= \frac{1}{\log b \log\frac{2b}{b+1}} \left[\log b \log\left(1+\frac1b\right) - \log(b-1)\log\left(1+\frac1b\right) - \sum_{\ell=2}^{b-1} \log\left(1-\frac1\ell\right) \log\left(1+\frac1\ell\right)\right] \\
&= -\frac{1}{\log b \log \frac{2b}{b+1}} \sum_{\ell=2}^b \log\left(1-\frac1\ell\right) \log\left(1+\frac1\ell\right) = \frac{1}{\log b \log \frac{b+1}{2b}} \sum_{\ell=2}^b \log\left(1-\frac1\ell\right) \log\left(1+\frac1\ell\right) = \mathcal{A}.
\end{align*}
Thus \eqref{eq:limlogcba} becomes
\[
\lim_{N \to \infty} \frac1N \sum_{n=1}^N \log_b(c_n b^{a_n}) = \mathcal{A},
\]
from which it follows that for almost all $\alpha \in (1,2)$,
\[
\lim_{N \to \infty} \left(\prod_{n=1}^N c_n b^{a_n}\right)^{\frac1N} = b^{\lim_{N \to \infty} \frac1N \sum_{n=1}^N \log_b(c_n b^{a_n})} = b^{\mathcal{A}},
\]
as required.
\end{proof}

\end{document}